\documentclass[a4paper,10pt]{amsart}
\usepackage[english]{babel}
\usepackage[utf8]{inputenc}
\usepackage[T1]{fontenc}
\usepackage[style=numeric,
	useprefix,%
	giveninits=true,%
	hyperref,%
	doi=false,%
	url=false,%
	isbn=false,%
	backend=bibtex,%
	maxbibnames=99%
	]{biblatex}
\bibliography{./BIB-LimitsOfSubFinslerOnGroups}

\usepackage{amssymb}
\usepackage{mathrsfs}
\usepackage{hyperref}
\usepackage{mathtools}
\usepackage[usenames,dvipsnames]{xcolor}
\hypersetup{colorlinks,%
citecolor=Black,%
filecolor=Black,%
linkcolor=Black,%
urlcolor=Black}
\usepackage{enumitem}	
\allowdisplaybreaks 
\newcommand{\scr}[1]{\mathscr{#1}}

\newcommand{\bb}[1]{\mathbb{#1}}
\newcommand{\cal}[1]{\mathcal{#1}}
\newcommand{\N}{\mathbb{N}}	
\newcommand{\Z}{\mathbb{Z}}	
\newcommand{\R}{\mathbb{R}}	
\DeclareMathOperator{\Span}{\mathrm{span}}	
\newcommand{\dd}{\,\mathrm{d}}	
\newcommand{\de}{\partial}		
\DeclareMathOperator{\diam}{\mathrm{diam}}
\DeclareMathOperator*{\esssup}{{esssup}}
\newcommand{\into}{\hookrightarrow}		
\newcommand{\THEN}{\Rightarrow}	

\newcommand{\E}{\bb E}
\newcommand{\f}{\mathtt{f}}
\newcommand{\End}{\mathrm{End}}
\newcommand{\LipS}{\mathrm{Lip}\Gamma}
\newcommand{\DomEnd}{\scr{D}_ \End}
\newcommand{\Enel}{\mathtt J}
\newcommand{\Length}{\ell}
\newcommand{\vertiii}[1]{{\left\vert\kern-0.25ex\left\vert\kern-0.25ex\left\vert #1 
    \right\vert\kern-0.25ex\right\vert\kern-0.25ex\right\vert}}

\theoremstyle{plain}
\newtheorem{proposition}{Proposition}[section]
\newtheorem{theorem}[proposition]{Theorem}
\newtheorem{lemma}[proposition]{Lemma}
\newtheorem{corollary}[proposition]{Corollary}

\theoremstyle{definition}
\newtheorem{definition}[proposition]{Definition}
\newtheorem{remark}[proposition]{Remark}

\theoremstyle{remark}

\title{Lipschitz Carnot-Carath\'eodory structures and their limits}
\author[Antonelli]{Gioacchino Antonelli}
\author[Le Donne]{Enrico Le Donne}
\author[Nicolussi Golo]{Sebastiano Nicolussi Golo}

\address{Gioacchino Antonelli  \\ Scuola Normale Superiore, Piazza dei Cavalieri, 7, Pisa, 56126}
\email{\tt gioacchino.antonelli@sns.it}
\address{Enrico Le Donne \\Department of Mathematics, University of Fribourg, Chemin du Musée~23, 1700 Fribourg, Switzerland \& 
University of Jyv\"askyl\"a, Department of Mathematics and Statistics, P.O. Box (MaD), FI-40014, Finland}
\email{\tt enrico.ledonne@unifr.ch}
\address{Sebastiano Nicolussi-Golo \\ University of Jyv\"askyl\"a, Department of Mathematics and Statistics, P.O. Box (MaD), FI-40014, Finland}
\email{\tt senicolu@jyu.fi}

\subjclass[2010]{%
28A75; 
53C60; 
53C17; 
53C23; 
49J52; 
26A16. 
}
\keywords{sub-Finsler geometry; %
sub-Riemannian geometry; %
Lipschitz vector fields; 
Mitchell's Theorem. %
}
\date{\today.} 

\begin{document}
\maketitle


\begin{abstract}
	In this paper we discuss the convergence of distances associated to converging structures of Lipschitz vector fields and continuously varying norms on a smooth manifold. We prove that, under a mild controllability assumption on the limit vector-fields structure, the distances associated to equi-Lipschitz vector-fields structures that converge uniformly on compact subsets, and to norms that converge uniformly on compact subsets, converge locally uniformly to the limit Carnot-Carathéodory distance.
	
	In the case in which the limit distance is boundedly compact, we show that the convergence of the distances is uniform on compact sets. We show an example in which the limit distance is not boundedly compact and the convergence is not uniform on compact sets.
	
	We discuss several examples in which our convergence result can be applied. Among them, we prove a subFinsler Mitchell's Theorem with continuously varying norms, and a general convergence result for Carnot-Carathéodory distances associated to subspaces and norms on the Lie algebra of a connected Lie group.
\end{abstract}
\setcounter{tocdepth}{2}
\phantomsection
\addcontentsline{toc}{section}{Contents}
\tableofcontents

\section{Introduction}
 	
	This paper deals with the following general problem. Let $M$ be a smooth manifold endowed with a family of vector fields and a continuously varying norm on the tangent spaces. Let us consider the length distance associated to the trajectories that infinitesimally follow such a family. What are the weakest notion of convergence and the most general assumptions on the family of vector fields and the norm that ensure the uniform convergence of the associated length distances? 
	
	Such a question is natural while studying metric geometry. For example, fixed a continuously varying norm on the tangents, a sub-Finsler distance on a manifold is defined as the length distance associated to a bracket-generating family of smooth vector fields. It is a classical fact that a sub-Finsler distance can be approximated from below by increasing sequences of Finsler distances, see, e.g., \cite{LD13, LDLP19}. Nevertheless, understanding whether a specific approximation of the vector fields (and, possibly, of the norm) gives raise to the convergence of the associated length distances seems not to have been deeply studied in the literature. A by-product of our study  goes in this direction, since it gives an effective tool to approximate subFinsler distances with Finsler distances in a controlled way. We refer the reader to the examples discussed in Theorem~\ref{thm6092a1a1} and Theorem~\ref{thm:PerEmilio}, which are consequences of the main Theorem~\ref{thm60929a0e} below.
	
	Another situation in which the convergence of the distance associated to converging subFinsler structures emerge is while studying asymptotic or tangent cones of subFinsler structures, see the celebrated works of Mitchell and Bellaiche \cite{Mitchell, Bellaiche} (and the account in \cite{Jean}), and the work of Pansu \cite{PansuCroissance}. We refer the reader to the examples discussed in Proposition~\ref{prop:AsymptoticRiemannianHeisenbergGroup}, and Theorem~\ref{thm:Mitchell} below.
	
	\medskip
	We now introduce the language that we will adopt in the paper.
	
	Let us fix from now on a finite-dimensional Banach space $\mathbb E$, and a smooth manifold $M$. A \textit{Lipschitz-vector-field structure $\f:M\to \mathbb E^*\otimes TM$} on 
	 $M$ modelled by $\mathbb E$ is a Lipschitz choice, for every point $p\in M$, of a linear map between $\E$  and $T_pM$, see Definition~\ref{def:LipschitzVectorFieldsStructure}. We say that a sequence of Lipschitz-vector-field structure \textit{$\{\f_n\}_{n\in\N}$ converges to a Lipschitz-vector-field structure $\f_\infty$} if, on every compact subset of $M$, $\{\f_n\}_{n\in N}$ is an equi-Lipschitz family that converges to $\f_\infty$ uniformly, see Definition~\ref{defSeqTop} for details. Attached to a Lipschitz-vector-field structure $\f$ we have the notion of an End-point map, as follows.
	\begin{definition}[End-point map]\label{def:EndPoin}
	Let $\E$ be a finite-dimensional Banach space, and $M$ be a smooth manifold. Let $\f$ be a Lipschitz-vector-field structure on $M$ modelled by $\mathbb E$,  and let $u\in L^\infty([0,1];\E)$.
	For $o\in M$, we define the \emph{End-point map}
	\[
	\End^\f_o[u] := \End(o,\f,u) := \gamma(1) ,
	\]
	where $\gamma:[0,1]\to M$ is the solution of the Cauchy problem
	\begin{equation}\label{eq609b82be}
	\begin{cases}
	\gamma'(t) &= \f|_{\gamma(t)}[u(t)], \\
	\gamma(0) &= o ,
	\end{cases}
	\end{equation}
	whenever it exists. 
	Notice that $\gamma(t) = \End^\f_o(tu)$ for every $t\in [0,1]$ whenever the End-point map is well-defined.
	\end{definition}
	
	 Let $\f$ be a Lipschitz-vector-field structure on $M$ modelled by $\mathbb E$. We say that $N:M\times\mathbb E\to\mathbb R$ is a \textit{continuously varying norm} on $M\times \mathbb E$ if $N$ is continuous, and $N(p,\cdot)$ is a norm on $\E$ for every $p\in M$.  We aim now at describing how the couple $(\f,N)$ gives raise to a distance on $M$. 
	 
	 First, we define the energy and the length functionals working on the space $L^\infty([0,1];\mathbb E)$, which we sometimes call \textit{space of controls}. We stress that in our approach the energy and the length are not defined at the level of the curves, but instead on the controls that define them. In particular, if $o\in M$, $u\in L^\infty([0,1];\mathbb E)$, $(o,\f,u)$ is in the domain of the End-point map, see Definition~\ref{def:EndPoin}, and $N$ is a continuously varying norm on $M\times\E$, we define the \textit{length} $\ell$ (resp., the \textit{energy} $\Enel$) \textit{associated to} $(o,\f,u,N)$ to be $\int_0^1 N(\gamma(t),u(t))\dd t$ (resp., $\esssup_{t\in [0,1]}N(\gamma(t),u(t))$), where $\gamma(t):=\End^\f_o(tu)$. Given the notion of energy, we can define the distance as follows.
	
	\begin{definition}[CC distance]\label{def:CCdistanceIntro}
		Let $\E$ be a finite-dimensional Banach space, and $M$ be a smooth manifold. Let $\f$ be a Lipschitz-vector-field structure on $M$ modelled by $\mathbb E$,  and let $N:M\times \E\to [0,+\infty)$ be a continuously varying norm. We define
		\emph{the Carnot-Carathéodory distance}, or \emph{CC distance}, between $p$ and $q$ as follows
		\begin{equation}\label{eqn:DistanceAsInfEnergyIntro}
			d_{(\f,N)}(p,q) := \inf\{ \Enel(p,\f,u,N) : \End^\f_p(u) = q \},
		\end{equation}
	where $\Enel(p,\f,u,N)$ is the above defined energy associated to $(p,\f,u,N)$.
	\end{definition}

	In Lemma~\ref{lem604ba336} we shall show that a constant-speed reparametrization of a curve always exists, hence we can also equivalently take the infimum of the lengths in~\eqref{eqn:DistanceAsInfEnergyIntro}. We remark that, without any further hypotheses on the couple $(\f,N)$, it might happen that $d_{(\f,N)}(p,q)=+\infty$ for some points $p,q\in M$.
	
	In Corollary~\ref{Cor:ApproccioAstrattoUgualeApproccioConcreto}, we shall show that, if $\f$ is a Lipschitz-vector-field structure on $M$ modelled by $\mathbb E$, and $N:M\times \E\to [0,+\infty)$ is a continuously varying norm, then, for every $p,q\in M$, we have also
		\[
		d_{(\f,N)}(p,q) = \inf\left\{
		\int_0^1 |\gamma'(t)|_{(\f,N)} \dd t :  
		\begin{array}{l}
			\gamma:[0,1]\to M\text{ absolutely continuous,} \\
			\text{with }\gamma(0)=p,\ \gamma(1)=q 
		\end{array}
		\right\},
		\]
		where
		\begin{equation}\label{eqn:NormFnIntro}
			|v|_{(\f,N)} := \inf\{N(p,u):u\in\E,\ \f(p,u)=v\} , \quad \text{for $p\in M$ and $v\in T_pM$}.
		\end{equation}
		This means that the definition of the distance $d_{(\f,N)}$ working with the controls, i.e., the one given in~\eqref{eqn:DistanceAsInfEnergyIntro}, is equivalent to the definition of the distance as the infimum of the lengths evaluated with respect to the natural sub-Finsler metric associated to $(\f,N)$, i.e.,~\eqref{eqn:NormFnIntro}.
		\medskip 
		
		We now aim at understanding which kind of convergence is expected from the sequence of distances $\{d_{(\f_n,N_n)}\}_{n\in\mathbb N}$ when we have that the sequence $\{(\f_n,N_n)\}_{n\in\mathbb N}$ converges. The key hypothesis in order to have the local uniform convergence of the distances is a kind of essential non-holonomicity of the limit vector-fields structure $\f_\infty$, which we next introduce.
	
		First of all we introduce the notion of essentially open map. We say that a continuous map $f:M\to N$ between two topological manifolds of the same dimension $k$ is \textit{essentially open at $p\in M$ at scale $U$} if $U$ is a neighborhood of $p$ homeomorphic to the $k$-dimensional Euclidean ball, with $\partial U$ homeomorphic to the sphere $\mathbb S^{k-1}$, and there exists $V$ a neighborhood of $f(p)$ homeomorphic to the $k$-dimensional Euclidean ball, such that $f(\partial U)\subset V\setminus f(p)$ and the map $f:\partial U\to V\setminus f(p)$ induces a nonconstant map between the $(k-1)$-homology groups, see Definition~\ref{def:EssentiallyOpen}.
		
		Notice that the definition of essential openness at $p\in M$ does depend on the scale $U\ni p$. Notice also that if the map induced between the local homology of $p$ and $f(p)$ is not trivial, i.e., $f_\star:H^k(M,M\setminus \{p\})\to H^k(N,N\setminus f(p))$ is nonzero, then $f$ is essentially open at $p\in M$ at every sufficiently small neighborhood $U$ of $p$. Essential opennes at a point $p$ at scale $U$ does not imply openness at $p$, but only that $f(p)$ is in the interior of $f(U)$, cf. Remark~\ref{rem:RemOnEssOpen}. Still a local homeomorphism at $p$ is indeed essentially open at $p$ at some scale $U\ni p$. We further stress that 
		 
		 Then we are ready to give the following definition. We recall that we denote by $\Phi_X^t$ the flow at time $t$ of a vector field $X$ on the smooth manifold $M$, whenever it exists.
		
		\begin{definition}[Essentially non-holonomic]\label{def6067306f}
			Let $M$ be a smooth manifold of dimension $m$ and let $\scr F$ be a family of Lipschitz vector fields on $M$.
			We say that $\scr F$ is \emph{essentially non-holonomic}
			if for every $T>0$ and every $o\in M$, there are $X_1,\dots,X_m\in\scr F$ and $\hat t\in \mathbb R^m$ with $|\hat t|_1<T$ 
			such that there exists a neighborhood $\Omega_{\hat t}\subseteq B(0,T)\subseteq \R^m$ of $\hat t$  for which the map
			\[
			\phi(t_1,\dots,t_m) := \Phi^{t_m}_{X_m}\circ\dots\circ\Phi^{t_1}_{X_1}(o), 
			\]
			is defined on $\Omega_{\hat t}$ and, when restricted to $\Omega_{\hat t}$, is an essentially open map at $\hat t$ in a neighborhood of it.
			
			A Lipschitz-vector-field structure $\f$ on $M$ modelled by a finite-dimensional Banach space $\mathbb E$ 
			is said to be {\em essentially non-holonomic} if
			there is a basis $(e_1,\dots,e_r)$ of $\E$ such that
			$\scr F=\{\f(\cdot,e_1),\dots,\f(\cdot,e_r)\}$ 
			is essentially non-holonomic.
			%
		\end{definition}
	Let us explain the definition above. Equivalently, a set $\scr F$ of Lipschitz vector fields on a smooth manifold $M^m$ is essentially non-holonomic at a point $p\in M$ whenever there exists a sequence of points $p_n\in M$ that converges to $p$ such that $p_n$ is connected to $p$ with the concatenation, starting at $p$, of line flows of $m$ vector fields in $\scr F$ for times $(t_1,\dots,t_m)$, and moreover such concatenation is essentially open around $(t_1,\dots,t_m)$. We stress that the latter notion is weaker than the bracket-generating condition in the case the vector fields are smooth, cf. Proposition~\ref{prop:EveryBracket}.
	\medskip
	
	We are now ready to give the main theorem of the paper. The following theorem is proved at the end of Section~\ref{sec:LimitCC}.
	\begin{theorem}\label{thm60929a0e}
		Let $M$ be a smooth manifold, and let $\E$ be a finite-dimensional Banach space. Let $\hat\f$ be an essentially non-holonomic Lipschitz-vector-field structure modelled by $\mathbb E$, and let $\hat N:M\times\E\to[0,+\infty)$ be a continuously varying norm. Then the following hold.
		\begin{enumerate}[label=(\roman*)]
			\item\label{thm60929a0e_1}
			if $M$ is connected, then $d_{(\hat\f,\hat N)}(p,q)<\infty$ for every $p,q\in M$;
			\item\label{thm60929a0e_2}
			$d_{(\hat\f,\hat N)}$ induces the manifold topology on $M$;
			\item\label{thm60929a0e_3}
			Let $\{\f_n\}_{n\in\N}$ be a sequence of Lipschitz-vector-field structures on $M$ modelled by $\mathbb E$, and let $\{N_n\}_{n\in\mathbb N}$ be a sequence of continuously varying norms on $M\times \E$. Let us assume that $\f_n\to\hat\f$ in the sense of Lipschitz-vector-field structures (see Definition~\ref{defSeqTop}), and $N_n\to \hat N$ uniformly on compact subsets of $M\times\E$. 
			
			Then
			$d_{(\f_n, N_n)} \to d_{(\hat\f,\hat N)}$ locally uniformly on $M$, i.e., every $o\in M$ has a neighborhood $U$ such that $d_{(\f_n, N_n)}\to d_{(\hat\f,\hat N)}$ uniformly on $U\times U$ as $n\to +\infty$.
			\item\label{thm60929a0e_4}
			If in the hypotheses of item (iii) we additionally have that $d_{(\hat \f,\hat N)}$ is a boundedly compact (or equivalently complete) distance, we conclude that
			\[
			\lim_{n\to +\infty} d_{(\f_n, N_n)} = d_{(\hat\f,\hat N)},
			\] 
			uniformly on compact subsets of $M\times M$. Moreover, for every $x\in M$, we have $(M,d_{(\f_n,N_n)},x)\to (M, d_{(\hat \f,\hat N)}, x)$ in the pointed Gromov--Hausdorff topology as $n\to +\infty$.
		\end{enumerate}
	\end{theorem}
	Let us comment on the latter statements. The first item is a Chow-Rashevskii type result for Lipschitz vector fields that satisfy the essentially non-holonomic condition. It implies in particular the classical Chow-Rashevskii theorem, \cite{MR3971262}. The item (i) is easy since the fact that two sufficiently near points can be connected by a finite-length curve is guaranteed by Definition~\ref{def6067306f}. Item (ii) requires the essentially non-holonomic condition. Indeed, for an arbitrary couple $(\f,N)$, one only has that the topology induced by $(\f,N)$ is finer than the topology of the manifold, see Lemma~\ref{lem60e6e893} and the discussion after it. Item (iii) is the main convergence result, and it only holds locally around every point. When one adds the hypothesis that the limit distance is boundedly compact, the uniform convergence on compact sets can be obtained, as stated in item (iv). Without the hypothesis on the boundedly compactness of the limit distance, the convergence result might be false, see the example in Remark~\ref{rem:EXAMPLE}.
	
	\medskip
	Let us now describe the main steps of the proof of the Theorem~\ref{thm60929a0e}. 
	The first nontrivial achievement to obtain the proof is that, when $\hat \f$ is essentially non-holonomic, the topology of $M$ is finer than the topology induced by $d_{(\hat \f,\hat N)}$; and thus equal taking into account that the other inclusion, i.e., the one in Lemma~\ref{lem60e6e893}, holds in general. The nontrivial inclusion follows from the fact that when $\hat \f$ is essentially non-holonomic, the End-point map associated to $\hat \f$ is open, see Theorem~\ref{thm60a3bb36}, and Lemma~\ref{lem604be4c3}. Such an opennes property is a direct by-product of the essentially non-holonomic condition, see Remark~\ref{rem:RemOnEssOpen}. 
	
	The latter described opennes property is stable along a sequence $\{\f_n\}_{n\in\mathbb N}$ of Lipschitz-vector-field structures that converge to $\hat\f$, and this is the key point to obtain item (iii). Such a stability is the content of Theorem~\ref{thm60a3bb36}. Its proof builds on the top of the joint continuity of the End-point map in the three variables $o\in M$, $\f$ a Lipschitz-vector-field structure on $M$ modelled by $\mathbb E$, and $u\in L^{\infty}([0,1],\E)$ (with the weak* topology), see Theorem~\ref{thm6092f873} and Proposition~\ref{prop600c1758}, and on the topological Lemma~\ref{lem6082ea3b} proved with the aid of degree theory, which tells us that continuous functions that are near to an essentially open function are uniformly surjective.
	
	To conclude the proof of item (iii) one exploits the latter stability property to prove that, on compact sets, the topology of $M$ is uniformly finer than the topologies induced by the distances $d_{(\f_n,N_n)}$, see Lemma~\ref{lem604be4c3}. This directly implies that the functions $\{d_{(\f_n,N_n)}\}_{n\in\N}$ are equicontinuous on compact sets, see Proposition~\ref{prop604be4c9}. To end the proof of item (iii) one then finally uses the previous equicontinuity together with the fact that, $d_{(\hat \f,\hat N)}$ is locally obtained as a relaxation of $d_{(\f_n,N_n)}$, see Proposition~\ref{prop604f912b}.
	
	Item (iv) is then proved by exploiting Item (iii) and the general metric result in Lemma~\ref{Lemma2} according to which one can pass from the local uniform convergence to the uniform converge on compact sets in a very general setting under mild assumptions, i.e., the metric spaces are length spaces and the limit distance is boundedly compact. 
		
	In Appendix~\ref{sec:Shorter} we offer a shorter proof of Theorem~\ref{thm60929a0e}(iv) assuming that the vector-fields  structures are smooth. In this case we can argue directly by using Gronwall's Lemma~\ref{lem1053}, and the quantitative open-mapping type result in Lemma~\ref{lem:HOMEO}, which does not need degree theory since we have enough regularity of the flow maps. 
	\medskip
	
	Let us now discuss some corollaries of the general convergence result in Theorem~\ref{thm60929a0e}. In Section~\ref{sec:Examples} we list some examples in which a direct application of Theorem~\ref{thm60929a0e} gives nontrivial consequences. We record here a couple of them. The first is Mitchell's Theorem for subFinsler manifolds with continuously varying norms.
	
	Let us give some preliminary notation and definitions. Let $M^m$ be a smooth manifold of dimension $m\in\mathbb N$, and let $\mathcal{X}:=\{X_1,\dots,X_k\}$ be a bracket-generating family of smooth vector fields on $M$. Let $\E:=\R^k$, and let $N:M\times \E\to [0,+\infty)$ be a continuously varying norm. For every $p\in M$ and $v\in T_pM$, let $|\cdot|_{\mathcal X,N}$ be the sub-Finsler metric defined as follows 
	\begin{equation}\label{eqn:VxN}
		|v|_{\mathcal X,N}:=\inf\left\{N(p,u):u\in\mathbb E,\sum_{i=1}^ku_iX_i(p)=v\right\}.
	\end{equation}
	The subFinsler distance between $p,q\in M$ is
	\begin{equation}\label{eqn:DxN}
		d_{(\mathcal X,N)}(p,q):=\inf\left\{\int_0^1|\gamma'(t)|_{\mathcal X,N}\dd t:\gamma:[0,1]\to M\,\text{is AC},\,\gamma(0)=p,\gamma(1)=q\right\}.
	\end{equation}
	We recall that the definition of a regular point can be found, e.g., in \cite[Definition 2.4]{Jean}. The following theorem is proved at the end of Section~\ref{sec6092e03a}.
\begin{theorem}[Mitchell's Theorem for subFinsler manifolds with continuously varying norm]\label{thm:Mitchell}
	Let $M^m,\mathcal{X},N$ be as above. Let us fix $o\in M$. There exists a bracket-generating family of polynomial vector fields $\hat{\mathcal X}:=\{\hat X_1,\dots,\hat X_k\}$ on $\mathbb R^m$ such that the following holds. 
	
	Let $N_o$ be the continuously varying norm on $\R^m\times\E$ such that  $N_o(p,\cdot)\equiv N(o,\cdot)$ for every $p\in\R^m$. Let $|\cdot|_{\hat{\mathcal X},N_o}$ be the subFinsler metric on $T\R^m$ defined as in~\eqref{eqn:VxN}, and let $d_{(\hat{\mathcal X},N_o)}$ be the subFinsler distance on $\mathbb R^m$, defined as in~\eqref{eqn:DxN}, by using the subFinsler metric $|\cdot|_{\hat{\mathcal X},N_o}$. 
	
	Then, the Gromov--Hausdorff tangent of $(M,d_{(\mathcal X,N)})$ at $o\in M$ is isometric to $(\mathbb R^m,d_{(\hat{\mathcal X},N_o)},0)$. Moreover, $(\mathbb R^m,d_{(\hat{\mathcal X},N_o)})$ is isometric to a quotient of a subFinsler Carnot group by one of its closed subgroups. If $o$ is regular, $(\mathbb R^m,d_{(\hat{\mathcal X},N_o)})$ is isometric to a subFinsler Carnot group.	
	
\end{theorem}
We remark that the construction of $\hat{\mathcal{X}}$, which is usually called the \textit{homogeneous nilpotent approximation of $\mathcal{X}$ at $o$}, can be made explicitely with respect to $\mathcal{X}$, compare with \cite[Section 2.1]{Jean}, and does not depend on $N$.

Giving for granted the construction of privileged coordinates, for which we give precise references in Section~\ref{sec6092e03a}, the proof of the Theorem~\ref{thm:Mitchell} is a direct consequence of the application of Theorem~\ref{thm60929a0e}(iv), see Section~\ref{sec6092e03a}.

\medskip
The last consequence we want to discuss is in the setting of Lie groups. We introduce some notation. Let $\mathbb G$ be a connected Lie group, and $\mathfrak g$ its Lie algebra. Given a vector subspace $\mathcal{H}\subseteq \mathfrak g$ of $\mathfrak g$, and a norm $b$ on $\mathcal{H}$, we associate to $(\mathcal{H},b)$ a left-invariant subFinsler structure $(\mathcal{D},b)$ as follows 
\begin{equation}\label{eqn:DistributionGroup}
\mathcal{D}_p:=\mathrm{d}L_p(\mathcal{H}),\,\, \quad b_p(X):=b((\mathrm{d}L_p)^{-1}X),\,\,\,\,\,\forall p\in\mathbb G, \forall X\in T_p\mathbb G,
\end{equation}
where $L_p(q):=p\cdot q$ for $p,q\in\mathbb G$. 

To each $(\mathcal{H},b)$, one attaches the (possibly infinite-valued) distance between $p,q\in\mathbb G$ as
\begin{equation}\label{eqn:DefinitionDistanceGroup}
	\begin{split}
		d^{(\mathcal{H},b)}(p,q):=\inf \Big\{&\int_a^b b_{\gamma(t)}(\gamma'(t)): \gamma:[a,b]\to \mathbb G\,\,\text{AC}, \gamma(a)=p,\gamma(b)=q, \\ 
		&\gamma'(t)\in\mathcal{D}_{\gamma(t)}\,\text{for a.e. $t\in [a,b]$}\Big\}.
	\end{split}
\end{equation}
Notice that if $\mathcal{H}$ is a bracket-generating vector subspace of $\mathfrak g$, the sub-Finsler structure $(\mathbb G,\mathcal{D},b)$ satisfies the bracket-generating condition and thus $d^{(\mathcal{H},b)}$ is finite-valued. Let us denote by $\mathrm{Gr}_{\mathfrak g}(k)$ the Grassmannian of $k$-planes endowed with the usual topology of the Grassmannian of a vector space. The proof of the following result is at the end of Section~\ref{sec:CCgroups} and it is an immediate consequence of Theorem~\ref{thm60929a0e}(iv).

\begin{theorem}[Convergence of distances on Lie groups]\label{thm:PerEmilio}
	Let $\mathbb G$ be a connected Lie group with Lie algebra $\mathfrak g$, and let $\mathcal{H}\subseteq \mathfrak g$ be a bracket-generating vector subspace of dimension $k$. 
	
	Let $\{\mathcal{H}_n\}_{n\in\mathbb N}$ be a sequence of vector subspaces of $\mathfrak g$ that converges in the topology of $\mathrm{Gr}_{\mathfrak g}(k)$ to $\mathcal{H}$, and let $\{b_n\}_{n\in\mathbb N}$ be a sequence of norms on $\mathfrak g$ that converges uniformly on compact sets to a norm $b$ on $\mathfrak g$. Then, being $d^{(\mathcal{H}_n,b_n)}$ and $d^{(\mathcal{H},b)}$ the distances defined as in~\eqref{eqn:DefinitionDistanceGroup}, we have
	$$
	d^{(\mathcal{H}_n,b_n)}\to d^{(\mathcal{H},b)}, \qquad \text{uniformly on compact subsets of $\mathbb G\times \mathbb G$}.
	$$
	
	Moreover, for any $p\in\mathbb G$, $(\mathbb G,d^{(\mathcal{H}_n,b_n)},p)\to (\mathbb G,d^{(\mathcal H,b)},p)$ in the pointed Gromov--Hausdorff topology as $n\to+\infty$. 
\end{theorem}

\medskip

We briefly describe the structure of the paper and we refer the reader to the introductions of the sections for more details. 

In Section~\ref{sec607ed0ed} we introduce the main notation and definitions of the paper, we study the continuity property of the End-point map, the notion of essentially open map and essentially non-holonomic set of vector fields, and we finally prove the open property of the End-point map associated to an essentially non-holonomic structure.

In Section~\ref{sec607ed104} we define and study the length, the energy and the \\Carnot-Carathéodory distance associated to a Lipschitz-vector-field structure and a continuously varying norm. We thus study how the distances behave under taking limit of the corresponding structures, and we prove the main Theorem~\ref{thm60929a0e}.

In Section~\ref{sec:Examples} we discuss some examples in which Theorem~\ref{thm60929a0e} applies, namely the fact that the asymptotic cone of the Riemannian Heisenberg group is the subRiemannian Heisenberg group; the subFinsler Mitchell's Theorem with continuously varying norms; a general convergence result for Carnot-Carathéodory distances on connected Lie groups; and a general convergence result for Carnot-Carathéodory distances on manifolds.

In the Appendix we give a self-contained proof of the version of the Gronwall Lemma that we need, we prove some ancillary lemmas about open maps, and we finally give a shorter and more direct proof of Theorem~\ref{thm60929a0e}(iv) in the case in which the vector fields are smooth.

\medskip
\textbf{Acknowledgements.} The authors are partially supported by the European Research Council (ERC Starting Grant 713998 GeoMeG `\emph{Geometry of Metric Groups}'). E.L.D. was partially supported by the Academy of Finland (grant
288501
`\emph{Geometry of subRiemannian groups}' and by grant
322898
`\emph{Sub-Riemannian Geometry via Metric-geometry and Lie-group Theory}'. 

Part of this research was done while the first author was visiting the second author at the University of Fribourg. The excellent work atmosphere is acknowledged.

The authors express their gratitute to Emilio Lauret for his interest in this work and his encouragement to add more material to a preliminary version of this article.

\section{The End-point map for Lipschitz-vector-field structures}\label{sec607ed0ed}

In this section we introduce the notion of Lipschitz-vector-field structure on a smooth manifold modelled by a finite-dimensional Banach space, see Definition~\ref{def:UniformlyLocallyLip} and Definition~\ref{def:LipschitzVectorFieldsStructure}. We give the definitions of the convergence of Lipschitz-vector-field structures, see Definition~\ref{defSeqTop}.

After the discussion of the basic definitions, we are going to prove the first main theorem of this section, see Theorem~\ref{thm6092f873}, which says that the domain of the End-point map associated to a Lipschitz-vector-field structure is open and the End-point map is continuous. The proof of Theorem~\ref{thm6092f873} heavily relies on Proposition~\ref{prop600c1758}, according to which the curves satisfying the Cauchy system~\eqref{eq609b82be} starting from a converging sequence $(o_n,\f_n,u_n)\in M\times\LipS(\E^*\otimes TM)\times L^\infty([0,1];\E)$ uniformly converge. The latter result relies on a general convergence criterion for flow lines of (not necessarily smooth) vector fields, see Propostion~\ref{prop6003179d}.

In Section~\ref{sub:EssentiallyNonHolonomic} we introduce the notion of essentially open map between two topological manifolds of the same dimension, see Definition~\ref{def:EssentiallyOpen}. We recall that such a notion is used to give the notion of essentially non-holonomic distribution of vector fields, see Definition~\ref{def6067306f}. 

In Section~\ref{sub:OpenMapOpenUniformly} we finally show the second main result of this section. Namely, we show that whenever a Lipschitz-vector-field structure $\hat \f$ modelled by $\mathbb E$ is essentially non-holonomic, the End-point map of $\hat\f$ enjoys a uniform openness property, see Theorem~\ref{thm60a3bb36} for further details. Such a uniform openness result, which follows both from the continuity result proved in Theorem~\ref{thm6092f873}, and from the ancillary Lemma~\ref{lem6082ea3b}, strongly requires the essentially non-holonomicity property. Eventually, the latter uniform openness  result will be of key importance in proving Lemma~\ref{lem604be4c3}, which is one of the main steps to prove the main result of the next section, see Theorem~\ref{thm60929a0e}(iv). We refer the reader to the introduction of the next section for further details.

\subsection{Lipschitz-vector-field structures on a manifold}\label{sub:LipschitzVectorFields}

In this subsection we study basic facts about Lipschitz-vector-field structures.
\begin{definition}[Uniformly locally Lipschitz sections $\LipS$]\label{def:UniformlyLocallyLip}
	Let $M$ be a smooth manifold and $E\to M$ a vector bundle.
	We say that a family $\cal X$ of sections of $E$ is \emph{uniformly locally Lipschitz}
	if for every $p\in M$ there exist a coordinate neighborhood of $p$ that trivializes $E$ and a constant $L>0$ so that every element of $\cal X$ is $L$-Lipschitz in this trivialization.
	We denote by $\LipS(E)$ the collection of all locally Lipschitz sections of $E$, i.e., sections $X$ so that the singleton $\{X\}$ is uniformly locally Lipschitz.
\end{definition}
It is a direct consequence of the definition that a family $\cal X$ of sections is uniformly locally Lipschitz if and only if it is 
uniformly Lipschitz on compact sets on every trivialization of $E$.


\begin{definition}[Sequential topology on uniformly locally Lipschitz sections]\label{defSeqTop}
%
	We define a sequential topology on $\LipS(E)$ as follows. We say that a sequence \\ $\{X_n\}_{n\in\N}\subset\LipS(E)$ converges to $X_\infty\in\LipS(E)$ if and only if
	$\{X_n\}_{n\in\N}$ is uniformly locally Lipschitz and \textit{$X_n\to X$ uniformly on compact subsets of $M$},
	that is, if every $p\in M$ has a coordinate neighborhood that trivializes $E$ on which $X_n$ converge to $X$ uniformly on compact sets.
\end{definition}

\begin{definition}[Lipschitz-vector-field structure]\label{def:LipschitzVectorFieldsStructure}
Let $\E$ be a finite-dimensional Banach space and $M$ a smooth manifold. We denote by $\E^*\otimes TM$ the vector bundle on $M$ whose fibers are $\E^*\otimes T_pM$. A section $\f$ of $\E^*\otimes TM$ is a choice, for each $p\in M$, of a linear map $\f|_p:\E\to T_pM$.
%
%
	An element $\f\in \LipS(\E^*\otimes TM)$ will be called a {\em Lipschitz-vector-field structre} on $M$ modelled by $\mathbb E$.
\end{definition}


The definition of the Banach spaces $L^p([0,1];\E)$ does not depend on the choice of a Banach norm on $\E$.
The predual of $L^\infty([0,1];\E)$ is $L^1([0,1];\E^*)$ with pairing
\[
\langle v | u \rangle := \int_0^1 \langle v(t)|u(t) \rangle \dd t ,
\]
for $u:[0,1]\to\E$ and $v:[0,1]\to\E^*$.
So, a sequence $\{u_n\}_{n\in\N}\subset L^\infty([0,1];\E)$ weakly* converges to $u_\infty\in L^\infty([0,1];\E)$ if for every $v\in L^1([0,1];\E^*)$ it verifies $\lim_{n\to\infty} \langle v | u_n \rangle = \langle v | u_\infty \rangle$.

%
We denote by $\DomEnd \subset M \times \LipS(\E^*\otimes TM) \times L^\infty([0,1];\E)$ 
the domain of the End-point map, as defined in Definition~\ref{def:EndPoin}.

\begin{remark}[Concatenation of flows of vector fields]\label{rem609b8d47}
	Let $X_1,\dots,X_\ell$ be locally Lipschitz vector fields on $M$. Take $\mathbb E=\mathbb R^\ell$ with the standard Euclidean norm, and with the canonical basis $\{e_1,\dots,e_\ell\}$, and define 
	$\f\in\LipS(\E^*\otimes TM)$ by extending linearly
	\[
	\f_p(e_i) := X_i(p),\qquad \text{for all $p\in M$ and $i\in 1,\dots,\ell$}.
	\]
	The Cauchy system in Equation~\eqref{eq609b82be}, for any $u:[0,1]\to\R^\ell$, becomes
	\[
	\begin{cases}
	\gamma'(t) &= \sum_{i=1}^\ell u_i(t) X_i(\gamma(t)) , \\
	\gamma(0) &= o.
	\end{cases}
	\]
	Fix $j\in\mathbb N$, $(t_1,\dots,t_j)\in\R^j$ and $e_{i_1},\dots,e_{i_j}\in\{e_1,\dots,e_\ell\}$.
	If we define the controls
	\newcommand{\sgn}{\mathrm{sgn}}
	\begin{equation*}
	u_{(t_1,\dots,t_j)}(s) := 
		T \sum_{k=1}^j \chi_{\left[\frac{\sum_{i=1}^{k-1}|t_i|}{T},\frac{\sum_{i=1}^{k}|t_i|}{T}\right]}(s)\,
		\sgn(t_k)\, e_{i_k} ,
	\end{equation*}
	where $T:=\sum_{i=1}^j |t_i|$,
	and 
	\[
	v_{(t_1,\dots,t_j)}(s)
	:= j \sum_{k=1}^j \chi_{\left[ \frac{k-1}{j} , \frac{k}{j} \right]}(s) t_k e_{i_k} ,
	\]
	we have that 
	\begin{equation}\label{EsplicitaControlli}
	\gamma_{(o,\f,u_{(t_1,\dots,t_j)})}(1)
	= \gamma_{(o,\f,v_{(t_1,\dots,t_j)})}(1)
		= \Phi_{X_{i_j}}^{t_j}\circ\dots\Phi_{X_{i_1}}^{t_1}(o),
	\end{equation}
	whenever one of the terms exists. 
	We recall that $\Phi_X^t(o)$ is the flow line at time $t$ of the vector field $X$ starting at $o$.
\end{remark}

\subsection{Continuity of the End-point map}\label{sub:Continuity}

In this section we shall prove the following continuity theorem for the End-point map. 
On $\DomEnd \subset M \times \LipS(\E^*\otimes TM) \times L^\infty([0,1];\E)$ we consider the topology that is the product of the manifold topology on $M$, the sequential topology defined in Definition~\ref{defSeqTop} on $ \LipS(\E^*\otimes TM) $, and the weak* topology on $L^\infty([0,1];\E)$.

\begin{theorem}\label{thm6092f873}
	Let $M$ be a smooth manifold and let
	$\E$ be a finite-dimensional Banach space.
	Then the domain $\DomEnd$ of the End-point map is open and the function $\End:\DomEnd\to M$ is continuous.	
	
	Moreover, given $(\hat o,\hat\f,\hat u)\in\DomEnd$, 
	and given
	 a weak*-compact neighborhood $\scr U$ of $\hat u$ such that $(\hat o,\hat\f,u)\in \DomEnd$ for every $u\in \scr U$, we have that that
	the limit
	\begin{equation}\label{eq6098ffbd}
	\lim_{(o,\f)\to(\hat o,\hat\f)} \End^\f_o = \End^{\hat\f}_{\hat o}
	\end{equation}
	is uniform on $\scr U$.
\end{theorem}

Inspired by \cite[Proposition 3.6]{MR3739202}, we prove the following ancillary proposition.

\begin{proposition}\label{prop6003179d}
	Let $M\subset\R^N$ be a closed subset 
	and for $n\in\N\cup\{\infty\}$ let $X_n:[0,1]\times M\to \R^N$
	be measurable (thought as non-autonomous vector-fields) such that there are $R,L\ge0$ with 
	$\|X_n\|_{L^\infty} \le R$ and there exists a full measure subset $\mathcal{T}\subset[0,1]$ such that
	$ |X_n(t,p)-X_n(t,q)| \le L |p-q|$ for all $p,q\in M$, and for all $t\in \mathcal{T}$.
	For each $n$, let $\gamma_n:[0,1]\to M$ be an integral curve of $X_n$.
	Suppose that $\gamma_n(0)\to\gamma_\infty(0)$ and that, for every $t\in[0,1]$,
	\begin{equation}\label{eq6002cad4}
	\lim_{n\to\infty} \int_0^t X_n(s,\gamma_\infty(s)) \dd s
		= \int_0^t X_\infty(s,\gamma_\infty(s)) \dd s .
	\end{equation}
	Then $\gamma_n\to\gamma_\infty$ uniformly.
\end{proposition}
\begin{proof}
	The family of curves $\gamma_n$ is equibounded and equicontinuous with values in $\R^N$.
	By Ascoli-Arzelà, $\{\gamma_n\}_n$ is precompact.
	Let $\xi:[0,1]\to M$ be a limit curve.
	We will show that $\xi = \gamma_\infty$, thus proving that the whole sequence $\gamma_n$ uniformly converge to $\gamma_\infty$.
	
	For each $n$, define $o_n := \gamma_n(0)$ and 
	\[
	\eta_n(t) := o_n + \int_0^t X_n(s,\gamma_\infty(s)) \dd s .
	\]
	By the assumption~\eqref{eq6002cad4}, $\eta_n(t)\to\gamma_\infty(t)$ for every $t\in[0,1]$.
	
	Next,
	\begin{align*}
	|\gamma_n(t)-\gamma_\infty(t)|
	&\le \begin{multlined}[t][0.75\textwidth]
	|o_n - o_\infty| + \int_0^t |X_n(s,\gamma_n(s)) - X_n(s,\gamma_\infty(s))| \dd s  \\
		+ \left| \int_0^t (X_n(s,\gamma_\infty(s)) - X_\infty(s,\gamma_\infty(s))) \dd s \right| 
	\end{multlined}\\
	&\le 2|o_n-o_\infty| + L \int_0^t |\gamma_n(s)-\gamma_\infty(s)| \dd s + |\eta_n(t)-\gamma_\infty(t)| .
	\end{align*}
	Letting $n\to\infty$, we obtain, for every $t\in[0,1]$,
	\[
	|\xi(t) - \gamma_\infty(t)| \le L \int_0^t |\xi(s)-\gamma_\infty(s)| \dd s .
	\]
	Starting with $\|\xi-\gamma_\infty\|\le C$ and iterating this inequality, we get $\xi=\gamma_\infty$.
	The latter implication is actually an instance of the Gronwall inequality $f'\le L f$.
\end{proof}

\begin{remark}
	Suppose that the hypothesis of Proposition~\ref{prop6003179d} except~\eqref{eq6002cad4} are satisfied.
	We claim that, if $X_n\to X_\infty$ pointwise a.e.~on $[0,1]\times M$, then~\eqref{eq6002cad4} is satisfied on every curve $\gamma$.
	Indeed, first, for a.e.~$t\in[0,1]$, we have that $X_n(t,\cdot)\to X_\infty(t,\cdot)$ almost everywhere on $M$.
	Since $X_n$ are uniformly Lipschitz and bounded in the spatial variable, 
	for a.e.~$t\in[0,1]$, actually the convergence $X_n(t,\cdot)\to X_\infty(t,\cdot)$ is  uniform on compact subsets of $M$.
	Therefore, since $\gamma([0,1])$ is compact, for almost every $s\in[0,1]$, $\lim_{n\to\infty} X_n(s,\gamma(s)) = X_\infty(s,\gamma(s))$.
	Since $X_n$ are uniformly bounded, by the Lebesgue dominated convergence theorem, we can exchange the limit and the integral in the left-hand side of~\eqref{eq6002cad4}.
\end{remark}

\begin{lemma}\label{lem60031617}
	Let $\E_1$ and $\E_2$ be two finite-dimensional Banach spaces.
	Suppose that $\{u_n\}_{n\in\N}\subset L^\infty([0,1];\E_1)$ is a sequence weakly* converging to $u_\infty$
	and that $A_s:\E_1\to\E_2$ is a continuous family of linear maps, with $s\in[0,1]$.
	Then $v_n(s) := A_s[u_n(s)]$ is a sequence in $L^\infty([0,1];\E_2)$ weakly* converging to $v_\infty$.
\end{lemma}
\begin{proof}
	We denote by $\langle \cdot|\cdot \rangle$ the pairing of a Banach space with its dual space.
	We need to show that, for every $\alpha\in L^1([0,1];\E_2^*)$,
	\begin{equation}\label{eq60d9877e}
	\lim_{n\to\infty} \langle \alpha|v_n \rangle 
	= \langle \alpha|v_\infty \rangle .
	\end{equation}
	Since $A$ is continuous in $s$, the operator norm of the adjoint operators $A_s^*$ is bounded uniformly in $s\in[0,1]$, by $R$ say.
	Hence $\int_0^1 \|A^*_s\alpha(s)\| \dd s \le R \|\alpha\|_{L^1} < \infty$,
	that is, $A^*\alpha: s\mapsto A_s^*\alpha(s)$ is an element of $L^1([0,1];\E_1^*)$.
	We conclude that, for every $\alpha\in L^1([0,1];\E_2^*)$,
	\begin{align*}
	\lim_{n\to\infty} \langle \alpha|v_n \rangle
	&= \lim_{n\to\infty} \int_0^1 \langle \alpha(s)|A_s[u_n(s)] \rangle \dd s 
	= \lim_{n\to\infty} \int_0^1 \langle A_s^*\alpha(s)|u_n(s) \rangle \dd s \\
	&= \lim_{n\to\infty} \langle A^*\alpha|u_n \rangle 
	=  \langle A^*\alpha|u_\infty \rangle 
	= \langle \alpha|Au_\infty \rangle ,
	\end{align*}
	hence~\eqref{eq60d9877e} indeed holds for every $\alpha\in L^1([0,1];\E_2^*)$.
\end{proof}

The next proposition will be at the core of the proof of Theorem~\ref{thm6092f873}.
\begin{proposition}\label{prop600c1758}
	 Let $\{(o_n,\f_n,u_n)\}_{n\in\N} \subset M \times \LipS(\E^*\otimes TM) \times L^\infty([0,1];\E)$ be a sequence converging to $(o_\infty,\f_\infty,u_\infty)\in\DomEnd$.
	 Then the following hold:
	 \begin{enumerate}[leftmargin=*,label=(\arabic*)]
	 \item\label{item60a24637}
	 There is $N\in\N$ such that $\{(o_n,\f_n,u_n)\}_{n\ge N}\subset\DomEnd$,
	 that is, for every $n\ge N$ the Cauchy system~\eqref{eq609b82be} has an integral curve $\gamma_{(o_n,\f_n,u_n)}$ defined on $[0,1]$;
	 \item\label{item60a24638}
	 The sequence of integral curves $\gamma_{(o_n,\f_n,u_n)}$ uniformly converge to $\gamma_{(o_\infty,\f_\infty,u_\infty)}$.
	 \end{enumerate}
\end{proposition}
\begin{proof}
	The convergence in the assumptions means that $o_n\to o_\infty$ in $M$,
	$u_n\to u_\infty$ weakly* in $L^\infty([0,1];\E)$,
	$\f_n\to\f_\infty$ uniformly on compact sets of every trivialization of $\mathbb E^*\otimes TM$, and
	$\{\f_n\}_n$ is uniformly locally Lipschit on compact sets of every trivialization of $\mathbb E^*\otimes TM$.
	For $n\in\N\cup\{\infty\}$, denote by $\gamma_n:[0,T_n)\to M$ the maximal integral curve of the Cauchy system~\eqref{eq609b82be} with control $u_n$ and initial point $o_n$.
	Notice that $\gamma_n$ can be extended to $T_n$ if and only if $T_n=1$, if and only if $\gamma_n([0,T_n))$ is contained in a compact subset of $M$.
	
	Let $\rho$ be a complete Riemannian metric on $M$ and let $\iota:M\into\R^N$ be a Riemannian isometric embedding.
	The $\rho$-length of curves in $M$ is thus equal to their Euclidean length in $\R^N$.
	We denote by $d_\rho$ and $\diam_\rho$ the Riemannian distance and the corresponding diameter on $M$ defined by $\rho$.
	Notice that, for every $r\ge0$ and $p\in M$, the closed $\rho$-ball of radius $r$ and center $p$ in $M$, $\bar B_\rho(p,r)$, is a compact subset of $\R^N$.
	
	By means of the isometric embedding $\iota$,
	we interpret the sections $\f_n\in\LipS(\E^*\otimes TM)$ 
	as locally uniformly Lipschitz maps $\f_n:M\to\E^*\otimes\R^N$.
	We will denote by $\|\f_n(p)\|$ the operator norm of the linear map $\f_n(p):\E\to\R^N$.
	
	Let $D = \diam_\rho(\gamma_\infty([0,1]))$, 
	$H = \sup\{\|\f_n(p)\|:n\in\N, p\in \bar B_\rho(o_\infty,4D)\}<\infty$, 
	and $R = \sup_n \|u_n\|_{L^\infty}<\infty$.
	
	We claim that, for every $\hat t\in[0,1]$, 
	if there is $N$ such that $\hat t<T_n$ for every $n\ge N$ and
	if $\gamma_n(\hat t)\to \gamma_\infty(\hat t)$, 
	then there is $\hat N\ge N$ such that, for all $n\ge \hat N$, the curve $\gamma_n$ is defined on $[\hat t,\hat t + \frac{D}{RH}]\cap[0,1]$
	and $\gamma_n([0,\hat t + \frac{D}{RH}]\cap[0,1]) \subset \bar B_\rho(o_\infty,4D)$ for all $n\ge\hat N$.
	
	To prove the claim, first notice that there is $\hat N\in\N$ such that $\gamma_n(\hat t)\in \bar B(o_\infty,2D)$ for all $n\ge\hat N$.
	Next, for almost every $s\in[0,1]$ with $\gamma_n(s)\in \bar B_\rho(o_\infty,4D)$, we have
	\begin{equation}\label{eq604b5375}
	|\gamma_n'(s)| 
	= |\f_n(\gamma_n(s))u_n(s)|
	\le \|\f_n(\gamma_n(s))\| |u_n(s)|
	\le HR ,
	\end{equation}
	and so, if $\gamma_n(r)\in \bar B_\rho(o_\infty,4D)$ for all $r\in[\hat t,s]$, then
	\begin{equation}\label{eq604b546c}
	|\gamma_n(s)-\gamma_\infty(\hat t)|
	\le |\gamma_n(s)-\gamma_n(\hat t)| + |\gamma_n(\hat t)-\gamma_\infty(\hat t)|
	\le HR |s-\hat t| + 2D .
	\end{equation}
	Now, fix $n>\hat N$ and define
	\[
	\hat s = \min\left\{ \hat t + \frac{D}{HR} ,
		\ \inf\{ r\ge\hat t:\gamma_n(r)\in\de\bar B_\rho(o_\infty,4D)\} , 
		\ T_n
		\right\} .
	\]
	If $\hat s=T_n$, then $\gamma_n([0,T_n))\subset \bar B_\rho(o_\infty,4D)$, which is compact, and thus $T_n=1$ and $\gamma_n$ is defined on $[0,1]$.
	Suppose $\hat s<T_n$.
	From~\eqref{eq604b546c}, we obtain 
	$|\gamma_n(\hat s)-\gamma_\infty(\hat t)| \le 3D$
	and thus we deduce that $\hat s = \hat t + \frac{D}{HR}$.
	We conclude that
	$\gamma_n$ is defined on $[0,\hat s] = [0,\hat t+\frac{D}{HR}]\cap[0,1]$ and that
	 $\gamma_n([\hat t,\hat s])\subset \bar B_\rho(o_\infty,4D)$ for all $n\ge\hat N$.
	
	Since $\frac{D}{HR}$ is a fixed positive quantity, the above claim readily implies the first part of the lemma.
	For the second part of the proof, we assume that $\gamma_n$ is defined on $[0,1]$ for all $n$.
	
	We claim that, for every $\hat t\in[0,1]$, 
	if $\gamma_n(\hat t)\to \gamma_\infty(\hat t)$, 
	then $\gamma_n|_{[\hat t,\hat t + \frac{D}{RH}]}$ uniformly converge to $\gamma_\infty|_{[\hat t,\hat t + \frac{D}{RH}]}$.
	
	To prove the latter claim, notice that from the previous claim 
	we have that there is $\hat N\ge N$ such that, for all $n\ge \hat N$,
	$\gamma_n([0,\hat t + \frac{D}{RH}]\cap[0,1]) \subset \bar B_\rho(o_\infty,4D)$.
	We next apply Proposition~\ref{prop6003179d}, whose hypotheses must be met.
	Define the vector fields $X_n:[0,1]\times M\to TM$, $X_n(t,p) = \f_n|_p[u_n(t)]$.
	Since $\bar B_\rho(o_\infty,4D)$ is compact, there is $L$ so that $\f_n$ is $L$-Lipschitz on $\bar B_\rho(o_\infty,4D)$ for all $n\in\N$.
	Thus, if $p,q\in \bar B_\rho(o_\infty,4D)$, then 
	\[
	|X_n(t,p)-X_n(t,q)| 
	\le \|\f_n(p)-\f_n(q)\| \|u_n\|_{L^\infty} 
	\le LR |p-q| .
	\]
	If $\gamma:[0,1]\to M$ is a continuous curve, then 
	\begin{align*}
	\int_0^t X_n(s,\gamma(s)) \dd s
	&= \int_0^t \f_\infty|_{\gamma(s)}[u_n(s)] \dd s 
		+ \int_0^t \f_n|_{\gamma(s)} [u_n(s)] - \f_\infty|_{\gamma(s)}[u_n(s)] \dd s ,
	\end{align*}
	where
	\[
	\lim_{n\to\infty} \left| \int_0^t \f_n|_{\gamma(s)} [u_n(s)] - \f_\infty|_{\gamma(s)}[u_n(s)] \dd s \right|
	\le \lim_{n\to\infty} \|\f_n-\f_\infty\|_{L^\infty} R t ,
	= 0
	\]
	and $v_n(s):=\f_\infty|_{\gamma(s)}[u_n(s)]$ weakly* converge to $v_\infty(s):= \f_\infty|_{\gamma(s)}[u_\infty(s)]$ by Lemma~\ref{lem60031617}.
	Since $v\mapsto \int_0^t v(s) \dd s$ is a continuous linear operator from $L^\infty([0,1];\R^N)$ to $\R^N$, we then have
	\[
	\lim_{n\to\infty} \int_0^t \f_\infty|_{\gamma(s)}[u_n(s)] \dd s
	= \int_0^t \f_\infty|_{\gamma(s)}[u_\infty(s)] \dd s .
	\]
	
	We have thus shown that the non-autonomous vector fields $X_n$ satisfy all conditions in Proposition~\ref{prop6003179d} on $\bar B_\rho(o_\infty,4D)$.
	Since $\gamma_n([\hat t,\hat t+\frac{D}{HR}])\subset \bar B_\rho(o_\infty,4D)$ for all $n\ge\hat N$
	and $\gamma_n(\hat t) \to \gamma_\infty(\hat t)$, 
	Proposition~\ref{prop6003179d} ensures that $\gamma_n|_{[\hat t,\hat t + \frac{D}{RH}]}$ uniformly converge to $\gamma_\infty|_{[\hat t,\hat t + \frac{D}{RH}]}$.
	The claim is proven.
	
	Finally, since the constant $\frac{D}{RH}$ does not depend on $\hat t$, we can subdivide $[0,1]$ into intervals of length less than $\frac{D}{RH}$ and apply the above claim iteratively to each interval, concluding the proof of the proposition.
\end{proof}

\begin{proof}[Proof of Theorem~\ref{thm6092f873}]
	Proposition~\ref{prop600c1758}.\ref{item60a24637} implies that $\DomEnd$ is open,
	while Proposition~\ref{prop600c1758}.\ref{item60a24638} implies that $\End$ is continuous.
	The uniform limit~\eqref{eq6098ffbd} is a direct consequence of 
	the continuity of $\End$ simultaneously in the three variables $(o,\f,u)$, and a standard compactness argument.
\end{proof}

\subsection{Essentially non-holonomic Lipschitz distributions}\label{sub:EssentiallyNonHolonomic}
In this section we discuss the notions of essentially open maps and essentially non-holonomic distributions of vector fields.

\begin{definition}[Essentially open map]\label{def:EssentiallyOpen}
	A continuous map $f:M\to N$ between topological manifolds of the same dimension $k$ is said to be {\em essentially open at $p\in M$ at scale $U$} if
	\begin{enumerate}
		\item $U$ is a neighborhood of $p$ such that $U$ is homeomorphic to the Euclidean $k$-dimensional ball, and $\partial U$ is homeomorphic to the $(k-1)$-dimensional Euclidean sphere $\mathbb S^{k-1}$;
		\item there exists a neighborhood $V$ of $f(p)$ such that $V$ is homeomorphic to the Euclidean $k$-dimensional ball;
		\item $f(\partial U)\subseteq V\setminus f(p)$;
		\item The induced group homomorphism $f_*: H_{k-1}(\partial U )\cong \mathbb Z\to H_{k-1}(V\setminus f(p))\cong \mathbb Z$ is not constant.
	\end{enumerate}
\end{definition}

\begin{remark}\label{rem:RemOnEssOpen}
	We claim that if a map $f$ is essentially open at $p$,
	then $f(p)$ is contained in the interior of $f(U)$. Indeed, up to homeomorphism, $f$ induces by restriction a map $f:\overline B(0,r_1)\subseteq\mathbb R^k\to B(0,r_2)\subseteq\mathbb R^k$, for some radii $r_1,r_2$, such that $f(0)=0$, $f(\partial B(0,r_1))\subseteq B(0,r_2)\setminus\{0\}$, and such that $f|_{\partial B(0,r_1)}:\partial B(0,r_1)\to B(0,r_2)\setminus\{0\}$ induces a non-constant group homomorphism between the $(k-1)$-th homology groups. Hence, arguing as in the first few lines of the proof of Lemma~\ref{lem608be6cc}, the map $\widetilde f:\mathbb S^{k-1}\to\mathbb S^{k-1}$ defined as
	$$
	\widetilde f(x):=\frac{f(r_1x)}{|f(r_1x)|},
	$$
	has nonzero degree. Then we can apply Lemma~\ref{lem6082ea3b} with the constant sequence $f_n:=f$ to obtain that $f(0)$ is contained in the interior of $f(B(0,r_1))$.
	
	We notice that by virtue of Lemma~\ref{lem6082ea3b}, the essential openness at a point $p$ at some scale, implies a surjectivity property at $p$ which is stable with respect to uniform convergence. This stability is of crucial importance in Definition~\ref{def6067306f}. Indeed, if we naïvely only require that $\phi$ in Definition~\ref{def6067306f} is  open at $\hat t$, it is not possible to prove the uniform openness in Theorem~\ref{thm60a3bb36},
	because the property of being open is not stable under uniform convergence. 
	For example, consider the functions $f_t(x):=tx$ defined on $[0,1]$ and let $t\to 0$.
\end{remark}

We show in the next proposition that the condition of 
essential non-holonomicity (as in Definition~\ref{def6067306f}) is a weakening of the bracket-generating condition for smooth vector fields.
\begin{proposition}\label{prop:EveryBracket}
	Every bracket-generating family of smooth vector fields is essentially non-holonomic.
\end{proposition}
\begin{proof}
	Let $\scr F\subset\Gamma(TM)$ be a bracket-generating family of smooth vector fields.
	Due to \cite[Lemma~3.33]{MR3971262}, cf.~also \cite[Section~5.4]{MR2062547}, we have that for every $T>0$ and every $o\in M$, there are $X_1,\dots,X_m\in\scr F$
	such that the map
	\[
	\phi(t_1,\dots,t_m) := \Phi^{t_m}_{X_m}\circ\dots\circ\Phi^{t_1}_{X_1}(o) 
	\]
	is a topological embedding into $M$ from an open subset $U$ of $\R^m$ 
	contained in the set $\{t\in\R^m:|t|_1<T\}$. Then an application of Lemma~\ref{lem608be6cc} gives the sought conclusion.
\end{proof}

\subsection{Uniform openness of the End-point map}
\label{sub:OpenMapOpenUniformly}
In this final part of Section~\ref{sec607ed0ed}, we prove that whenever a Lipschitz-vector-field $\hat\f$ is essentially non-holonomic, and a point $o\in M$ is given, then the End-point map $\End^\f_o(\cdot)$ is open at $0$ uniformly as $\f$ varies in a neighborhood of $\hat\f$. 
\begin{theorem}\label{thm60a3bb36}
	Let $\hat\f\in \LipS(\E^*\otimes TM)$ be essentially non-holonomic, let
	$o\in M$, and 
	$R>0$.
	Then there are neighborhoods $\scr F$ of $\hat\f$ and $U$ of $o$
	such that 
	\[
	\f\in\scr F \Rightarrow
	U \subset \End^\f_o(B_{L^\infty}(0,R)).
	\]
\end{theorem}
\begin{proof}
	Let $m$ be the dimension of the manifold $M$. Arguing by contradiction, suppose that there are $\f_n\to\hat\f$ and $p_n\to o$ such that $p_n\notin \End^{\f_n}_o(B_{L^\infty}(0,R))$.
	
	Let $e_1,\dots,e_r$ be a basis of $\E$ as in Definition~\ref{def6067306f}.
	Given	$\sigma:\{1,\dots,m\}\to\{1,\dots,r\}$ and $\f\in\LipS(\E^*\otimes TM)$, define, for every $j=1,\dots,m$, 
	\[
	X^{\sigma,\f}_j(p) = \f(p,e_{\sigma(j)}) ,
	\]
	and, for every $(s_1,\dots,s_m,t_1,\dots,t_m)\in\mathbb R^{2m}$, let us define
	\[
	\psi^{\sigma,\f}(s_1,\dots,s_m,t_1,\dots,t_m) := 
	\Phi^{s_1}_{-X^{\sigma,\f}_1}\circ\dots\circ\Phi^{s_m}_{-X^{\sigma,\f}_m}\circ
	\Phi^{t_m}_{X^{\sigma,\f}_m}\circ\dots\circ\Phi^{t_1}_{X^{\sigma,\f}_1}(o),
	\]
	whenever it exists.
	Notice that 
	\[
	\psi^{\sigma,\f}(s_1,\dots,s_m,t_1,\dots,t_m)
	= \End^{\f}_o(u(s_1,\dots,s_m,t_1,\dots,t_m))
	\]
	where, for every $t\in [0,1]$,
	\[
	u(s_1,\dots,s_m,t_1,\dots,t_m)(t) = 2m\left(
	\sum_{j=1}^m
	t_j e_{\sigma(j)}\chi_{(\frac{j-1}{2m},\frac{j}{2m})}(t) 
	-\sum_{j=1}^m
	s_{j} e_{\sigma(j)}\chi_{(\frac{2m-j}{2m},\frac{2m-j+1}{2m})}(t) \right),
	\]
	and there is $C>0$ such that
	$\|u(s_1,\dots,s_m,t_1,\dots,t_m)\|_{L^\infty}\le C \sum_j(|t_j|+|s_j|)$.
	
	Since $\hat\f$ is essentially non-holonomic, 
	there are  $\hat t\in\R^m$, a neighborhood $\Omega$ of $\hat t$, and $\sigma$ such that $2C|t|_1 < R$ for all $t\in\Omega$, 
	 and $\hat f = \psi^{\sigma,\hat\f}(\hat t,\cdot)$ restricted to $\Omega$ is essentially open at $\hat t$. We recall that by definition $|(t_1,\dots,t_m)|_1:=\sum_{i=1}^m |t_i|$, for every $(t_1,\dots,t_m)\in\mathbb R^m$.
	
	Now, the maps $f_n = \psi^{\sigma,\f_n}(\hat t,\cdot)$ are continuous on $\Omega$ and converge uniformly to~$\hat f$,
	by Theorem~\ref{thm6092f873}.
	By reasoning as at the beginning of Remark~\ref{rem:RemOnEssOpen}, we have that at some scale around $\hat t$, $\hat f$ satisfies the hypothesis on the degree required to apply Lemma~\ref{lem6082ea3b}.  Hence, using Lemma~\ref{lem6082ea3b}, there is an open neighborhood $U$ of $o$ such that $U\subset f_n(\Omega)$ for large $n$.
	Since $f_n(\Omega) \subset \End^{\f_n}_o(B_{L^\infty}(0,R))$, 
	we reach a contradiction with $p_n\to o$ and $p_n\notin \End^{\f_n}_o(B_{L^\infty}(0,R))$.
\end{proof}

\section{Sub-Finsler distances for Lipschitz-vector-field structures}\label{sec607ed104}
	
	In this section we shall fix a smooth manifold $M$ and a Banach space $\mathbb E$. In $M$ we shall fix a Lipschitz-vector-field structure $\f$ modelled by $\mathbb E$, as in Definition~\ref{def:LipschitzVectorFieldsStructure}.
	
	In Section~\ref{sub:DefinitionEnergyLength}, given a continuously varying norm $N$ on $M\times \mathbb E$, we are going to define the energy and the length functionals associated to controls $u\in L^\infty([0,1];\mathbb E)$,
	we prove that they are lower semicontinuous, see Proposition~\ref{prop600bf91a}, and we prove that every curve has a constant speed reparametrization, see Lemma~\ref{lem604ba336}. 
	
	In Section~\ref{sub:CCdistance}, we define the distance $d_{(\f,N)}(p,q)$ between $p,q\in M$ to be the infimum of the energy (equivalently the length) of a control that gives raise to a curve that connects the two points, see Definition~\ref{def:CCdistance}. We then prove the local existence of geodesics, see Proposition~\ref{prop600c157a}, and in Proposition~\ref{prop60a7d838} we give a criterion, ultimately based on the growth of $(\f,N)$, to show that $(M,d_{(\f,N)})$ is a complete, boundedly compact, geodesic metric space. We prove that the topology generated by $d_{(\f,N)}$ is larger than the topology of $M$, see Lemma~\ref{lem60e6e893}. We shall then show that they are equal if $\f$ is essentially non-holonomic, see Theorem~\ref{thm60929a0e}(ii).
	
	In Section~\ref{sec:CClength} we shall address the problem of linking the definition of the distance given in Section~\ref{sub:DefinitionEnergyLength} with the Lagrangian one given directly on the curves. In particular, having a couple $(\f,N)$ on $M$, one can define a sub-Finsler metric on $TM$, see ~\eqref{eqn:NormFn}. We show that a curve $\gamma:[0,1]\to M$ is $d_{(\f,N)}$-Lipschitz if and only if $\gamma'$ has a bounded sub-Finsler metric as in~\eqref{eqn:NormFn}, see Lemma~\ref{lem60e5552f}. As a consequence, on $d_{(\f,N)}$-Lipschitz curves, the length associated to $d_{(\f,N)}$ coincides with the one associated to the sub-Finsler metric, see Proposition~\ref{prop60e6e5a4}. We thus conclude that the distance $d_{(\f,N)}(p,q)$ is the infimum of the sub-Finsler lengths of the curves connecting the points $p,q$, see Corollary~\ref{Cor:ApproccioAstrattoUgualeApproccioConcreto}.
	
	In Section~\ref{sec:LimitCC} we finally investigate what happens when we take the limit of CC distances. We first prove a relaxation result which tells us that when $(\f_n,N_n)\to (\f_\infty,N_\infty)$ the distance $d_{(\f_\infty,N_\infty)}$ on compact sets can be recovered as the relaxation of the distances $d_{(\f_n,N_n)}$, see Lemma~\ref{prop604f912b}. Hence we prove that when $(\f_n,N_n)\to (\f_\infty,N_\infty)$ and $\f_\infty$ is essentially non-holonomic, then $\{d_{(\f_n,N_n)}\}_{n\in\mathbb N}$ are equi-continuous functions on compact sets, see Lemma~\ref{lem604be4c3} and Proposition~\ref{prop604be4c9}. The equi-continuity together with the relaxation property allows us to show that if $(\f_n,N_n)\to (\f_\infty,N_\infty)$, and $\f_\infty$ is essentially non-holonomic, then $d_{(\f_n,N_n)}$ converges to $d_{(\f_\infty,N_\infty)}$ locally uniformly, see Theorem~\ref{thm60929a0e}(iii). Finally, when $(M,d_{(\f_\infty,N_\infty)})$ is boundedly compact, the local uniform convergence can be upgraded to a uniform convergence on compact sets by means of the metric Lemma~\ref{Lemma2}, see Theorem~\ref{thm60929a0e}(iv). The boundedly compact hypothesis is necessary to obtain such a uniform convergence on compact sets, see the example in Remark~\ref{rem:EXAMPLE}.


\subsection{Continuous norms along a segment}
In this section we give a preliminary and technical discussion about norms depending on a parameter that we will later use.

Let $N:[0,1]\times\E\to[0,\infty)$ be a continuous function so that, for every $t\in[0,1]$, the restriction $N_t:= N(t,\cdot)$ is a norm on the finite-dimensional vector space $\E$.
We define, for a measurable  $u:[0,1]\to\E$,
\begin{align*}
	\|u\|_{N,1} &:= \int_0^1 N_t(u(t)) \dd t , \\
	\|u\|_{N,\infty} &:= \esssup_{t\in[0,1]} N_t(u(t)) .
\end{align*}
Notice that, since $N$ is continuous, $\|u\|_{N,1}$ is a Banach norm on $L^1([0,1];\E)$ 
and $\|u\|_{N,\infty}$ is a Banach norm on $L^\infty([0,1];\E)$, and both are bi-Lipschitz equivalent to the standard Banach norms on those spaces.

If $\|\cdot\|$ is a Banach norm on a vector space, we denote by $\|\cdot\|^*$ its Banach dual norm on the dual space.
With this convention in mind, we denote by $N^*$ the continuous function $[0,1]\times\E^*\to[0,\infty)$ that, 
for each $t\in[0,1]$, gives the norm $N_t^*$ on $\E^*$ dual to $N_t$, i.e.,
\[
N_t^*(w) = \sup\left\{ \langle w|u \rangle : u\in\E,\ N_t(u)\le 1 \right\} .
\]

\begin{remark}\label{rem60e32e39}
	Recall the following consequence of Hahn--Banach Theorem, cfr.~\cite[Corollary 1.4]{MR2759829}:
	if $(V,\|\cdot\|)$ is a Banach space, then, for every $u\in V$,
	\[
	\|u\| = \max\{\langle w |u \rangle : w\in V^*,\ \|w\|^*\le 1 \} .
	\]
\end{remark}

\begin{lemma}\label{lem60a26232}
	Let $N:[0,1]\times\E\to[0,\infty)$ be a continuous function so that, for every $t\in[0,1]$, the restriction $N_t:= N(t,\cdot)$ is a norm.
	Then 
	\begin{equation}\label{eq60e3079d}
	\begin{array}{ll}
	\|\cdot\|_{N^*,\infty} = \|\cdot\|_{N,1}^* & \text{ on }L^\infty([0,1];\E^*)\text{ and} \\
	\|\cdot\|_{N,\infty} = \|\cdot\|_{N^*,1}^* & \text{ on }L^\infty([0,1];\E),
	\end{array}
	\end{equation}
	and 
	\begin{equation}\label{eq60e307a5}
	\begin{array}{cl}
	\|\cdot\|_{N^*,1} = \|\cdot\|_{N,\infty}^* & \text{ on }L^1([0,1];\E^*)\text{ and} \\
	\|\cdot\|_{N,1} = \|\cdot\|_{N^*,\infty}^* & \text{ on }L^1([0,1];\E) .
	\end{array}
	\end{equation}
	In particular, if $u:[0,1]\to\E$, then
	\begin{align}
	\label{eq60e32e58}
	\|u\|_{N,1} &= \max\left\{ \langle w|u \rangle:
		w\in L^\infty([0,1];\E^*),\ \|w\|_{N^*,\infty} \le 1 \right\} ,\\
	\label{eq60e32e5f}
	\|u\|_{N,\infty} &= \sup\left\{ \langle w|u \rangle :
		w\in L^1([0,1];\E^*),\ \|w\|_{N^*,1} \le 1 \right\} .
	\end{align}
	Moreover, if $w\in L^\infty([0,1];\E^*)$ is an argument of the maximum of~\eqref{eq60e32e58},
	then, for almost every $t\in[0,1]$, 
	$N_t^*(w(t)) = 1$ and $\langle w(t)|u(t) \rangle = N_t(u(t))$.
\end{lemma}

Notice that~\eqref{eq60e3079d} is a pair of genuine identities of norms, because $L^\infty([0,1];\E^*)$ is the dual space of $L^1([0,1];\E)$
and $L^\infty([0,1];\E)$ is the dual space of $L^1([0,1];\E^*)$.
The second pair of identities~\eqref{eq60e307a5} is understood via the standard embedding of a Banach space into its bi-dual, because $L^1$ is not reflexive, 
i.e., the dual of $L^\infty([0,1];\E^*)$ (resp. $L^\infty([0,1];\E)$) is not $L^1([0,1];\E)$ (resp. $L^1([0,1];\E^*)$).

The fact that $\|u\|_{N,1}$ is a maximum is an application of Remark~\ref{rem60e32e39}.
On the contrary, $\|u\|_{N,\infty}$ is only a supremum
and we can see this phenomenon with $\E=\R$ and $N_1\equiv|\cdot|$:
indeed, 
on the one hand, $u(x):=x$ is an element of $L^\infty([0,1];\R)$ and $\esssup_x|u(x)|=1$;
on the other hand, there is no function $w\in L^1([0,1];\R)$ so that $\int_0^1|w(x)|\dd x\le 1$ and $\int_0^1 xw(x) \dd x = 1$. One can show the latter claim from the fact that it would hold 
$$
1=\int_0^1 xw(x)\dd x\leq \int_0^1 |w(x)|\dd x \leq 1,
$$
from which it would follow that $\int_0^1 xw(x)\dd x=\int_0^1 |w(x)|\dd x$. Writing $w$ as the difference between its positive and negative part, one gets a cotradiction.


\begin{proof}[Proof of Lemma~\ref{lem60a26232}]
	Notice that in both~\eqref{eq60e3079d} and~\eqref{eq60e307a5}, the two identities are equivalent because $\E$ has finite dimension.
	
	Let us now prove the second identity in~\eqref{eq60e3079d}.  
	Fix $u\in L^\infty([0,1];\E)$. Up to choosing a representative, we can assume $u$ is defined on the whole $[0,1]$.
	Notice that, by definition of dual norm,
	\[
	\|u\|_{N^*,1}^*
	= \sup\{ \langle w|u \rangle : w\in L^1([0,1];\E^*),\ \|w\|_{N^*,1} \le 1 \} ,
	\]
	where $\langle w|u\rangle := \int_0^1 \langle w(t)|u(t)\rangle \dd t $.
	Moreover, for every $w\in L^1([0,1];\mathbb E^*)$,
	we have that 
	\[
	\int_0^1 \langle w(t)|u(t)\rangle \dd t 
	\leq \int_0^1 N_t(u(t))N_t^*(w(t))\dd t
	\leq \esssup_{t\in [0,1]}N_t(u(t)) \int_0^1 N_t^*(w(t))\dd t
	\]
	Hence we get $\|u\|_{N^*,1}^* \le \|u\|_{N,\infty}$.
%
	
	In order to show that $\|u\|_{N^*,1}^* \ge \|u\|_{N,\infty}$,
	it is sufficient to prove that, for every $\varepsilon>0$, there exists $w_\varepsilon\in L^1([0,1];\E^*)$ such that $\|w_\varepsilon \|_{N^*,1}\le 1+\varepsilon$ and $\langle w_\varepsilon|u \rangle \ge \|u\|_{N,\infty} - 2\varepsilon$, and taking $\varepsilon\to 0$.
	So, for $0<\varepsilon<1/2$, define
	\[
	B_\varepsilon := \left\{ s\in [0,1] : N_s(u(s)) \geq \|u\|_{N,\infty} - \varepsilon \right\},
	\]
	and notice that $\mathscr{L}^1(B_\varepsilon)>0$. 
	Let $\{w_j\}_{j\in\N}$ be a dense subset of $\E^*$ and define
	\[
	A_{\varepsilon,j} := \left\{ t\in B_\varepsilon : N_t^*(w_j) \le 1+\varepsilon\text{ and }\langle w_j|u(t) \rangle \ge N_t(u(t))-\varepsilon \right\},
	\]
	and notice that $B_\varepsilon= \bigcup_{j\in\N} A_{\varepsilon,j}$. 
	Finally, define $w_\varepsilon:[0,1]\to\E^*$ by 
	\[
	w_\varepsilon (t) := \frac{1}{\mathscr{L}^1(B_\varepsilon)}\sum_{j=1}^\infty \chi_{A_{\varepsilon,j}\setminus\bigcup_{i=1}^{j-1}A_{\varepsilon,i}}(t) \cdot w_j .
	\]
	From how $w_\varepsilon$ is defined, we obtain that $N_t^*(w_\varepsilon(t))\leq (1+\varepsilon)/(\mathscr{L}^1(B_\varepsilon))$ for  every $t\in B_\varepsilon$ and $w_\varepsilon(t)=0$ for  every $t\in [0,1]\setminus B_\varepsilon$. 
	Hence $\|w_\varepsilon \|_{N^*,1} = \int_0^1 N_t^*(w_\varepsilon(t))\dd t \leq 1+\varepsilon$. 
	We conclude that
	\[
	\int_0^1 \langle w_\varepsilon(t)|u(t)\rangle \dd t
	\ge \int_{B_\varepsilon} \frac{ N_t(u(t))-\varepsilon }{ \mathscr{L}^1(B_\varepsilon) } \dd t
	\geq \|u\|_{N,\infty} - 2\varepsilon.
	\]

	Next, we prove the second identity in~\eqref{eq60e307a5}:
	Fix $u\in L^1([0,1];\E)$ and notice that
	\[
	\|u\|_{N^*,\infty}^* = 
	\sup\left\{ \langle w|u \rangle : w\in L^\infty([0,1];\E^*),\ \|w\|_{N^*,\infty} \le 1 \right\} .
	\]
	Since 
	\begin{equation}\label{ANESTIMATE}
		\langle w|u \rangle \le \|w\|_{N^*,\infty} \cdot \|u\|_{N,1},
	\end{equation}
	then 
	$\|u\|_{N,1} \ge \|u\|_{N^*,\infty}^*$.
	For the opposite inequality, it is sufficient to prove that,
	for every $\varepsilon>0$ there exists $w_\varepsilon\in L^\infty([0,1];\E^*)$ with $\|w_\varepsilon\|_{N^*,\infty} \le 1+\varepsilon$ and 
	$\langle w_\varepsilon|u  \rangle \ge \|u\|_{N,1} - \varepsilon$, and then taking $\varepsilon\to 0$.
	Indeed, let $\{w_j\}_{j\in\N}$ be a dense subset of $\E^*$ and, 
	for $\varepsilon>0$
	and each $j\in\N$, define the sets
	\[
	A_{\varepsilon,j} := \left\{ t\in[0,1] : N_t^*(w_j) \le 1+\varepsilon\text{ and }\langle w_j|u(t) \rangle \ge N_t(u(t))-\varepsilon \right\} ,
	\]
	and the function $w_\varepsilon \in L^\infty([0,1];\E^*)$,
	\[
	w_\varepsilon (t) = \sum_{j=1}^\infty \chi_{A_{\varepsilon,j}\setminus\bigcup_{i=1}^{j-1}A_{\varepsilon,i}}(t) \cdot w_j .
	\]
	Then we have $[0,1] = \bigcup_{j\in\N} A_{\varepsilon,j}$,
	$\|w_\varepsilon\|_{N^*,\infty}\le 1+\varepsilon$ and
	\[
	\int_0^1 \langle w_\varepsilon(t)|u(t) \rangle \dd t
	\ge \int_0^1 N_t(u(t)) \dd t - \varepsilon .
	\]
	
	Both suprema in~\eqref{eq60e32e58} and~\eqref{eq60e32e5f} are a direct consequence of the previous identities.
	The fact that the supremum in~\eqref{eq60e32e58} is attained, is an application of Remark~\ref{rem60e32e39}, since we know that the dual space of
	$(L^1([0,1];\E) , \|\cdot\|_{N,1})$ is $(L^\infty([0,1];\E^*) , \|\cdot\|_{N^*,1})$,
	thanks to the first identity in~\eqref{eq60e3079d}.

	To prove the last statement of the lemma,
	let $w$ be an argument of the maximum in~\eqref{eq60e32e58}.
	First, we claim that $N_t^*(w(t)) = 1$ for almost every $t\in[0,1]$.
	If this were not the case, then there would be a set $A\subset[0,1]$ with positive measure such that $N_t^*(w(t)) \le \lambda < 1$, for some $\lambda>0$ and all $t\in A$.
	Define $\tilde w(t) := \chi_{[0,1]\setminus A}(t) w(t) + \chi_A(t) w(t)/\lambda$.
	Then $\|\tilde w\|_{N^*,\infty}\le 1$ and $\langle \tilde w(t)| u(t) \rangle > \langle w(t)|u(t) \rangle$ for $t\in A$, in contradiction to the maximality of $w$.
	
	Finally on the one hand from~\eqref{ANESTIMATE} we consequently have
	$\langle w(t)|u(t) \rangle \le N_t(u(t))$ for almost every $t\in[0,1]$. On the other hand, by maximality of $w$ we have
	$$\int_0^1 \langle w(t)|u(t) \rangle \dd t = \int_0^1 N_t(u(t))\dd t.$$
	Consequently $\langle w(t)|u(t) \rangle = N_t(u(t))$ for almost every $t\in[0,1]$.
\end{proof}

\subsection{Energy and length}\label{sub:DefinitionEnergyLength}

As in Section~\ref{sec607ed0ed}, we are in the setting where $M$ is a smooth manifold and $\E$ is a finite-dimensional Banach space.
Let $\mathcal{N}$ be the space of all continuously varying norms on $M\times\E$; that is $N\in\mathcal{N}$ if $N$ is a continuous function $N:M\times\E\to [0,+\infty)$ that is a norm on fibers. We endow $\cal{N}$ with the topology of uniform convergence on compact sets. We now define the energy and the lengths associated to controls, i.e., elements of $L^\infty([0,1];\mathbb E)$, and we study some of their basic properties.

\begin{definition}[Energy and length]\label{def:EnergyAndLength}
	Let $(o,\f,u)\in\DomEnd$ and $\gamma(t) := \End^\f_o(tu)$ for every $t\in[0,1]$. We define the \emph{energy}
	\[
	\Enel(o,\f,u,N) := \esssup_{t\in[0,1]}  N(\gamma(t),u(t)),
	\]
	and the \emph{length}
	\[
	\Length(o,\f,u,N) := \int_0^1 N(\gamma(t),u(t)) \dd t .
	\]
\end{definition}

For the next result, we recall that $\DomEnd$ is the domain of the End-point map.
\begin{proposition}[Semi-continuity of energy and length]\label{prop600bf91a}
	Both functions $\Enel$ and $\Length$ from $\DomEnd\times\mathcal{N}$ to $\R$ are lower-semicontinuous.
	In other words, if $(o_n,\f_n,u_n,N_n) \to (o_\infty,\f_\infty,u_\infty,N_\infty)$ in $\DomEnd\times \cal{N}$, then
	\begin{equation}\label{eq600bfe25}
	\begin{aligned}
	\Enel(o_\infty,\f_\infty,u_\infty,N_\infty) &\le
	\liminf_{n\to\infty} \Enel(o_n,\f_n,u_n,N_n) , \\
	\Length(o_\infty,\f_\infty,u_\infty,N_\infty) &\le
	\liminf_{n\to\infty} \Length(o_n,\f_n,u_n,N_n) .
	\end{aligned}
	\end{equation}
\end{proposition}
\begin{proof}
	Let $(o_n,\f_n,u_n,N_n)\in\DomEnd\times\cal N$ be a sequence converging to $$(o_\infty,\f_\infty,u_\infty,N_\infty)\in\DomEnd\times\cal N.$$
	Let us show that $u\mapsto \esssup_{t\in[0,1]}  N_\infty(\gamma_\infty(t),u(t))$ is lower-semicontinuous with respect to the weak* convergence. Notice that the previous one is a Banach norm on $L^\infty([0,1];\mathbb E)$. In particular, it is the dual norm of the norm defined on $L^1([0,1];\mathbb E^*)$ via 
	$$
	v\to \int_0^1 N_\infty^*(\gamma_\infty(t),v(t))\dd t,
	$$
    see Lemma~\ref{lem60a26232}. 
	Since the weak* convergence on $L^{\infty}([0,1];\mathbb E)$ does not depend on the choice of biLipschitz equivalent Banach norms on the predual $L^{1}([0,1];\mathbb E^*)$, and since every dual norm is weakly* lower-semicontinuous, we get the sought claim.
	
	Let us now prove the first inequality in~\eqref{eq600bfe25}.
	Notice that, for every $n\in\mathbb N$ and $t\in [0,1]$ we have
	\begin{equation}\label{eq600bed95}
		N_n(\gamma_n(t),u_n(t))
		\ge N_\infty(\gamma_\infty(t),u_n(t)) - \left| N_n(\gamma_n(t),u_n(t))- N_\infty(\gamma_\infty(t),u_n(t)) \right|,
	\end{equation}
	where, for $j\in\mathbb N\cup\{\infty\}$, and $t\in [0,1]$, $\gamma_j(t):=\End_{o_j}^{\f_j}[tu_j]$.
	As a consequence of Proposition~\ref{prop600c1758} we have that $\gamma_n(t)\to \gamma_\infty(t)$ uniformly on $[0,1]$. Hence, the functions $(t,v)\mapsto N_n(\gamma_n(t),v)$ converge uniformly on compact subsets of $[0,1]\times\E$ to $(t,v)\mapsto N_\infty(\gamma_\infty(t),v)$,
	and since the $u_n$ are uniformly bounded,
	the second term in the lower bound~\eqref{eq600bed95} goes to 0 as $n\to\infty$, uniformly in $t$.
	
	The first inequality in~\eqref{eq600bfe25} thus follows from the following computation
	\begin{align*}
	\liminf_{n\to\infty} &\Enel(o_n,\f_n,u_n,N_n) \\
	&\ge \liminf_{n\to\infty} \esssup_{t\in[0,1]} (N_\infty(\gamma_\infty(t),u_n(t))+ \\
	&\phantom{+}- \left| N_n(\gamma_n(t),u_n(t))- N_\infty(\gamma_\infty(t),u_n(t)) \right|) \\
	&= \liminf_{n\to\infty}\esssup_{t\in[0,1]}  N_\infty(\gamma_\infty(t),u_n(t)) \\
	&\ge \Enel(o_\infty,\f_\infty,u_\infty,N_\infty) .
	\end{align*}
	Similarly, to prove that $\Length$ is lower-semicontinuous, we need, in addition to~\eqref{eq600bed95}, to prove that $u\mapsto\int_0^1 N_\infty(\gamma_\infty(s),u(s)) \dd s$ is lower-semicontinuous on $L^\infty([0,1];\E)$
	(endowed with the weak* topology).

	Let $\bb S^*(t) = \{v\in\E^*:N_\infty^*(\gamma_\infty(t),v) = 1 \}$, so that, for almost every $t\in [0,1]$, we have 
	$$
	N_\infty(\gamma_\infty(t),u_\infty(t)) = \sup_{v\in\bb S^*(t)} \langle v|u_\infty(t) \rangle. 
	$$
	Moreover, if we set $\bb S^* = \{v\in L^\infty([0,1];\E^*): v(t)\in\bb S^*(t)\,\text{for almost every $t\in[0,1]$}\}$,
	then, by Lemma~\ref{lem60a26232}, for every $u\in L^1([0,1];\E)$ there is $v\in\bb S^*$ such that $\langle v(t)|u(t) \rangle = N_\infty(\gamma_\infty(t),u(t))$ for a.e.~$t\in[0,1]$.
	Therefore, for every $u\in L^1([0,1];\E)$,
	\[
	\int_0^1 N_\infty(\gamma_\infty(t),u(t)) \dd t 
	= \max_{v\in\bb S^*}\int_0^1 \langle v(t)|u(t) \rangle \dd t .
	\]
	We can finally conclude, by using the fact that since $u_n,u_\infty\in L^\infty([0,1];\mathbb E)$ we have a fortiori that $u_n,u_\infty\in L^1([0,1];\mathbb E)$, with the following chain of inequalities
	\begin{align*}
	\int_0^1 N_\infty(\gamma_\infty(t),u_\infty(t)) \dd t
	&= \max_{v\in\bb S^*} \int_0^1 \langle v(t)|u_\infty(t) \rangle \dd t \\
	&= \max_{v\in\bb S^*} \liminf_{n\to\infty} \int_0^1 \langle v(t)|u_n(t) \rangle \dd t \\
	&\le \max_{v\in\bb S^*} \liminf_{n\to\infty} \max_{w \in\bb S^*} \int_0^1 \langle w(t)|u_n(t) \rangle \dd t \\
	&= \liminf_{n\to\infty} \max_{w \in\bb S^*} \int_0^1 \langle w(t)|u_n(t) \rangle \dd t \\
	&= \liminf_{n\to\infty}  \int_0^1 N_\infty(\gamma_\infty(t),u_n(t)) \dd t.
	\end{align*}
\end{proof}

With an abuse of language, 
we say that a curve $\gamma$ has \textit{constant $N$-speed or simply constant speed}
if $\gamma=\gamma_{(o,\f,u)}$ with $t\mapsto N(\gamma(t),u(t))$ almost everywhere
constant for $t\in [0,1]$.
That a curve $\gamma$ has constant speed, really depends on all the data $(o,\f,u,N)$.

\begin{lemma}[Constant speed reparametrization]\label{lem604ba336}
	If $(o,\f,u,N)\in\DomEnd\times\cal N$, then there exists $v\in L^\infty([0,1],\E)$ such that $(o,\f,v)\in\DomEnd$,
	$\gamma_{(o,\f,u)}$ is a reparametrization of $\gamma_{(o,\f,v)} =: \gamma$,
	and
	\begin{equation}
	N(\gamma(t),v(t)) = \Enel(o,\f,v,N) = \ell(o,\f,v,N) =  \ell(o,\f,u,N),
	\end{equation}
	for almost every $t\in[0,1]$.
	In particular, $\gamma$ has constant speed and $\Enel(o,\f,v,N) \le \Enel(o,\f,u,N)$. 
\end{lemma}
\begin{proof}
	Let us assume that $\ell := \ell(o,\f,u,N) \neq 0 $, otherwise the result is trivial. Define
	\[
	\psi(t) := \frac1\ell \int_0^t N(\eta(s), u(s)) \dd s ,
	\]
	where $\eta := \gamma_{(o,\f,u)}$.
	Notice that $\psi:[0,1]\to[0,1]$ is a Lipschitz function with $\psi(0)=0$, $\psi(1)=1$ and $\psi'\ge0$.
	Since $\psi$ is Lipschitz, the image set $E:=\psi(\{t:\psi'(t)=0\text{ or $\psi'(t)$ does not exist}\})$ has measure zero by the area formula.
	Notice that for every $s\in [0,1]\setminus E$ there exists a unique $t\in [0,1]$ such that $\psi(t)=s$, and thus the formula 
	\begin{equation}\label{eqn:CHANGECURVE}
	v(s)= \frac{u(t)}{\psi'(t)},
	\end{equation}
	defines a measurable function $v:[0,1]\to \mathbb E$. Moreover
	\begin{equation}\label{eqn:CHANGECURVE2}
	N(\eta(t),v(s)) 
	= \frac{ N(\eta(t), u(t)) }{ \psi'(t) } = \ell ,
	\end{equation}
	and then, since $N(\eta(\cdot),\cdot)$ is a continuous family of norms, we get that $v\in L^\infty([0,1];\E)$. Furthermore, notice that from~\eqref{eqn:CHANGECURVE} it follows $\gamma_{(o,\f,u)}(t) = \gamma_{(o,\f,v)}(\psi(t))$.
	Therefore, for every $s\in [0,1]\setminus E$, we have
	$N(\gamma(s),v(s))= \ell$.
	Finally, since $N(\gamma(\cdot),v(\cdot))$ is constant, then it is equal to both the energy and the length of the curve $\gamma$.
\end{proof}

\subsection{Carnot-Carathéodory distances}\label{sub:CCdistance}
In this section we define the Carnot-Carathéodory distance, \textit{CC distance} for brevity, associated to $(\f,N)$ and we investigate some of its properties. We use the notation $\LipS$ from Section~\ref{sub:LipschitzVectorFields}, and $\mathcal{N}$ from Section~\ref{sub:DefinitionEnergyLength}.
 
\begin{definition}[CC distance]\label{def:CCdistance}
	Given $\f\in\LipS(\E^*\otimes TM)$ and $N\in\cal N$, we define
	\emph{the Carnot-Carathéodory distance}, or \emph{CC distance},
	\begin{equation}\label{eqn:DistanceAsInfEnergy}
	d_{(\f,N)}(p,q) := \inf\{ \Enel(p,\f,u,N) : u\in L^{\infty}([0,1];\E), \End^\f_p(u) = q \} .
	\end{equation}
\end{definition}
It is immediate from Lemma~\ref{lem604ba336} that 
\begin{equation}\label{eq600c1a8d}
	d_{(\f,N)}(p,q) = \inf\{ \Length(p,\f,u,N) :  u\in L^{\infty}([0,1];\E), \End^\f_p(u) = q \} .
\end{equation}
We will denote the open $d_{(\f,N)}$-ball of radius $r$ and center $p$ by $B_{(\f,N)}(p,r)$.
Moreover, if $M'\subset M$ is open, 
we denote by $d_{(\f,N)}|_{M'}$ the distance defined by restricting $(\f,N)$ to $M'$. 
In particular, the infimum defining $d_{(\f,N)}|_{M'}$ is taken over the family of controls whose integral curves lie in $M'$.
Thus, $d_{(\f,N)}|_{M'}(p,q) \ge d_{(\f,N)}(p,q)$ for every $p,q\in M'$.

\begin{lemma}\label{lem6050a055} 
	Let $\vertiii{\cdot}$ be a norm on $\E$,
	and define on $L^\infty([0,1];\E)$ the Banach norm
	\[
	\vertiii{u}_\infty = \esssup_{t\in[0,1]} \vertiii{u(t)} .
	\]
	Let $\rho$ be a complete Riemannian metric on $M$, which induces a norm $|\cdot|_\rho$ on $TM$.
	
	Let $\scr K\subset\LipS(\E^*\otimes TM)$, $K\subset M$, and $\cal K\subset\cal N$ be compact sets, and $R>0$.
	Then there exists $L\ge1$ such that 
	\begin{equation}\label{eq60abc2df}
	\begin{aligned}
	1/L \vertiii{v} \le N(p,v) &\le L \vertiii{v}, \\
	 |\f(p)[v]|_\rho &\le L \vertiii{v},
	\end{aligned}
	\qquad\forall p\in \bar B_\rho(K,R),\ \f\in\scr K,\ N\in\cal K, v\in\E .
	\end{equation}
	
	Moreover, the following hold for every $\f\in\scr K$, $p\in K$, $u\in L^\infty([0,1];\E)$ and $N\in\cal K$:
	\begin{enumerate}[leftmargin=*,label=(\alph*)]
	\item\label{itemlem6050a055_a}
	if $\vertiii{u}_\infty < \frac{R}{L}$,
	then $(p,\f,u)\in\DomEnd$ and
	$\gamma_{(p,\f,u)}(t)\in B_\rho(K,R)$ for every $t\in [0,1]$;
	\item\label{itemlem6050a055_b}
	if $(p,\f,u)\in\DomEnd$ and $\Length(p,\f,u,N) < R/L^2$,
	then
	$\gamma_{(p,\f,u)}(t)\in B_\rho(K,R)$ for every $t\in [0,1]$;
	\item\label{itemlem6050a055_c}
	if $(p,\f,u)\in\DomEnd$ and 
	$\gamma_{(p,\f,u)}(t)\in\bar B_\rho(K,R)$ for every $t\in [0,1]$,
	then 
	\begin{equation}\label{eq6050a6cf}
	\rho(p,\End^\f_p(u)) 
	\le L \vertiii{u}_\infty
	\le L^2 \Enel(p,\f,u,N)
	\le L^3 \vertiii{u}_\infty ;
	\end{equation}
	\item\label{itemlem6050a055_d}
	if $q\in M$ is such that
	$d_{(\f,N)}(p,q) < \frac{R}{L^2}$,
	then
	\begin{equation}\label{eq60a7d938}
	\frac1{L^2} \rho(p,q) \le  d_{(\f,N)}(p,q) =  d_{(\f,N)}|_{B_\rho(K,R)}(p,q) .
	\end{equation}
\end{enumerate}
\end{lemma}
\begin{proof}
	The existence of $L$ that satisfies~\eqref{eq60abc2df} is a consequence of the compactness of 
	$ \bar B_\rho(K,R)\times\scr K\times\cal K$, and the continuity of $\f$ and $N$.
	
	Let $p\in K$, $\f\in\scr K$, $u\in L^\infty([0,1];\E)$ and 
	$\gamma:[0,t_0]\to M$ be a solution of 
	\begin{equation}\label{eq6050a34a}
	\begin{cases}
	\gamma'(t) = \f(\gamma(t))[u(t)] , \\
	\gamma(0) = p ,
	\end{cases}
	\end{equation}
	for some $t_0>0$.
	Then, for every $a,b\in [0,t_0]$ with $\gamma([a,b])\subset B_\rho(K,R)$, 
	we have
	\begin{equation}\label{eq6050a7c1}
	\int_a^b |\gamma'(s)|_\rho \dd s
	\le L \int_a^b \vertiii{u(s)} \dd s
	\le L \vertiii{u}_\infty |b-a|.
	\end{equation}
	It follows that, if $\vertiii{u}_\infty < \frac{R}{L}$, then $\gamma$ can be extended, as solution to~\eqref{eq6050a34a}, to a curve $\gamma:[0,1]\to B_\rho(K,R)$,
	i.e., part~\ref{itemlem6050a055_a} holds.
	
	Similarly, we infer part~\ref{itemlem6050a055_b} from the estimate
	\begin{equation}\label{eq60ac8ca8}
	\int_a^b |\gamma'(s)|_\rho \dd s
	\le L \int_a^b \vertiii{u(s)} \dd s
	\le L^2 \int_a^b N(\gamma(s),u(s)) \dd s
	\le L^2 \Length(p,\f,u,N),
	\end{equation}
    that holds whenever $\gamma([a,b])\subseteq B_\rho(K,R)$.
	
	The first inequality in~\eqref{eq6050a6cf} follows from~\eqref{eq6050a7c1}. 
	The other two estimates in~\eqref{eq6050a6cf} are obtained from the first line of~\eqref{eq60abc2df}, 
	and also part~\ref{itemlem6050a055_c} is proven.
	
	For part~\ref{itemlem6050a055_d}, we see from~\eqref{eq600c1a8d} and item (b) that, if $d_{(\f,N)}(p,q) < R/L^2$, then the infimum in~\eqref{eq600c1a8d} can be taken on curves laying in $\bar B_\rho(K,R)$.
	This shows the equality in~\eqref{eq60a7d938};
	the first inequality in~\eqref{eq60a7d938} is then a direct consequence of~\eqref{eqn:DistanceAsInfEnergy} and item (c).
\end{proof}

\begin{lemma}\label{lem60e6e893}
	Let $\f\in\LipS(\E^*\otimes TM)$ an $N\in\cal N$.
	Denote by $\tau_M$ the manifold topology of $M$ and by $\tau_{(\f,N)}$ the topology of $(M,d_{(\f,N)})$.
	Then $\tau_M \subset \tau_{(\f,N)}$.
\end{lemma}
\begin{proof}
	Let $U\in\tau_M$ and $\hat p\in U$.
	We need to show that there is $r>0$ such that $B_{(\f,N)}(\hat p,r)\subset U$.
	Fix a complete Riemannian metric $\rho$ on $M$ and let $L\ge1$ be a constant that satisfies~\eqref{eq60abc2df} for $R=1$, $K=\{\hat p\}$.
	By part~\ref{itemlem6050a055_d} of Lemma~\ref{lem6050a055}, 
	if $d_{(\f,N)}(\hat p,q) < 1/L^2$, then $\rho(\hat p,q) \le L^2 d_{(\f,N)}(\hat p,q)$.
	Let $\varepsilon>0$ be such that $B_\rho(\hat p,\varepsilon)\subset U$ 
	and fix $0<r<\min\{1/L^2,\varepsilon/L^2\}$.
	If $q\in B_{(\f,N)}(\hat p,r)$, 
	then $d_{(\f,N)}(\hat p,q) < 1/L^2$
	and thus $\rho(\hat p,q) \le L^2 d_{(\f,N)}(\hat p,q) < \varepsilon$.
	Therefore,  $B_{(\f,N)}(\hat p,r)\subset B_\rho(\hat p,\varepsilon) \subset U$.
\end{proof}

In the setting of the above Lemma~\ref{lem60e6e893}, the two topologies may not be equal. 
As an example, consider a structure $\f$ defined by a constant line field on $\R^2$:
the integral lines of such a structure are open sets in $\tau_{(\f,N)}$ but not in the standard topology of $\mathbb R^2$.
We will later show that the two topologies do agree under an essentially non-holomic condition on $\f$, see Theorem~\ref{thm60929a0e}.

\begin{proposition}[Local existence of geodesics]\label{prop600c157a}
	Let $\scr K\subset\LipS(\E^*\otimes TM)$, $K\subset M$ and $\cal K\subset\cal N$ be compact sets.
	Then there is a constant $C>0$ so that, for  every $\f\in\scr K$, every $N\in\cal K$, every $p\in K$, and every $q\in M$ with $d_{(\f,N)}(p,q)\leq C$, there exists $u\in L^\infty([0,1],\E)$
	of constant $N$-speed
	such that $\End^\f_p(u)=q$ and 
	\[
	\Enel(p,\f,u,N) =  \Length(p,\f,u,N) = d_{(\f,N)}(p,q) .
	\]
	In particular, the curve $\gamma_{(p,\f,u)}:[0,1]\to (M,d_{(\f,N)})$ is a homothetic embedding, i.e., a $d_{(\f,N)}$-length minimizing curve.
	
	In the setting of Lemma~\ref{lem6050a055}, 
	we can take $C=\frac{R}{L^2}$, for $R,L>0$ satisfying~\eqref{eq60abc2df}.
\end{proposition}
\begin{proof}
	In the setting of Lemma~\ref{lem6050a055}, fix $R>0$ and the corresponding $L>0$ and set $C:=\frac{R}{L^2}$.
	Let $\f\in\scr K$, $N\in\cal K$, $p\in K$ and $q\in M$ with $d_\f(p,q) \leq C$.
	By Lemma~\ref{lem6050a055}.\ref{itemlem6050a055_d},
	there exists a sequence $u_n\in L^\infty([0,1],\E)$ such that $\End^\f_p(u_n)=q$, $\lim_{n\to\infty} \Enel(p,\f,u_n,N) = d_{(\f,N)}(p,q)$,
	and $\gamma_{(p,\f,u_n)}(t)\in B_\rho(K,R)$ for all $t\in[0,1]$ and all $n$.
	
	From Lemma~\ref{lem6050a055}.\ref{itemlem6050a055_c},
	$\vertiii{u_n}_\infty$ is uniformly bounded in $n$.
	Therefore, up to passing to a subsequence, $u_n$ weakly* converge to some $u_\infty\in L^\infty([0,1],\E)$, with 
	\[
	\vertiii{u_\infty}_\infty
	\le \liminf_{n\to\infty}\vertiii{u_n}_\infty 
	\le \liminf_{n\to\infty} L \Enel(p,\f,u_n,N)
	= Ld_{(\f,N)}(p,q) 
	<\frac{R}{L}.
	\]
	By Lemma~\ref{lem6050a055}.\ref{itemlem6050a055_a}, we thus have $(p,\f,u_\infty)\in\DomEnd$. Finally,
	From Theorem~\ref{thm6092f873} we get $\End^\f_p(u_\infty)=q$, 
	while from Proposition~\ref{prop600bf91a} we get 
	$$
	\Enel(p,\f,u_\infty,N) \le \lim_{n\to\infty} \Enel(p,\f,u_n,N) = d_{(\f,N)}(p,q),
	$$
	and thus $d_{(\f,N)}(p,q)=\Enel(p,\f,u_\infty,N)$.
	
	We claim that $N(\gamma(t),u_\infty(t))$ is constant for almost every $t\in [0,1]$. Indeed, if not, we have that $\ell(p,\f,u_\infty,N)<J(p,\f,u_\infty,N)$. Moreover, since we have~\eqref{eqn:DistanceAsInfEnergy},~\eqref{eq600c1a8d}, and the trivial inequality $\ell\leq \Enel$, we get that on controls that realize the distance one has $\ell=\Enel$, which gives the sought contradiction.
	Therefore, from the minimality, we get $N(\gamma(t),u_\infty(t)) = d_{(\f,N)}(p,q)$ for almost every $t\in[0,1]$.
	Notice that, for curves of constant speed, energy and length are equal.
	
	Finally, the claim that
	$\gamma=\gamma_{(p,\f,u)}:[0,1]\to (M,d_{(\f,N)})$ is a homothetic embedding,
	i.e., that for every $s,t\in [0,1]$
	\[
	d_{(\f,N)}\left( \gamma(s),\gamma(t) \right) 
	= d_{(\f,N)}(\gamma(0),\gamma(1)) \cdot |t-s|,
	\]
	is a direct consequence of the minimality of $u$ and the fact that $t\mapsto N(\gamma(t),u(t)) = d_{(\f,N)}(\gamma(0),\gamma(1))$ for a.e.~$t\in[0,1]$.
\end{proof}

\begin{proposition}\label{prop60a7d838}
	Let $\f\in\LipS(\E^*\otimes TM)$ and $N\in\cal N$.
	Suppose that there exist a complete Riemannian metric $\rho$, a norm $\vertiii{\cdot}$ on $\E$ 
	and $L>0$ so that~\eqref{eq60abc2df} holds for every $R>0$ and every compact set $K\subseteq M$.
	Then $(M,d_{(\f,N)})$ is a complete and geodesic metric space, whose closed bounded sets are compact.
	However, $d_{(\f,N)}$ may take value $\infty$ and be not continuous with respect to the manifold topology.
\end{proposition}
\newcommand{\spt}{\mathtt{spt}}
\begin{proof}
	By Lemma~\ref{lem6050a055}.\ref{itemlem6050a055_d},
	the bound~\eqref{eq60a7d938} holds for every $p,q\in M$.
	Therefore, a $d_{(\f,N)}$-closed set is also $\rho$-closed,
	and a $d_{(\f,N)}$-bounded set is also $\rho$-bounded.
	In particular, a $d_{(\f,N)}$-closed and $d_{(\f,N)}$-bounded set is compact.
	Hence closed bounded sets in $(M,d_{(\f,N)})$ are compact, and thus $d_{(\f,N)}$-Cauchy sequences converge.
	Proposition~\ref{prop600c157a} readily implies that $(M,d_{(\f,N)})$ is a geodesic space.
\end{proof}


\begin{remark}
	We know that for every $R>0$ there are $L_R^{(1)}>0$ and $L_R^{(2)}>0$ so that the two conditions in~\eqref{eq60abc2df} hold, with two independent constants.
	One can modify Proposition~\ref{prop60a7d838} by requiring that the growth of $L_R^{(1)}$ and $L_R^{(2)}$ as $R\to\infty$ are slow enough, although not bounded.
	However, we don't need such a finer analysis.
\end{remark}
	
\subsection{CC distances and sub-Finsler lengths of curves}	\label{sec:CClength}
We use the notation $\LipS$ from Section~\ref{sub:LipschitzVectorFields}, and $\mathcal{N}$ from Section~\ref{sub:DefinitionEnergyLength}. In this section we prove that the distance $d_{(\f,N)}$ is obtained as the infimum of the length of curves, where the length element is the natural sub-Finsler structure on $TM$ associated to $(\f,N)$, see~\eqref{eqn:NormFn}.

For $\f\in\LipS(\E^*\otimes TM)$ and $N\in\cal N$, define 
\begin{equation}\label{eqn:NormFn}
	|v|_{(\f,N)} := \inf\{N(p,u):u\in\E,\ \f(p,u)=v\} ,
\end{equation}
for $p\in M$ and $v\in T_pM$.
Notice that $|v|_{(\f,N)}<\infty$ if and only if $v\in \f(p,\E)$.


\begin{lemma}\label{lem6166f1cf}
	Let $\f\in\LipS(\E^*\otimes TM)$ and $N\in\cal N$.
	If $\gamma:[0,1]\to (M,d_{(\f,N)})$ is Lipschitz, 
	then $\gamma:[0,1]\to M$ is absolutely continuous and there exists $u\in L^\infty([0,1];\E)$ such that $\gamma'(t)=\f(\gamma(t),u(t))$ for a.e.~$t\in[0,1]$.
\end{lemma}
\begin{proof}
	By Lemma~\ref{lem60e6e893}, $\gamma:[0,1]\to M$ is continuous, hence $\gamma([0,1])$ is compact.
	We fix a complete Riemannian metric $\rho$ on $M$ and apply~\eqref{eq60a7d938} with $K=\gamma([0,1])$, and $R=1$, obtaining that 
	the curve $\gamma:[0,1]\to(M,\rho)$ is locally Lipschitz and thus an absolutely continuous curve in the manifold $M$.
	
	Without losing in generality, we can assume that $\gamma$ is 1-Lipschitz, that is, for every $t,s\in[0,1]$,
	\begin{equation}\label{eq61671dc1}
	d_{(\f,N)}(\gamma(s),\gamma(t)) \le |t-s| .
	\end{equation}
	By Proposition~\ref{prop600c157a} (with $\scr K=\{\f\}$, $K=\gamma([0,1])$ and $\cal K=\{N\}$),
	for every $n\in\N$ large enough and $0\le j\le 2^n-1$ integer, there is a control $u^{(n)}_j\in L^\infty([0,1];\E)$ such that the corresponding integral curve $\gamma^{(n)}_j:[0,1]\to M$ starting from $\gamma^{(n)}_j(0) = \gamma(\frac{j}{2^n})$ is a geodesic parametrized with constant speed,
	with end point $\gamma^{(n)}_j(1) = \gamma(\frac{j+1}{2^n})$,
	and, by also exploiting~\eqref{eq61671dc1}, 
	\[
	\esssup_{t\in [0,1]} N(\gamma^{(n)}_j(t), u^{(n)}_j(t)) 
	= d_{(\f,N)}\left(\gamma\left(\frac{j}{2^n}\right),\gamma\left(\frac{j+1}{2^n}\right)\right)
	\le 1/2^n .
	\]

	Thus, from Lemma~\ref{lem6050a055}.\ref{itemlem6050a055_b} and~\ref{itemlem6050a055_c},
	we obtain that there exists $L$ such that, for $n$ large enough, 
	$\vertiii{u^{(n)}_j}_\infty\le L/2^n$ (here $\vertiii{\cdot}_\infty$ is  a reference norm as in Lemma~\ref{lem6050a055}).
	Define 
	\[
	u^{(n)}(t) := 2^n \sum_{j=0}^{2^n-1} u^{(n)}_j(2^n(t-j/2^n)),
	\]
	where we mean that $u^{(n)}_j\equiv 0$ outside $[0,1]$ for every $n\in\mathbb N$ and $0\leq j\leq 2^{n}-1$. Hence $u^{(n)}\in L^\infty([0,1];\E)$, $(\gamma(0),\f,u^{(n)})\in\DomEnd$,  and $\vertiii{u^{(n)}}_\infty\le L$.
	Since the controls $u^{(n)}$ are uniformly bounded in $n$, there is a subsequence $\{u^{(n_k)}\}_k$ that weakly* converges to some $u^{(\infty)}$.

	We claim that $u^{(\infty)}$ is a control for $\gamma$.
	Let $\gamma_n$ be the integral curve of $u^{(n)}$ with starting point $\gamma_n(0)=\gamma(0)$.	
	Notice that $\gamma_n$ is the reparametrization on $[0,1]$ of the concatenation of the $\gamma^{(n)}_j$ for $j$ from $0$ to $2^n-1$.
	In particular, $\gamma_n(\frac{j}{2^n}) = \gamma(\frac{j}{2^n})$ for all $n$ and $j$.
	
	For every $\hat n$ large enough and every $\hat j<2^{\hat n}$, there exist $\varepsilon_{\hat n,\hat j}>0$ for which the control $u^{(\infty)}$ can be integrated on the interval $\left[ \frac{\hat j}{2^{\hat n}} , 
	\frac{\hat j}{2^{\hat n}}+\varepsilon_{\hat n,\hat j} \right]$. Namely, there exists a curve
	\[
	\eta_{\hat n,\hat j}:\left[ \frac{\hat j}{2^{\hat n}} , 
		\frac{\hat j}{2^{\hat n}}+\varepsilon_{\hat n,\hat j} \right] \to M
	\]
	that has control $u^{(\infty)}$ and starting point $\eta_{\hat n,\hat j}(\frac{\hat j}{2^{\hat n}}) = \gamma(\frac{\hat j}{2^{\hat n}})$.
	By Proposition~\ref{prop600c1758} (and by taking an affine reparametrization),
	the restrictions
	\[
	\gamma_{n_k}|_{\left[ \frac{\hat j}{2^{\hat n}} , 
		\frac{\hat j}{2^{\hat n}}+\epsilon_{\hat n,\hat j} \right] }
	\]
	uniformly converge to $\eta_{\hat n,\hat j}$ on compact subsets. By continuity of $\eta_{\hat n,\hat j}$, and by the previous convergence, it follows that $\eta_{\hat n,\hat j}(t)=\gamma(t)$ for all $t$ in the respective domains, and thus $\gamma$ is an integral curve of $u^{(\infty)}$.
%
%
%
%
%
%
%
%
\end{proof}

In the rest of this section we will not need the following
Lemma~\ref{lem60d9b340}, since it will be enough to use Lemma~\ref{lem6166f1cf}.
We decided to keep this result here for an independent interest because the proof is different and because it says something more precise, that is, that $\gamma$ is tangent to the image of $\f$ exactly at all points of differentiability of $\gamma$.

\begin{lemma}\label{lem60d9b340}
	Let $\f\in\LipS(\E^*\otimes TM)$ and $N\in\cal N$.
	If $\gamma:[0,1]\to (M,d_{(\f,N)})$ is Lipschitz,
	then $\gamma$ is absolutely continuous and $\gamma'(t)\in \f(\gamma(t),\E)$ for  every $t\in[0,1]$ where $\gamma$ is differentiable.
\end{lemma}
\begin{proof}
	Let $L$ be the Lipschitz constant of $\gamma$.
	By Lemma~\ref{lem60e6e893}, $\gamma:[0,1]\to M$ is continuous, hence $\gamma([0,1])$ is compact.
	Hence, by~\eqref{eq60a7d938} applied with $K=\gamma([0,1])$, and $R=1$, for a complete Riemannian metric $\rho$, the curve $\gamma:[0,1]\to(M,\rho)$ is locally Lipschitz and thus an absolutely continuous curve in the manifold $M$.
	Let $t_0$ be a point of differentiability for $\gamma$, 
	define $Z:=\gamma'(t_0)$
	and, arguing by contradiction, suppose that $Z\notin\f(\gamma(t_0),\E)$.
	
	We choose coordinates $(x_1,\dots,x_n)$ so that $\gamma(t_0)=0$, $Z=\partial_n$, and $$\f(\gamma(t_0),\E)\subset\Span\{\de_1,\dots,\de_{n-1}\}.$$
	For $v\in\R^n$, we write $v_n$ for the last coordinate of $v$. 
	
	Using the fact that $\f\in\LipS(\E^*\otimes TM)$, that $N$ is continuous, and that $\f(0,v)_n=0$ for every $v\in\E$, we deduce that
	there are $r>0$ and $C>0$ such that, for every $p\in\R^n$ with $|p|\le r$ 
	and every $v\in\E$,
	\[
	|\f(p,v)_n| \le C |p| N(p,v) .
	\]
	
	By definition of $d_{(\f,N)}$, for every $t$ there is $u^t\in L^\infty([0,1];\E)$ such that $\End_0^\f(u^t) = \gamma(t)$ and 
	$\Enel(0,\f,u^t,N) \le 2 d_{(\f,N)}(0,\gamma(t))$.
	Let $\sigma^t:[0,1]\to M$ be the curves with control $u^t$ and starting point $0$.
	
	We claim that $\sigma^t$ uniformly converge to the constant curve $\sigma^{t_0}\equiv0$ as $t\to t_0$.
	Indeed, since $\gamma$ is Lipschitz, then,  as $t\to t_0$, 
	$d_{(\f,N)}(0,\gamma(t))\to 0$ and thus $\Enel(0,\f,u^t,N)\to 0$.
	By Lemma~\ref{lem6050a055}.\ref{itemlem6050a055_d} (with $K=\{0\}$ in there)
	for $|t-t_0|$ small enough, we have $\sigma^t([0,1])\subset B_{\rho}(0,r)$.
	Therefore, $u^t\to0$ in $L^\infty([0,1];\E)$ as $t\to t_0$.
	By Proposition~\ref{prop600c1758},
	the curves $\sigma^t$ uniformly converge to the constant curve $\sigma^{t_0}\equiv0$ as $t\to t_0$.
	
	Since $\gamma'(t_0)=\de_n$ then $\gamma_n(t) \ge \frac{|t-t_0|}{2}$ for $|t-t_0|$ small enough.
	So, we can make the following estimates for $|t-t_0|$ small enough
	\begin{align*}
	\frac{ |t-t_0| }{ 2 } 
	&\le \gamma_n(t)
	= \sigma^t_n(1) \\
	&\le \int_0^1 |\dot\sigma^t_n(s) | \dd s 
	= \int_0^1 | \f(\sigma^t(s), u^t(s))_n | \dd s \\
	&\le C \int_0^1 |\sigma^t(s)| N(\sigma^t(s),u^t(s)) \dd s \\
	&\le C \Enel(0,\f,u^t,N) \int_0^1  |\sigma^t(s)|  \dd s \\
	&\le 2 C  d_{(\f,N)}(0,\gamma(t)) \int_0^1  |\sigma^t(s)|  \dd s \\
	&\le 2 C  L |t-t_0| \int_0^1  |\sigma^t(s)|  \dd s
	\end{align*}
	Hence, $\int_0^1  |\sigma^t(s)|  \dd s \ge \frac1{4CL} >0$ for $t$ close to $t_0$, which is in contradiction with the uniform convergence of $\sigma^t$ to $0$.
\end{proof}

\begin{lemma}[Borel right inverse]\label{lem61696783}
	Let $(X,d)$ be a locally compact, complete, separable metric space, $Y$ a topological space and $f:X\to Y$ a surjective continuous function.
	Then there is a Borel function $g:Y\to X$ such that $f(g(y))=y$ for all $y\in Y$.
\end{lemma}
\begin{proof}
	Let $\zeta:\N\to X$ be a function such that $\zeta(\N)$ is dense in $X$
	and define $Z_n = \zeta(\{0,1,\dots,n\})$.
	For every $n\in\N$, we define a function $g_n:Y\to Z_n\subset X$ 
	iteratively as follows:
	$g_0(y) = \zeta(0)$ for every $y\in Y$; 
	if $g_n$ is given, define
	$g_{n+1}(y) := g_n(y)$, unless there exists $x\in Z_{n+1}\cap B(g_n(y), 2d(g_{n}(y),f^{-1}(y)))$
	such that $d(x,f^{-1}(y)) < \frac14 d(g_{n}(y),f^{-1}(y))$,
	in which case we set $g_{n+1}(y):=x$.
	Notice that, since $Z_n$ is finite, $g_n$ is well defined and Borel.
	
	By construction, we have that for every $y\in Y$ the following holds
	\begin{equation}\label{Nonsonome}
	g_k(y)\in B(g_n(y), 4d(g_{n}(y),f^{-1}(y))) \qquad\forall k>n ,
	\end{equation}
	and, since $\bigcup_{n\in\N}Z_n$ is dense in $X$, we have $\lim_{n\to\infty} d(g_{n}(y),f^{-1}(y)) = 0$. 
	Therefore, taking into account~\eqref{Nonsonome},
	every $y\in Y$ gives a Cauchy sequence $\{g_n(y)\}_n$ with $\lim_{n\to\infty} d(g_{n}(y),f^{-1}(y)) = 0$.
	Since $X$ is complete, we can define
	\[
	g(y) := \lim_{n\to\infty} g_n(y) .
	\]
	As $g$ is the pointwise limit of a sequence of Borel functions, $g$ is also Borel.
	Moreover, $d(g(y),f^{-1}(y))=0$, i.e., there is a sequence $x_j\in f^{-1}(y)$ converging to $g(y)$.
	Hence, by the continuity of $f$, we have $f(g(y)) = \lim_{j\to\infty} f(x_i) = y$.
\end{proof}

\begin{lemma}\label{lem60e55343}
	Let $\E_1$ and $\E_2$ be finite-dimensional Banach spaces,
	$N:[0,1]\times\E_1\to[0,+\infty)$ a continuous function that is a norm on $\E_1$ for each $t\in[0,1]$
	and $\pi:[0,1]\times\E_1\to\E_2$ a continuous function that is a linear map $\E_1\to\E_2$ for each $t\in[0,1]$.
	If $v:[0,1]\to\E_2$ is a measurable (resp., Borel) function such that $v(t)\in\pi(t,\E_1)$ for almost every $t\in[0,1]$,
	then there exists a measurable (resp., Borel) function $u:[0,1]\to \E_1$ such that,
	for almost every $t\in[0,1]$, 
	$\pi(t,u(t))=v(t)$ 
	and $N(t,u(t)) = \inf\{N(t,w):\pi(t,w)=v(t)\}$.
\end{lemma}
\begin{proof}
	Let 
	\[
	Z := \{(t,\pi(t,w)):t\in[0,1],\ w\in\E_1\} \subset [0,1]\times \E_2, 
	\]
	and $S:Z\to[0,\infty)$ defined as $S(t,v) := \min\{N(t,w):\pi(t,w)=v\}$.
	Notice that $S$ is semi-continuous, in particular Borel.
	Indeed, if $(t_k,v_k)\in Z$ converge to $(t_\infty,v_\infty)$, then there are $w_k\in\E_1$ with $\pi(t_k,w_k)=v_k$ and $N(t_k,w_k) = S(t_k,v_k)$.
	If $S(t_k,v_k)$ is uniformly bounded in $k$, then, up to passing to a subsequence, $w_k\to w_\infty$ with $\pi(t_\infty,w_\infty) = \lim_{k\to\infty} \pi(t_k,w_k) = v_\infty$ 
	and $S(t_\infty,v_\infty) \le N(t_\infty, w_\infty) = \lim_{k\to\infty} N(t_k,w_k) = \lim_{k\to\infty} S(t_k,v_k)$.
	So, 
	\[
	S(t_\infty,v_\infty) \le \liminf_{(t,v)\to (t_\infty,v_\infty)} S(t,v)  .
	\]
	
	Define $f:[0,1]\times\E_1 \to [0,1]\times \E_2\times [0,+\infty)$ as $f(t,w) = (t,\pi(t,w),N(t,w))$.
	We apply Lemma~\ref{lem61696783} with $X=[0,1]\times\E_1$ and $Y=f(X)$, obtaining a Borel function $g:Y\to X$ with $f(g(y))=y$ for all $y\in Y$.
	Notice that $g(t,v,\lambda)=(t,w)$ with $\pi(t,w)=v$ and $N(t,w)=\lambda$.
	
	Now, let $v:[0,1]\to\E_2$ be a Borel function with $v(t)\in\pi(t,\E_1)$ for almost every $t\in[0,1]$.
	Define $\bar v(t) := (t,v(t),S(t,v(t)))$, which is a Borel function $[0,1]\to Y$.
	Then $\bar u(t):=g(\bar v(t))$ is also a Borel function of the form $\bar u(t) = (t,u(t))$, with $u:[0,1]\to\E_1$ Borel such that
	$\pi(t,u(t)) = v(t)$ and $N(t,u(t)) = S(t,v(t))$.	
	
%
	
	If $v$ is only measurable, then there is a Borel function $v'$ that is equal to $v$ almost everywhere, and so we can apply the proposition from the Borel setting.
\end{proof}

\begin{lemma}\label{lem60e5552f}
	Let $\f\in\LipS(\E^*\otimes TM)$, $N\in\cal N$ and $\gamma:[0,1]\to M$.
	The following statements are equivalent
	\begin{enumerate}[label=(\roman*)]
	\item
	$\gamma:[0,1]\to (M,d_{(\f,N)})$ is Lipschitz;
	\item
	$\gamma$ is absolutely continuous and $\esssup_{t\in[0,1]} |\gamma'(t)|_{(\f,N)} < \infty$;
	\item\label{lem60e5552f_c}
	there is $u\in L^\infty([0,1];\E)$ such that $\gamma = \gamma_{(\gamma(0),\f,u)}$, and  $|\gamma'(t)|_{(\f,N)}=N(\gamma(t),u(t))$ for almost every $t\in [0,1]$;
	\item\label{lem60e5552f_d} there is $u\in L^\infty([0,1];\E)$ such that $\gamma = \gamma_{(\gamma(0),\f,u)}$.
	\end{enumerate}
\end{lemma}
\begin{proof}
	The implication $(i)\THEN(iv)$ is proven in Lemma~\ref{lem6166f1cf}.
	The implication $(iv)\THEN(i)$ is a direct consequence of the definition of $d_{(\f,N)}$.
	
	The implication $(iv)\Rightarrow(ii)$ comes from the definition of $|\cdot|_{(\f,N)}$.
	The implication $(ii)\THEN(iii)$ is a consequence of Lemma~\ref{lem60e55343}.
	The implication $(iii)\THEN(iv)$ is trivial.
%
\end{proof}

Notice that the statement~\ref{lem60e5552f_d} in Lemma~\ref{lem60e5552f} does not depend on $N$: that a curve is Lipschitz does not depend on the particular norm we choose.
We can thus say that a curve $\gamma:[a,b]\to M$ is \emph{$\f$-Lipschitz} if, up to an affine reparametrization, statement~\ref{lem60e5552f_d} in Lemma~\ref{lem60e5552f} holds.

We can define the length of $\f$-Lipschitz curves in three ways, which we will show being equivalent.

\begin{proposition}\label{prop60e6e5a4}
	Let $\f\in\LipS(\E^*\otimes TM)$ and $N\in\cal N$.
	For every $\f$-Lipschitz curve $\gamma:[a,b]\to M$ the following three quantities are equal:
	\begin{align*}
		L_1(\gamma) &= \sup\left\{ \sum_{j=1}^k d_{(\f,N)}(\gamma(t_j),\gamma(t_{j+1}))
			: a=t_1\le t_2\le \dots\le t_k\le t_{k+1}=b \right\} , \\
		L_2(\gamma) &= \int_a^b |\gamma'(t)|_{(\f,N)} \dd t , \\
		L_3(\gamma) &= \inf\left\{ \Length(\f,\gamma(0),u,N): 
			u\in L^\infty([0,1];\E)\text{ s.t.~}
			[t\mapsto\gamma(a+t(b-a))] = \gamma_{(\gamma(0),\f,u)}\right\} .
	\end{align*}
	Moreover, the infimum in the definition of $L_3(\gamma)$ is a minimum.
\end{proposition}
	\begin{proof}
		Notice that all three quantities $L_j(\gamma)$ are invariant under affine reparametrizations of $\gamma$, so we may assume $a=0$, and $b=1$.
		
		The quantities $L_2(\gamma)$ and $L_3(\gamma)$ are equal by Lemma~\ref{lem60e5552f},
		which also shows that the infimum in the definition of $L_3(\gamma)$ is a minimum.
		
			Let us now show that $L_1(\gamma)=L_2(\gamma)$. Let us first show that $L_1(\gamma)\leq L_2(\gamma)$. 
			
			Since $\gamma:[0,1]\to (M,d_{(\f,N)})$ is Lipschitz, we have that its metric derivate
			\[
			|\gamma'(t)|:=\lim_{\varepsilon \to 0}\frac{d_{(\f,N)}(\gamma(t+\varepsilon),\gamma(t))}{|\varepsilon|},
			\]
			exists for almost every $t\in (0,1)$, and $L_1(\gamma)=\int_0^1 |\gamma'(t)|\dd t$, see \cite[Theorem 4.1.1]{MR2039660}. By Lemma~\ref{lem60e5552f}(iii) we have that there is $u\in L^\infty([0,1];\E)$ such that $\gamma=\gamma_{(\gamma(0),\f,u)}$, and $|\gamma'(t)|_{(\f,N)}=N(\gamma(t),u(t))$ for almost every $t\in [0,1]$. Let us now fix $t_0\in (0,1)$ a point at which the metric derivative $|\gamma'(t_0)|$ exists and $t\mapsto |\gamma'(t)|_{(\f,N)}$ has a Lebesgue point at $t_0$. Hence, by definition of $d_{(\f,N)}$, we have that for every $\varepsilon>0$ small enough
			\[
			\frac{d_{(\f,N)}(\gamma(t_0+\varepsilon),\gamma(t_0))}{\varepsilon}\leq \frac{1}{\varepsilon}\int_{t_0}^{t_0+\varepsilon}N(\gamma(s),u(s))\dd s=\frac{1}{\varepsilon}\int_{t_0}^{t_0+\varepsilon}|\gamma'(s)|_{(\f,N)}\dd s.
			\]
			Hence, taking $\varepsilon\to 0$ in the previous inequality, we get that 
			\[
			|\gamma'(t_0)|\leq |\gamma'(t_0)|_{(\f,N)}.
			\]
			Thus we conclude that $|\gamma'(t)|\leq |\gamma'(t)|_{(\f,N)}$ for almost every $t\in [0,1]$. Integrating the previous inequality between $0$ and $1$ we thus get $L_1(\gamma)\leq L_2(\gamma)$, which is the sought claim. 
			
			Let us now show $L_1(\gamma)\geq L_2(\gamma)$.  Let us argue similarly as in the last part of the proof of Lemma~\ref{lem6166f1cf}. Let us assume without loss of generality that $\gamma:[0,1]\to (M,d_{(\f,N)})$ is $1$-Lipschitz.
			By Proposition~\ref{prop600c157a},
			for every $n\in\N$ large enough and $0\le j\le 2^n-1$ integer, there is a control $u^{(n)}_j\in L^\infty([0,1];\E)$ such that the corresponding integral curve $\gamma^{(n)}_j:[0,1]\to M$ starting from $\gamma^{(n)}_j(0) = \gamma(\frac{j}{2^n})$ is a geodesic parametrized with constant speed (see Lemma~\ref{lem604ba336}),
			with end point $\gamma^{(n)}_j(1) = \gamma(\frac{j+1}{2^n})$,
			and, by the fact that $\gamma$ is $1$-Lipschitz, the following holds for almost every $t\in [0,1]$,
			\[
			N(\gamma^{(n)}_j(t), u^{(n)}_j(t)) 
			= d_{(\f,N)}\left(\gamma\left(\frac{j}{2^n}\right),\gamma\left(\frac{j+1}{2^n}\right)\right)
			\le 1/2^n .
			\]
			
			Thus, from Lemma~\ref{lem6050a055}.\ref{itemlem6050a055_b} and~\ref{itemlem6050a055_c},
			we obtain that there exists $L$ such that, for $n$ large enough, 
			$\vertiii{u^{(n)}_j}_\infty\le L/2^n$ (here $\vertiii{\cdot}_\infty$ is  a reference norm as in Lemma~\ref{lem6050a055}).
			Define 
			\[
			u^{(n)}(t) := 2^n \sum_{j=0}^{2^n-1} u^{(n)}_j(2^n(t-j/2^n)),
			\]
			where we mean that $u^{(n)}_j\equiv 0$ outside $[0,1]$ for every $n\in\mathbb N$ and $0\leq j\leq 2^{n}-1$.
			Hence $u^{(n)}\in L^\infty([0,1];\E)$, $(\gamma(0),\f,u^{(n)})\in\DomEnd$,  and $\vertiii{u^{(n)}}_\infty\le L$. Notice that $u^{(n)}$ is a control associated to the curve $\gamma^{(n)}$ which is the reparametrization on $[0,1]$ of the concatenation of all the curves $\gamma^{(n)}_j$ for $0\leq j\leq 2^n-1$.
			Since the controls $u^{(n)}$ are uniformly bounded in $n$, there is a subsequence $\{u^{(n_k)}\}_k$ that weakly* converges to some $u^{(\infty)}$. In the proof of Lemma~\ref{lem6166f1cf} we showed that $u^{(\infty)}$ is a control for $\gamma$. Let us rename $n_k=n$.
			
			Let us now fix $\varepsilon>0$. Since $u^{(n)}$ weakly* converges to $u^{(\infty)}$, by Proposition~\ref{prop600bf91a} we get that 
			\[
			\ell(\gamma(0),\f,u^{(\infty)},N)\leq \liminf_{n\to +\infty}\ell(\gamma^{(n)}(0),\f,u^{(n)},N).
			\]
			Hence, from the previous inequality, we have, up to subsequences and for $n$ large enough, 
			\[
			\int_0^1 N(\gamma(t),u^{(\infty)}(t))\dd t\leq \int_0^1 N(\gamma^{(n)}(t),u^{(n)}(t))\dd t + \varepsilon.
			\]
			Hence, using that $u^{(\infty)}$ is a control for $\gamma$, the previous inequality, the fact that $\gamma^{(n)}_j$ are geodesics parametrized with constant speed, we have, for $n$ large enough,
			\[
			\begin{split}
				\int_0^1 |\gamma'(t)|_{(\f,N)}\dd t &\leq \int_0^1 N(\gamma(t),u^{(\infty)}(t))\dd \leq \int_0^1 N(\gamma^{(n)}(t),u^{(n)}(t))\dd t+\varepsilon \\ &=\sum_{j=0}^{2^n-1}\int_0^1 N(\gamma^{(n)}_j(t),u^{(n)}_j(t))\dd t+\varepsilon \\
				&=\sum_{j=0}^{2^n-1}d_{(\f,N)}\left(\gamma\left(\frac{j}{2^n}\right),\gamma\left(\frac{j+1}{2^n}\right)\right) +\varepsilon \\
				&\leq L_1(\gamma)+\varepsilon,
			\end{split}
			\]
			where in the last inequality we have used the definition of the length $L_1(\gamma)$. Taking $\varepsilon\to 0$ in the previous inequality we get $L_1(\gamma)\geq L_2(\gamma)$, which is the sought inequality. 
			
			Thus we finally get that $L_1(\gamma)=L_2(\gamma)$, and the proof is concluded.
	\end{proof}

\begin{corollary}\label{cor:MetricDerivative}
	Let $\f\in\LipS(\E^*\otimes TM)$ and $N\in\cal N$.
	The metric derivative of a Lipschitz curve $\gamma:[a,b]\to (M,d_{(\f,N)})$ 
	(cfr.~\cite{MR2039660}) is $t\mapsto|\gamma'(t)|_{(\f,N)}$.
\end{corollary}
\begin{proof}
	It is a direct consequence of Proposition~\ref{prop60e6e5a4}, according to which $L_1=L_2$, and \cite[Theorem 4.1.1]{MR2039660}.
\end{proof}

\begin{corollary}\label{Cor:ApproccioAstrattoUgualeApproccioConcreto}
	If $\f\in\LipS(\E^*\otimes TM)$ and $N\in\cal N$, then, for every $p,q\in M$,
	\[
	d_{(\f,N)}(p,q) = \inf\left\{
			\int_a^b |\gamma'(t)|_{(\f,N)} \dd t :  
			\begin{array}{l}
			\gamma:[a,b]\to M\text{ absolutely continuous,} \\
			\text{with }\gamma(a)=p,\ \gamma(b)=q 
			\end{array}
		\right\}
	\]
\end{corollary}
\begin{proof}
	It is a direct consequence of Proposition~\ref{prop60e6e5a4}, according to which $L_1=L_2$, and \cite[Proposition 2.4.1.]{MR1835418}.
\end{proof}

\subsection{Limits of CC distances}\label{sec:LimitCC}
We use the notation $\LipS$ from Section~\ref{sub:LipschitzVectorFields}, and $\mathcal{N}$ from Section~\ref{sub:DefinitionEnergyLength}. In this section we investigate what happens when one takes the limit of Carnot-Carathéodory distances associated to $(\f,N)$. We prove a relaxation property of the limit distance, and finally the main theorem of this section, i.e., Theorem~\ref{thm60929a0e}.

\begin{proposition}\label{prop604f912b}
	If $\f_n\to\f_\infty$ in $\LipS(\E^*\otimes TM)$, $N_n\to N_\infty$ in $\cal N$, and $p,q\in M$, then 
	\begin{equation}\label{eq600bfbff}
	d_{(\f_\infty,N_\infty)}(p,q) \geq \inf\left\{ \liminf_{n\to\infty} d_{(\f_n,N_n)}(p_n,q_n) : p_n\to p,\ q_n\to q \right\} .
	\end{equation}

	Moreover, if $K\subset M$ is a compact set, there exists $C>0$ such that whenever $p\in K$ and $q\in M$ with $d_{(\f_\infty,N_\infty)}(p,q)\leq C$, one has that 
	\begin{equation}\label{eq600bfbffNEW}
		d_{(\f_\infty,N_\infty)}(p,q) = \inf\left\{ \liminf_{n\to\infty} d_{(\f_n,N_n)}(p_n,q_n) : p_n\to p,\ q_n\to q \right\} .
	\end{equation}
	In the setting of Lemma~\ref{lem6050a055}, 
	we can take $C=\frac{R}{L^2}$, for $R,L>0$ satisfying~\eqref{eq60abc2df} 
	on the compact set $\bar B_\rho(K,1)$ in place of $K$
	and with $\scr K=\{\f_n\}_n\cup\{\f_\infty\}$, $\cal K=\{N_n\}_n\cup\{N_\infty\}$.
\end{proposition}
\begin{proof}
	Let us first prove~\eqref{eq600bfbff}.
	If $d_{(\f_\infty,N_\infty)}(p,q)=\infty$, there is nothing to prove, so we assume $d_{(\f_\infty,N_\infty)}(p,q)<\infty$.
	Fix $\varepsilon>0$. Then there is 
	$u\in L^\infty([0,1];\E)$ such that $\End^{\f_\infty}_p(u) = q$ and $\Enel(p,\f_\infty,u,N_\infty)\le d_{(\f_\infty,N_\infty)}(p,q) + \varepsilon$.
	Let $p_n\equiv p$ and $q_n=\End^{\f_n}_p(u)$ that exists for $n$ big enough thanks to Proposition~\ref{prop600c1758}(1).
	By Theorem~\ref{thm6092f873}, $q_n\to q$, and moreover, by the definition of the distance, $d_{(\f_n,N_n)}(p,q_n)\le \Enel(p,\f_n,u,N_n)$.
	
	Let $\gamma_n$ be the curve $\gamma_n(t) := \End^{\f_n}_p(tu)$.
	By Proposition~\ref{prop600c1758}, $\gamma_n\to \gamma_\infty$ uniformly.
	Hence, by possibly passing to subsequences,
	\begin{equation*}
	\begin{split}
	\lim_{n\to\infty} \Enel(p,\f_n,u,N_n) 
	&= \lim_{n\to\infty}\esssup\{N_n(\gamma_n(t),u(t)) : t\in[0,1]\} \\
	&= \esssup\{N_\infty(\gamma_\infty(t),u(t)) : t\in[0,1]\} 
	= \Enel(p,\f_\infty,u,N_\infty).
	\end{split}
	\end{equation*}
	We conclude that $\liminf_{n\to\infty} d_{(\f_n,N_n)}(p,q_n) \le \Enel(p,\f_\infty,u,N_\infty) \le d_{(\f_\infty,N_\infty)}(p,q) + \varepsilon$.
	Since $\varepsilon$ can be taken arbitrarily small, we get the sought claim.
	
	Let us now prove~\eqref{eq600bfbffNEW}. In the notation of Lemma~\ref{lem6050a055}, take $\scr K:=\{\f_n\}_{n\in\mathbb N}\cup\{\f_\infty\}$, $\cal K:=\{N_n\}_{n\in\mathbb N}\cup\{N_\infty\}$,
	and $R>0$.
	Let $L$ be the constant of Lemma~\ref{lem6050a055} associated to these choices 
	on the compact set $\bar B_\rho(K,1)$ in place of $K$,
	and
	set $C:=R/L^2$.
	
	Suppose by contradiction that~\eqref{eq600bfbffNEW} is not true. Hence, since~\eqref{eq600bfbff} always holds, there exist $p\in K$, $q\in M$ with $d_{(\f_\infty,N_\infty)}(p,q)\leq R/L^2$, $p_n\to p$, $q_n\to q$ such that, up to passing to subsequences, 
	\begin{equation}\label{eqn:SoughtContr}
	d_{(\f_\infty,N_\infty)}(p,q)>\lim_{n\to\infty}d_{(\f_n,N_n)}(p_n,q_n).
	\end{equation}
	Hence, for $n$ large enough, we have $d_{(\f_n,N_n)}(p_n,q_n)< R/L^2$, and $p_n\in \bar B_\rho(K,1)$, since $p_n\to p\in K$.
	
	Therefore, by applying Lemma~\ref{lem6050a055}(d), thanks to the choice of the compact $\bar B_\rho(K,1)$,
	we know that we can take, for $n$ large enough, $u_n\in L^{\infty}([0,1];\mathbb E)$ such that $\gamma_{(p_n,\f_n,u_n)}\subseteq \bar B_\rho(K,R+1)$ and 
	$$
	J(p_n,\f_n,u_n,N_n)\leq d_{(\f_n,N_n)}(p_n,q_n)+1/n.
	$$
	Therefore, by using~\eqref{eq6050a6cf}, we have that 
	$$
	\vertiii{u_n}_\infty\leq L(R/L^2+1/n),
	$$
	for $n$ large enough. Henceforth, up to subsequences, $u_n\to u_\infty$ weakly*. From Theorem~\ref{thm6092f873} we get $\End^{\f_\infty}_p(u_\infty)=q$, 
	while from Proposition~\ref{prop600bf91a} we get 
	$$
	d_{(\f_\infty,N_\infty)}(p,q) 
	\le \Enel(p,\f_\infty,u_\infty,N_\infty) 
	\le \liminf_{n\to\infty} \Enel(p_n,\f_n,u_n,N_n)
	\le \liminf_{n\to\infty} d_{(\f_n,N_n)}(p_n,q_n).
	$$ 
	TWe obtained a contradiction with~\eqref{eqn:SoughtContr}.	
\end{proof}


\begin{lemma}\label{lem604be4c3}
		Let  $\f_n\to \f_\infty$ in $\LipS(\E^*\otimes TM)$, $N_n\to N_\infty$ in $\mathcal{N}$, and $K\subset M$ be a compact set. 
		Assume that $\f_\infty$ is essentially non-holonomic and
		fix a complete Riemannian metric $\rho$ on~$M$. 
		
		Then, for every $\varepsilon>0$ there exist $\delta>0$ and $n_0\in\mathbb N$ such that, for all $n\geq n_0$ (also $n=\infty$)
		and all $p\in K$,
		\begin{equation*}
			B_\rho(p,\delta) \subset B_{(\f_n,N_n)}(p,\varepsilon) .
		\end{equation*}
\end{lemma}
\begin{proof}
	First, we claim that for every $\hat p\in K$
	there are $n_1\in\N$ and a neighborhood $U$ of $\hat p$ such that, 
	for all $n\ge n_1$,
	\begin{equation}\label{eq60a3c95d}
	U \subset B_{(\f_n,N_n)}(\hat p,\varepsilon/2) .
	\end{equation}
	To prove the claim, we use Lemma~\ref{lem6050a055} with $R=1$ and the compact sets $K$, $\{\f_n\}_{n\in\N}\cup\{\f_\infty\}$, and $\{N_n\}_{n\in\N}\cup\{N_\infty\}$. Let $L$ be as in Lemma~\ref{lem6050a055}.
	Define the set
	\[
	\scr U = \left\{ u\in L^\infty([0,1];\E) : \vertiii{u}_\infty \le 
		\min\left(\frac{\varepsilon}{2L} , \frac{1}{L} \right)  \right\},
	\]
	where $\vertiii{\cdot}_\infty$ is as in Lemma~\ref{lem6050a055}. By Theorem~\ref{thm60a3bb36},
	there are $n_1\in\N$ and a neighborhood $U$ of $\hat p$
	 so that 
	$U \subset \End^{\f_n}_{\hat p}(\scr U) $, for all $n>n_1$.
	By Lemma~\ref{lem6050a055}(a) and  Lemma~\ref{lem6050a055}(c), for every $n>n_1$ and every $u\in\mathscr{U}$,
	we have 
	$d_{(\f_n,N_n)}(\hat p,\End^{\f_n}_{\hat p}(u)) 
	\le \Enel(\hat p,\f_n,u,N_n) 
	\le L \vertiii{u}_\infty \le \varepsilon/2$,
	that is,
	\[
	U \subset \End^{\f_n}_{\hat p}(\scr U) \subset B_{(\f_n,N_n)}(\hat p,\varepsilon/2) .
	\]
		
	Second, by the compactness of $K$, we obtain that there are $\hat p_1,\dots,\hat p_k\in K$, open neighborhoods $U_j$ of $\hat p_j$ with $K\subset\bigcup_j U_j$,
	and $n_0\in\N$,
	such that, for all $n>n_0$ and all $j$,
	\[
	U_j \subset B_{(\f_n,N_n)}(\hat p_j,\varepsilon/2) .
	\]
	
	Next, on the one hand, there is $\delta>0$ such that for every $p\in K$ there is $j$ with $B_\rho(p,\delta) \subset U_j$.
	On the other hand, if $p\in B_{(\f,N)}(\hat p_j,\varepsilon/2)$, then $B_{(\f,N)}(\hat p_j,\varepsilon/2) \subset B_{(\f,N)}(p,\varepsilon)$. 
	
	We conclude that, if $n>n_0$, 
	then for every $p\in K$ there is $j$ such that
	\[
	B_\rho(p,\delta) \subset U_j
	\subset B_{(\f_n,N_n)}(\hat p_j,\varepsilon/2)
	\subset B_{(\f_n,N_n)}(p,\varepsilon) ,
	\]
	which concludes the proof.
\end{proof}

\begin{proposition}\label{prop604be4c9}
	Let  $\f_n\to \f_\infty$ in $\LipS(\E^*\otimes TM)$, $N_n\to N_\infty$ in $\mathcal{N}$, and $K\subset M$ be a compact set. 
	Assume that $\f_\infty$ is essentially non-holonomic
	and fix a complete Riemannian metric $\rho$ on~$M$. 
	
	Then for every $\varepsilon>0$ there are $n_0\in\N$ and $\delta>0$ 
	so that for every $n>n_0$ and $p,p',q,q'\in K$
	\begin{equation}
	\rho(p,p')+\rho(q,q') < \delta
	\quad\THEN\quad
	\left| d_{(\f_n,N_n)}(p,q) - d_{(\f_n,N_n)}(p',q') \right|
	\le \varepsilon
	\end{equation}
\end{proposition}
\begin{proof}
	We define
	\[
	\beta_k(s) := \sup\{d_{(\f_n,N_n)}(p,q) : p,q\in K,\ n\ge k,\ \rho(p,q) \le s \} .
	\]
	Lemma~\ref{lem604be4c3} immediately implies that,
	for every $\varepsilon$, there are $\delta>0$ and $n_0$ such that,
	for every $s\in[0,\delta]$,
	we have $\beta_{n_0}(s)<\varepsilon/2$.
	Now, if $p,p',q,q'\in K$, $\rho(p,p')+\rho(q,q') < \delta$ and $n\ge n_0$, then 
	\begin{align*}
	\left| d_{(\f_n,N_n)}(p,q) - d_{(\f_n,N_n)}(p',q') \right|
	&\le d_{(\f_n,N_n)}(p,p') +  d_{(\f_n,N_n)}(q,q')  \\
	&\le 2\beta_{n_0}(\delta) 
	\le \varepsilon.
	\end{align*}
\end{proof}

\begin{corollary}\label{SemiSup}
	If $\hat\f$ is essentially non-holonomic and $\hat N\in\cal N$, then, for every $p,q\in M$,
	\[
	\limsup_{(\f,N)\to(\hat\f,\hat N)} d_{(\f, N)}(p,q) \le d_{(\hat\f,\hat N)}(p,q).
	\] 
\end{corollary}
\begin{proof}
	The proof of this corollary is a slight modification of the proof of~\eqref{eq600bfbff},
	where we now have that $d_{(\f_n, N_n)}(q_n,q)\to 0$ by Proposition~\ref{prop604be4c9}.
\end{proof}

We are now ready to prove the main result of the paper, namely Theorem~\ref{thm60929a0e}.


\begin{proof}[Proof of Theorem~\ref{thm60929a0e}]
	In order to prove item~\ref{thm60929a0e_1}, we define for $p\in M$
	\[
	\scr U(p) = \{q\in M:d_{(\hat\f,\hat N)}(p,q)<\infty\} .
	\]
	Notice that, if $\scr U(p)\cap\scr U(q)\neq\emptyset$ then $\scr U(p)=\scr U(q)$.
	Therefore, $\{\scr U(p)\}_{p\in M}$ is a partition of $M$.
	Moreover, by Lemma~\ref{lem604be4c3} applied to the constant sequence $(\hat\f,\hat N)$ and the compact set $\{p\}$, 
	$p$ is always in the interior of $\scr U(p)$.
	Now, if $p'\in\scr U(p)$ then $p'$ is in the interior of $\scr U(p')=\scr U(p)$ and thus we proved that $\scr U(p)$ is open for every $p\in M$.
	Since $M$ is connected, we conclude that $\scr U(p)=M$.
	
	The proof of item~\ref{thm60929a0e_2} has two parts and uses an auxiliary complete Riemannian distance $\rho$ on $M$.
	First, if $U\subset M$ is $d_{(\hat \f,\hat N)}$-open and $p\in U$,
	then there is $\varepsilon>0$ so that $B_{(\hat \f,\hat N)}(p,\varepsilon)\subset U$.
	By Lemma~\ref{lem604be4c3} applied to the constant sequence $(\hat\f,\hat N)$ and the compact set $\{p\}$, there is $\delta>0$ such that $B_\rho(p,\delta)\subset B_{(\hat \f,\hat N)}(p,\varepsilon)$.
	Since $p$ is an arbitrary point in $U$, we have proven that $U$ is $\rho$-open. The fact that a $d_{(\hat\f,\hat N)}$-open set is also a $\rho$-open set is a direct consequence of Lemma~\ref{lem60e6e893}.
	
	
	Item~\ref{thm60929a0e_3} is proven as follows.
	We show that every $o\in M$ has a compact neighborhood $U$ so that $d_{(\f_n,N_n)}\to d_{(\hat\f,\hat N)}$ pointwise on $U\times U$.
	By Proposition~\ref{prop604be4c9}, 
	$\{d_{(\f_n,N_n)}\}_{n\in\mathbb N}\cup\{d_{(\hat\f,\hat N)}\}$ is an equicontinuous family of functions $U\times U\to\R$,
	hence pointwise convergence would imply the uniform convergence on $U\times U$.
	
	Let us show the pointwise convergence.
	Let $r>0$ so that $\bar B_{(\hat\f,\hat N)}(o,r)$ is compact. This can be done thank to the item (ii) that we previously proved.
	By Proposition~\ref{prop604f912b} there is $C>0$ so that for every 
	$p\in \bar B_{(\hat\f,\hat N)}(o,r)$ and $q\in M$ with $d_{(\hat\f,\hat N)}(p,q)\leq C$, one has~\eqref{eq600bfbffNEW}.
	We may assume $r<C/2$, that is, that~\eqref{eq600bfbffNEW} holds for every $p,q\in\bar B_{(\hat\f,\hat N)}(o,r)$.
	So, let $p,q\in\bar B_{(\hat\f,\hat N)}(o,r)$
	and $p_n\to p$ and $q_n\to q$ so that
	\[
	d_{(\hat\f,\hat N)}(p,q) = \lim_{n\to\infty} d_{(\f_n,N_n)}(p_n,q_n) .
	\]
	Then
	\begin{multline*}
	\limsup_{n\to\infty}
		\left| d_{(\hat\f,\hat N)}(p,q) - d_{(\f_n,N_n)}(p,q) \right| \\
	\le \limsup_{n\to\infty}
		\left| d_{(\hat\f,\hat N)}(p,q) - d_{(\f_n,N_n)}(p_n,q_n) \right|
		+ 
		\left| d_{(\f_n,N_n)}(p_n,q_n) - d_{(\f_n,N_n)}(p,q) \right| \\
	= \limsup_{n\to\infty}\left| d_{(\f_n,N_n)}(p_n,q_n) - d_{(\f_n,N_n)}(p,q) \right| .
	\end{multline*}
	We need to show that the previous limit is zero.
	Let $\varepsilon>0$.
	By Proposition~\ref{prop604be4c9}, there are $n_0>0$ and $\delta>0$ so that,
	if $n>n_0$, we have $\rho(p_n,p)+\rho(q_n,q)<\delta$, and then 
	$\left| d_{(\f_n,N_n)}(p_n,q_n) - d_{(\f_n,N_n)}(p,q) \right|<\varepsilon$.
	This shows that the limit is zero.
	
	Item~\ref{thm60929a0e_4} is a consequence of item~\ref{thm60929a0e_3} together with the forthcoming metric Lemma~\ref{Lemma2}.
\end{proof}

\begin{remark}[About the completeness assumption in Theorem~\ref{thm60929a0e}(iv)]\label{rem:EXAMPLE}
	In this remark we show that the assumption of the completeness of $d_{(\hat \f,\hat N)}$ in Theorem~\ref{thm60929a0e}(d) is necessary in order to have the uniform convergence $\lim_{(\f,N)\to(\hat\f,\hat N)} d_{(\f, N)} = d_{(\hat\f,\hat N)}$ on compact subsets. In the following example we show that one may not even have pointwise convergence. Thus Corollary~\ref{SemiSup} cannot be improved in general.
	
	Let us fix $M:=\mathbb R\times (-1,1)\subseteq\mathbb R^2$ and $p:=(-2;0)$, $q:=(2;0)$. Let us take, for every $n\in\mathbb N$, a smooth function $g_n:\mathbb R^2\to[1,+\infty)$ such that 
	\[
	\begin{cases}
	g_n=1 &\text{outside $[-2,2]\times(-1+1/(4n),1-1/(4n))$} \\
	g_n=10 &\text{inside $[-1,1]\times(-1+1/(2n),1-1/(2n))$}.
	\end{cases}
	\]
	Let $d_n$ be the Riemannian distance associated to the Riemannian tensor $g_n(\dd x\otimes\dd x+\dd y\otimes \dd y)$ on $M$. We can take $g_n$ such that we have $g_n\to g_\infty$ uniformly on compact subsets of $M$, where $g_\infty$ is a smooth function with $g_\infty=10$ inside $[-1,1]\times(-1,1)$. Let $d_\infty$ the Riemannian distance associated to the Riemannian tensor $g_\infty(\dd x\otimes\dd x+\dd y\otimes \dd y)$ on $M$. We have that $d_n\to d_\infty$ locally uniformly on $M$, i.e., every $p\in M$ has a neighborhoof $U$ such that $d_n\to d_\infty$ uniformly on $U\times U$. Nevertheless, $d_n$ does not converge uniformly to $d_\infty$ on compact subsets of $M$. Indeed, we have that $d_n(p,q)\leq 6$, while $d_\infty(p,q)\geq 10$. 
\end{remark}

\begin{lemma}\label{Lemma2}
	Let $\Lambda$ be endowed with a sequential topology. Let $X$ be a set. For $t\in\Lambda$, let $d_t$ be a length metric on $X$. Assume that for some $t_0\in\Lambda$ we have that $(X,d_{t_0})$ is boundedly compact. Assume that for every point $x\in X$ and every sequence $t_n\to t_0$ there exists a $d_{t_0}$-neighborhood $U$ of $x$ such that 
	\begin{equation}\label{eqn:Holds}
		\sup_{(p,q)\in \overline U\times\overline U}|d_{t_n}(p,q)-d_{t_0}(p,q)|\to 0,
	\end{equation}
as $n\to +\infty$.
	Hence for every $d_{t_0}$-compact set $K$ we have 
	\begin{equation}\label{eqn:ToProve}
		\sup_{(p,q)\in K\times K}|d_{t_n}(p,q)-d_{t_0}(p,q)|\to 0, \qquad \text{for every sequence ${t_n}\to {t_0}$}.
	\end{equation}
	Moreover,we have that, for every $x\in X$, and for every sequence ${t_n}\to t_0$,
	\begin{equation}\label{eqn:CONVergence}
		(X,d_{t_n},x)\to (X,d_{t_0},x),
	\end{equation}
	in the pointed Gromov--Hausdorff sense.
\end{lemma}
\begin{proof}
	Let us first prove~\eqref{eqn:ToProve}. Fix $K$ a $d_{t_0}$-compact set. Let $D:=\diam_{d_{t_0}}K$ and let $K':=B_{d_{t_0}}(K,D+3)$ be the closed $(D+3)$-tubular neighborhood of $K$. Since $(X,d_{t_0})$ is boundedly compact, $K'$ is compact. 
	
	Suppose by contradiction that ~\eqref{eqn:ToProve} does not hold for some sequence $t_n\to t_0$. Hence, up to passing to subsequences, we have that for some $0<\varepsilon<1$ and every $n\in\mathbb N$ the following holds
	\begin{equation}\label{eqn:SoughtEsti}
	\sup_{(p,q)\in K\times K}|d_{t_n}(p,q)-d_{t_0}(p,q)|\geq 2\varepsilon.
	\end{equation}
	
	For every $s\in K'$ there exists $U_s$ a $d_{t_0}$-neighborhood of $s$ such that~\eqref{eqn:Holds} holds for the sequence $\{t_n\}_{n\in\mathbb N}$. Since $K'$ is compact, we can extract a finite covering of $K'$ from $\{U_s\}_{s\in K'}$. Hence there exists $m\in \mathbb N$ and $s_1,\dots,s_m\in K'$ such that 
	$$
	K'\subseteq U_{s_1}\cup\dots\cup U_{s_m}.
	$$
	For the sake of simplicity we rename $U_{s_i}=:U_i$ for every $i=1,\dots, m$. Let us take $N$ big enough such that for every $n\geq N$ and every $i=1,\dots,m$ we have 
	\begin{equation}\label{eqn:EstimateEpsilon2m}
		\sup_{(p,q)\in \overline U_{i}\times\overline U_{i}}|d_{t_n}(p,q)-d_{t_0}(p,q)|\leq \frac{\varepsilon}{2m}.
	\end{equation}
	
	We now aim at showing that for every $p,q\in K$ and every $n\geq N$ we have 
	\begin{equation}\label{eqn:Esti}
		d_{t_n}(p,q)\leq d_{t_0}(p,q)+\varepsilon.
	\end{equation}
	Since $d_{t_0}$ is a length distance, given $p,q\in K$, there exists a curve $\gamma:[0,1]\to X$ such that 
	\begin{equation}\label{eqn:ChooseTheRightCurve}
		L_{d_{t_0}}(\gamma)\leq d_{t_0}(p,q)+\varepsilon/2.
	\end{equation}
	For every $\alpha\in [0,1]$ we have that 
	$$
	d_{t_0}(p,\gamma(\alpha))\leq L_{d_{t_0}}(\gamma)\leq d_{t_0}(p,q)+\varepsilon/2\leq D+1/2,
	$$
	and hence $\gamma\subseteq \mathrm{int}(K')$. We now aim at finding on $\gamma$ a finite number $i$, with $i\leq m$, of points $p=p_1,p_2,\dots,p_i=q$ such that for every $j=1,\dots,i-1$ we have that that $p_j,p_{j+1}$ are in the same $\overline U_{{k_j}}$, for some $k_j\in\{1,\dots,m\}$. We define such a sequence inductively. 
	
	First, since $p\in K$, there exists a $k_1\in\{1,\dots,m\}$ such that $p\in \overline U_{{k_1}}$.  Let us suppose that the sequence $p=p_1,\dots,p_\ell$ has been defined for some $\ell\in\mathbb N$, in such a way that 
	\begin{itemize}\label{itembuoni}
		\item[(i)] for every $j=1,\dots,\ell$, there exist $k_j\in\{1,\dots,m\}$ that are \textbf{pairwise distinct} such that
		\item[(ii)] $p_j,p_{j+1}\in \overline U_{k_j}$ for every $j=1,\dots,\ell-1$, and
		\item[(iii)] $p_\ell \in\overline U_{k_\ell}$.
	\end{itemize}
	Hence define 
	\begin{equation}\label{eqn:Mado}
		\alpha_{\ell+1}:=\max\{\alpha:\gamma(\alpha)\in \overline U_{k_\ell}\}, \quad p_{\ell+1}:=\gamma(\alpha_{\ell+1}).
	\end{equation}
	Obviously we have $p_{\ell+1}\in \overline U_{k_\ell}$. If $\alpha_{\ell+1}=1$ the process ends and $q=p_{\ell+1}\in \overline U_{k_{\ell+1}}$ with $k_{\ell+1}$ distinct from every $k_1,\dots,k_\ell$ by the inductive definition of the $\alpha$'s. If not, we now show that $p_{\ell+1}\in \overline U_{k_{\ell+1}}$ for some $k_{\ell+1}\in\{1,\dots,m\}$ different from every $k_1,\dots,k_\ell$. This is true since for every $\eta>0$ small enough we have that $\gamma(\alpha+\eta)\in K'$ and hence $\gamma(\alpha+\eta)\in U_{k_\eta}$, where $k_\eta\in\{1,\dots,m\}$. Since $k_\eta$ ranges in a finite set, there exists $k_{\ell+1}\in\{1,\dots,m\}$ such that $\gamma(\alpha+\eta_j)\in U_{k_{\ell+1}}$ for a sequence $\eta_j\to 0$. Moreover $k_{\ell+1}$ has to be different from every $k_1,\dots,k_\ell$, since it is inductively defined by means of~\eqref{eqn:Mado}. This eventually proves that, after at most $m$ steps, we end the process at $q$, since also $q\in \overline{U}_{k_i}$ for some $k_i\in \{1,\dots, m\}$. Hence the claim is shown.
	
	Hence we now want to obtain~\eqref{eqn:Esti}. Fix $p,q\in K$, $n\geq N$, and take the chain of points $p=p_1,\dots,p_i=q$ previously defined. Hence 
	\begin{equation}\label{eqn:Split}
		\begin{split}
			d_{t_n}(p,q)&\leq \sum_{\ell=1}^{i-1}d_{t_n}(p_\ell,p_{\ell+1})\leq \sum_{\ell=1}^{i-1}d_{t_0}(p_\ell,p_{\ell+1})+\varepsilon/2 \\ 
			&\leq L_{d_{t_0}}(\gamma)+\varepsilon/2\leq d_{t_0}(p,q)+\varepsilon.
		\end{split}
	\end{equation}
	where the first inequality is an application of the triangle inequality; the second inequality comes from~\eqref{eqn:EstimateEpsilon2m}, the fact that $p_\ell,p_{\ell+1}\in\overline U_{k_\ell}$ for some $k_\ell\in\{1,\dots,m\}$, and the fact that $i\leq m$; the third inequality comes from the definition of length; and the fourth is a consequence of~\eqref{eqn:ChooseTheRightCurve}.
	
	With a slight variation of the previous argument, we now aim at showing that for every $p,q\in K$ and every $n\geq N$ we have 
	\begin{equation}\label{eqn:Esti2}
		d_{t_0}(p,q)\leq d_{t_n}(p,q)+\varepsilon.
	\end{equation}
	Given $p,q\in K$ and $n\geq N$, since $d_{t_n}$ is a length distance, there exists a curve $\gamma_{t_n}:[0,1]\to X$ such that 
	\begin{equation}\label{eqn:ChooseTheRightCurve2}
		L_{d_{t_n}}(\gamma_{t_n})\leq d_{t_n}(p,q)+\varepsilon/2.
	\end{equation}
	We do not know a priori if $\gamma_{t_n}\subseteq K'$, but nevertheless we may argue as before, paying attention to one more detail. Again, we aim at finding on $\gamma_{t_n}$ a finite number $i$, with $i\leq m$, of points $p=p_1,p_2,\dots,p_i=q$ such that for every $j=1,\dots,i-1$ we have that that $p_j,p_{j+1}$ are in the same $\overline U_{{k_j}}$, for some $k_j\in\{1,\dots,m\}$. We proceed by induction. 
	
	Since $p\in K$, there exists a $k_1\in\{1,\dots,m\}$ such that $p\in \overline U_{{k_1}}$.  Let us suppose that the sequence $p=p_1,\dots,p_\ell$ has been defined for some $\ell\in\mathbb N$, in such a way that items (i), (ii), and (iii) above hold.
	Hence define 
	\begin{equation}\label{eqn:Mado2}
		\alpha_{\ell+1}:=\max\{\alpha:\gamma_{t_n}(\alpha)\in \overline U_{k_\ell}\}, \quad p_{\ell+1}:=\gamma_{t_n}(\alpha_{\ell+1}).
	\end{equation}
	Clearly $p_{\ell+1}\in \overline U_{k_\ell}$. We now first show that $p_{\ell+1}\in \mathrm{int}(K')$. Indeed 
	\begin{equation}
		\begin{split}
			d_{t_0}(p,p_{\ell+1})&\leq \sum_{k=1}^{\ell}d_{t_0}(p_k,p_{k+1}) \leq \sum_{k=1}^{\ell}d_{t_n}(p_k,p_{k+1})+\varepsilon/2 \\
			&\leq L_{d_{t_n}}(\gamma_{t_n}|_{[p,p_{\ell+1}]})+\varepsilon/2\leq L_{d_{t_n}}(\gamma_{t_n})+\varepsilon/2 \\
			&\leq d_{t_n}(p,q)+\varepsilon \leq d_{t_0}(p,q)+2\varepsilon\leq D+2,
		\end{split}
	\end{equation} 
	where the first inequality is a consequence of the triangle inequality; the second is a consequence of~\eqref{eqn:EstimateEpsilon2m} and the fact that the chain of points has cardinality not greater than $m$; the third inequality is a consequence of the definition of length; the fifth is a consequence of~\eqref{eqn:ChooseTheRightCurve2}; and the sixth is a consequence of~\eqref{eqn:Esti}. Now, arguing exactly as before, we can show that $p_{\ell+1}\in \overline U_{k_{\ell+1}}$ with $k_{\ell+1}\in\{1,\dots,m\}$ different from all $k_1,\dots,k_\ell$.
	
	Now to obtain~\eqref{eqn:Esti2} one argues exactly as before. Namely, for $p,q\in K$, and $n\geq N$ we fix a chain of points $p=p_1,\dots,p_i=q$ inductively constructed as above, and we repeat the estimate~\eqref{eqn:Split} exchanging the roles of $d_{t_n}$ and $d_{t_0}$. Hence,~\eqref{eqn:Esti} and~\eqref{eqn:Esti2} give the sought contradiction with~\eqref{eqn:SoughtEsti}, thus proving~\eqref{eqn:ToProve}.
	
	Taking into account the definition of pointed Gromov--Hausdorff convergence, see \cite[page 272]{MR1835418}, to prove~\eqref{eqn:CONVergence} it is sufficient to use~\eqref{eqn:ToProve} and that, if we fix $x\in X$, we have that, for every $t_n\to t_0$ and for every $R$, 
	$$
	\limsup_{n\to +\infty}\diam_{d_{t_0}} B_{d_{t_n}}(x,R)<+\infty.
	$$
	The previous inequality is a direct consequence of a slight variation of the second argument above. Indeed, arguing as before, one can show that for every sequence $t_n\to 0$ and every $R>0$ there exists $N$ sufficiently big such that for every $n\geq N$ we have 
	$$
	\bar B_{d_{t_n}}(x,R)\subseteq \bar B_{d_{t_0}}(x,R+1).
	$$
\end{proof}

\section{Examples}\label{sec:Examples}
In this section we discuss several examples in which we can apply our main convergence result Theorem~\ref{thm60929a0e}.  

In Section~\ref{sec60929f6f} we use Theorem~\ref{thm60929a0e} to directly prove that the asymptotic cone of the Riemannian Heisenberg group is the sub-Riemannian Heisenberg group, see Proposition~\ref{prop:AsymptoticRiemannianHeisenbergGroup}. The same reasoning can be easily generalized to arbitrary Carnot groups. 

In Section~\ref{sec6092e03a} we state and prove Mitchell's Theorem in the sub-Finsler cathegory for a continuously varying norm on the manifold, see Theorem~\ref{thm:Mitchell}. We give for granted the construction of privileged coordinates and of the nilpotent approximation, for which we refer the reader to standard and well-established references, see \cite{Bellaiche}, \cite[Section 2.1]{Jean}, \cite[Sections 10.4-10.5-10.6]{MR3971262}, or the recent \cite{MontiPigatiVittone}. Hence we exploit Theorem~\ref{thm60929a0e} to directly prove the final convergence part of Mitchell's Theorem in such a general setting.

In Section~\ref{sec:CCgroups} we use Theorem~\ref{thm60929a0e} to directly prove Theorem~\ref{thm:PerEmilio}. Namely, we prove that on a connected Lie group the CC distances associated to  bracket-generating sub-spaces and norms that converge are uniformly convergent on compact subsets. The latter result has been used in the very recent \cite{Lauret21}. 

In Section~\ref{sec:LimSub} we record a general approximation theorem for sub-Finsler distances associated to converging vector fields on a manifold. Notice that Theorem~\ref{thm6092a1a1} can be used to produce Finsler approximation of sub-Finsler manifolds. 

\subsection{Asymptotic cone of the Riemannian Heisenberg group}\label{sec60929f6f}
The first application we discuss is the well-known fact that the asymptotic cone of the Riemannian Heisenberg group is the sub-Riemannian Heisenberg group. 

Using exponential coordinates of the first kind,
we identify the first Heisenberg group $\bb H^1$ with the manifold $\R^3$ 
endowed with left-invariant frame
\[
X = \de_x - \frac{y}{2} \de_z,
\qquad
Y = \de_y + \frac{x}{2} \de_z,
\qquad
Z = \de_z .
\]
Let $\langle \cdot,\cdot \rangle$ be the left-invariant Riemannian tensor on $\bb H^1$ that makes the above frame orthonormal, 
and $d_R$ the corresponding distance.
Let $d_{sR}$ be the sub-Riemannian distance defined by $X,Y$, namely
$$ 
d_{sR}(p,q):=\inf_{\gamma(0)=p,\gamma(1)=q}\left\{\int_0^1 \sqrt{a_1(t)^2+a_2(t)^2}\dd t: \gamma'(t)=a_1(t)X_1|_{\gamma(t)}+a_2(t)X_2|_{\gamma(t)} \right\},
$$
for every $p,q\in M$, where the infimum is taken over absolutely continuous curves $\gamma$.

\begin{proposition}\label{prop:AsymptoticRiemannianHeisenbergGroup}
	The asymptotic cone of $(\bb H^1,d_R)$ is $(\bb H^1,d_{sR})$.
\end{proposition}
\begin{proof}
	Define $\E = \R^3$, $\|\cdot\|$ the Euclidean norm on $\E$ and, for $\varepsilon\in\R$, 
	$\f_\varepsilon\in\LipS(\E^*\otimes T\bb H^1)$ to be
	\[
	\f_\varepsilon(p)(v) = v_1 X|_p + v_2 Y|_p + \varepsilon v_3 Z|_p .
	\]
	Let $d_\varepsilon:= d_{(\f_\epsilon,\|\cdot\|)}$. 
	Then, , 
	we have that $d_1=d_R$ and $d_0=d_{sR}$.
	Since $\f_0$ is totally non-holonomic, see Proposition~\ref{prop:EveryBracket}, we obtain from Theorem~\ref{thm60929a0e}(iv) that
	$d_\varepsilon\to d_0$ as $\varepsilon\to0$, uniformly on compact subsets of $\bb H^1\times\bb H^1$.

	What is only left to show is that $(\bb H^1,d_\varepsilon)$ is isometric to $(\bb H^1,\varepsilon d_R)$.
	Notice that $\varepsilon d_R$ is the Riemannian distance defined by the orthonormal frame
	$(X/\varepsilon, Y/\varepsilon, Z/\varepsilon)$.
	Notice also that the map
	$\delta_\varepsilon(x,y,z) := (\varepsilon x,\varepsilon y,\varepsilon^2 z)$ 
	satisfies
	\begin{align*}
	(\delta_\varepsilon)_* X/\varepsilon
		&= X = \f_\varepsilon(\cdot)(1,0,0) ,\\
	(\delta_\varepsilon)_* Y/\varepsilon
		&= Y = \f_\varepsilon(\cdot)(0,1,0) ,\\
	(\delta_\varepsilon)_* Z/\varepsilon
		&= \varepsilon Z = \f_\varepsilon(\cdot)(0,0,1) .
	\end{align*}
	 It follows that $\delta_\varepsilon:(\bb H^1,\varepsilon d_R)\to(\bb H^1,d_\varepsilon)$ is an isometry.
\end{proof}
\begin{remark}
	Mutatis mutandis the statement of Proposition~\ref{prop:AsymptoticRiemannianHeisenbergGroup} works, with the same proof, for arbitrary Carnot groups. Actually, with some additional work, one could also recover the well-known fact that the asymptotic cone of a sub-Finsler nilpotent Lie group is a sub-Finsler Carnot group.
\end{remark}

\subsection{Tangents to sub-Finsler manifolds}\label{sec6092e03a}
In this section we discuss the celebrated Mitchell's Theorem. We take for granted the existence of a system of priviliged coordinates, for which we refer the reader, e.g., to the complete discussion in \cite[Section 2.1]{Jean}. 
We stress that, up to the authors' knowledge, this is the first time that such a theorem is stated in the generality of sub-Finsler cathegory. Standard references for the sub-Riemannian Mitchell's Theorem are \cite[Theorem 2.5]{Jean}, \cite{Mitchell, Bellaiche}, and \cite[Sections 10-4-10.5-10.6]{MR2062547}. Our aim, here, is to give a complete and detailed proof of the final convergence part of the proof.

Hereafter we follow the terminology of Jean's book \cite[Section 2.1]{Jean}. We fix a smooth manifold $M^m$ of dimension $m\in\mathbb N$, $k\in\mathbb N$, and $k$ smooth vector fields $X_1,\dots,X_k$. We say that the family $\mathcal{X}:=\{X_1,\dots,X_k\}$ is a {\em non-holonomic system} if it is bracket-generating. Notice that we do not assume that the rank of $\mathcal X$ is constant.

Let us fix $\mathbb E:=\mathbb R^k$, with canonical basis $\{e_1,\dots,e_k\}$, and let us fix $N:M\times \E\to [0,+\infty)$ a continuously varying family of norms on $\E$. Attached to the previously defined non-holonomic system $\mathcal X$ and to $N$ there is a notion of a sub-Finsler metric $|\cdot|_{\mathcal X,N}$, see~\eqref{eqn:VxN}.
In the specific case in which $N(p,\cdot)$ is the sandard Euclidean norm for every $p\in M$, the previous sub-Finsler metric is exactly the one considered in Jean's book \cite[Equation (1.4)]{Jean}. 

The sub-Finsler metric $|\cdot|_{(\mathcal{X},N)}$ gives raise to a length distance $d_{(\mathcal{X},N)}$, see~\eqref{eqn:DxN} and compare with \cite[Definition 1.3]{Jean}. Notice that, by definition, and by exploiting Lemma~\ref{lem6050a055}(d) and Theorem~\ref{thm60929a0e}(ii), we get that any two distances $d_{(\mathcal{X},N_1)}$ and $d_{(\mathcal{X},N_2)}$ are locally bi-Lipschitz equivalent.  As a consequence we stress that the notion of non-holonomic order of a smooth function/vector field at a point $p\in M$, see \cite[Section 2.1]{Jean}, can be equivalently given by using any of such distances.

We define the Lipschitz-vector-field structure on $M$ modelled by $\mathbb R^k$
\begin{equation}\label{eqn:LVFXi}
\widetilde\f_1(p)(e_i) := X_i|_{p},
\end{equation}
for every $p\in M$, and $i=1,\dots,k$. The distance $d_{(\widetilde\f_1,N)}$ as in Definition~\ref{def:CCdistance}, coincides with the length distance induced by $|\cdot|_{\mathcal X,N}$ as above, by virtue of Corollary~\ref{Cor:ApproccioAstrattoUgualeApproccioConcreto}.

\begin{proof}[Proof of Theorem~\ref{thm:Mitchell}]
Let $o\in M$ be as in the statement. There exists a neighborhood $U\subseteq M^m$ of $o$ and a neighborhood $V\subseteq \mathbb R^m$ of $0$ such that $(x_1,\dots,x_m):U\to V$ is a system of priviliged coordinates, see \cite[Definition 2.5]{Jean} and \cite[pages 22-23]{Jean}. Hereafter we will identify $U\subseteq M$ with $V\subseteq \mathbb R^m$ by means of the coordinates $(x_1,\dots,x_m)$.

Being $w_i$ the weights defined as in \cite[page 20]{Jean}, for $0<\varepsilon\leq1$ we define the diffeomorphism  $\delta_\varepsilon:\delta_\varepsilon^{-1}V\to V$ as follows
$$
\delta_\varepsilon(x_1,\dots,x_m):=(\varepsilon^{w_1}x_1,\dots,\varepsilon^{w_j}x_j,\dots,\varepsilon^{w_m}x_m).
$$

Let us define, for every $\varepsilon>0$, every $i=1,\dots,k$, and every $p\in \delta_\varepsilon^{-1} V\subseteq \mathbb R^m$, the Lipschitz-vector-fields structure
\begin{equation}\label{eqn:feps}
\f_\varepsilon(p)(e_i) := \varepsilon (D\delta_\varepsilon|_{p})^{-1} [X_i|_{\delta_\varepsilon p}].
\end{equation}
Notice that $\delta_\varepsilon^{-1}V$ invades $\mathbb R^m$ as $\varepsilon\to 0$. 

Notice, moreover, that when $\varepsilon=1$ we are defining, by means of~\eqref{eqn:feps}, the Lipschitz-vector-field structure $\widetilde\f_1$ restricted to $U$, namely $\f_1=(\widetilde\f_1)|_U$, cf.~\eqref{eqn:LVFXi}. In addition, as a consequence of Lemma~\ref{lem6050a055}(d), we can find $U'\Subset U$ sufficiently small such that 
\begin{equation}\label{eqn:IsIsometricTo}
(U',d_{(\f_1,N)}) \quad \text{is isometric to} \quad (U',d_{(\widetilde\f_1,N)}).
\end{equation}

Let $\{\hat X_1,\dots,\hat X_k\}$ be the homogeneous nilpotent approximation of $\{X_1,\dots,X_k\}$ at $o$ associated to the coordinates $(x_1,\dots,x_m)$, see \cite[Definition 2.7]{Jean}. Notice that, for every $i=1,\dots,k$, $\hat X_i$ are polynomial vector fields on $\mathbb R^m$. Let us define, for every $p\in\mathbb R^m$, and every $i=1,\dots,k$, the Lipschitz-vector-fields structure
\[
\f_0(p)(e_i) := \hat X_i|_p.
\]
Notice that the family $\{\hat X_1,\dots,\hat X_k\}$ is a bracket-generating family of vector fields on $\mathbb R^m$, cf. \cite[Lemma 2.1, (i)]{Jean}, and hence $\f_0$ is essentially non-holonomic, see Proposition~\ref{prop:EveryBracket}.

Moreover, by the definitions of homogeneous nilpotent approximation, of $\f_\varepsilon$, and of the maps $\delta_\varepsilon$ the following holds: for every $\varepsilon>0$, $\f_\varepsilon$ is a Lipschitz-vector-field structure defined on $\delta_\varepsilon^{-1}V$ modelled by $\mathbb R^k$, and 
$$
\f_\varepsilon\to \f_0, \qquad \text{as $\varepsilon\to 0$},
$$
 in the sense of Definition~\ref{defSeqTop}.

Let us now define, for every $0<\varepsilon\leq 1$, $N_\varepsilon:\delta_\varepsilon^{-1}V\times \E\to[0,+\infty)$ as
\begin{equation}\label{eqn:DefinitionNEpsilon}
N_\varepsilon(p,v):=N(\delta_\varepsilon p,v),
\end{equation}
for every $p\in \delta_\varepsilon^{-1}V$, and $ v\in \E$. We stress that $N_1\equiv N$ on $V\times \E$. 
It is readily seen that $N_\varepsilon$ converge, as $\varepsilon\to 0$, uniformly on compact subsets of $\mathbb R^m\times \E$ to $N_0:\R^m\times\E\to [0,+\infty]$ defined as
\[
N_0(p,v):=N(o,v),
\]
for every $p\in\mathbb R^m$ and $v\in \E$. 

In case $o$ is a regular point then the metric space $(\mathbb R^m,d_{(\f_0,N_0)})$ is isometric to a sub-Finsler Carnot group, cf. \cite[Lemma 2.3]{Jean}. In the general case, $(\mathbb R^m,d_{(\f_0,N_0)})$ is isometric to a quotient of a sub-Finsler Carnot group by one of its closed subgroups, cf. \cite[Theorem 2.6]{Jean}. In any of such alternatives, we have that $(\mathbb R^m,d_{(\f_0,N_0)})$ is boundedly compact since it is a locally compact, homogeneous, length space.

Hence, an application of Theorem~\ref{thm60929a0e}(iv) gives that
\begin{equation}\label{eqn:CONVERGENZA}
d_{(\f_\varepsilon,N_\varepsilon)}\to_{\varepsilon\to 0} d_{(\f_0,N_0)},
\end{equation}
uniformly on compact subsets of $\mathbb R^m\times\mathbb R^m$. Notice that the structures $\f_\varepsilon$ are not defined on the entire $\mathbb R^m$, but just on $\delta_\varepsilon^{-1}V$. Anyway, we can still apply Theorem~\ref{thm60929a0e} since for every compact set $K\subseteq\mathbb R^m\times\mathbb R^m$ there exists $\varepsilon_0$ such that $K\subseteq \delta_\varepsilon^{-1}V\times \delta_\varepsilon^{-1}V$ for every $\varepsilon<\varepsilon_0$.
Moreover, as a consequence of the last part of the proof of Lemma~\ref{Lemma2},  one also has that, for any $R>0$, there exists $\varepsilon_0$ small enough such that $\overline B_{R}^{d_{(\f_\varepsilon,N_\varepsilon)}}(0)$ are contained in a common compact set of $\R^m$ for every $\varepsilon<\varepsilon_0$; and moreover the Gromov--Hausdorff distance between $\overline B_{R}^{d_{(\f_\varepsilon,N_\varepsilon)}}(0)$ and $\overline B_{R}^{d_{(\f_0,N_0)}}(0)$ converges to $0$ as $\varepsilon\to 0$, cf.~\eqref{eqn:CONVERGENZA}.

We now claim that 
for every $\varepsilon$ and every $p,q\in \delta_\varepsilon^{-1}V$, the following equality holds
\begin{equation}\label{eqn:SoughtClaimDue}
d_{(\f_\varepsilon,N_\varepsilon)}(p,q)=\varepsilon^{-1} d_{(\f_1,N_1)}(\delta_\varepsilon p,\delta_\varepsilon q).
\end{equation}
Indeed, let us fix $\varepsilon>0$, and let us take $p,q\in \delta_\varepsilon^{-1}V$. 
For every curve $\gamma:[0,1]\to\mathbb R^m$, whose image is contained in $\delta_\varepsilon^{-1}V$, and such that  
\begin{equation}\label{CurvaGammaEpsilon}
	\begin{cases}
		\gamma'(t) &=  \f_\varepsilon(\gamma(t),u(t)) \\
		\gamma(0) &= p ,
	\end{cases}
\end{equation}
for almost every $t\in [0,1]$, we let $\gamma_\varepsilon:=\delta_\varepsilon\circ\gamma$. We notice that $\gamma_\varepsilon$ has support contained in $V$ and, from the definition of $\f_\varepsilon$, we have
\begin{equation}\label{CurvaGammaEpsilonDue}
	\begin{cases}
		\gamma_\varepsilon'(t) &=  \f_1(\gamma_\varepsilon(t),\varepsilon u(t)), \qquad \text{for almost every $t\in [0,1]$}, \\
		\gamma_\varepsilon(0) &= \delta_\varepsilon p.
	\end{cases}
\end{equation}
Since we have that $\gamma_\varepsilon\subseteq V$, by taking into account the previous computation, the definition of the norms $N_\varepsilon$, see~\eqref{eqn:DefinitionNEpsilon}, and the definition of the distance, see~\eqref{eqn:DistanceAsInfEnergy}, we finally get~\eqref{eqn:SoughtClaimDue} for every $p,q\in \delta_\varepsilon^{-1}V$.

The latter reasoning implies that for every $R>0$ there exists $\varepsilon_0$ small enough  such that for every $\varepsilon<\varepsilon_0$ one has
\begin{equation}\label{eqn:IsIsometricToDue}
	\left( \overline B_ {R}^{d_{(\f_\varepsilon,N_\varepsilon)}} (0) , d_{(\f_\varepsilon,N_\varepsilon)} \right) 
	\quad \text{is isometric to} \quad 
	\left( \overline B_{R}^{\varepsilon^{-1} d_{(\f_1,N)}}(0),\varepsilon^{-1} d_{(\f_1,N_1)} \right),
\end{equation}
and the isometry is given by $\delta_\varepsilon$.

 Thus, the latter, together with the convergence in~\eqref{eqn:CONVERGENZA}, and~\eqref{eqn:IsIsometricTo}, directly implies that the Gromov--Hausdorff tangent of $(M,d_{(\widetilde\f_1,N)})$ at $o$ is $\mathbb R^n$ equipped with the sub-Finsler distance induced by the vector fields $\hat X_1,\dots,\hat X_k$ and the norm $N(o,\cdot)$, which is what we wanted.
\end{proof}
\begin{remark}
	We stress that our convergence result Theorem~\ref{thm60929a0e} holds in the Lipschitz cathegory, provided the essentialy non-holonomicity of the limit. Hence, whenever one has some analogues of priviliged coordinates while dealing with less regular vector fields, the proof of Theorem~\ref{thm:Mitchell} is very likely to be adapted.
\end{remark}




\subsection{Left-invariant CC distances on Lie groups}\label{sec:CCgroups}

Let $\mathbb G$ be a connected Lie group, and let $\mathfrak g$ be its Lie algebra. Given a vector subspace $\mathcal{H}\subseteq \mathfrak g$ of $\mathfrak g$, and a norm $b$ on $\mathcal{H}$, we associate to $(\mathcal{H},b)$ a left-invariant sub-Finsler structure $(\mathcal{D},b)$ as in~\eqref{eqn:DistributionGroup}. Moreover, we define $d^{(\mathcal{H},b)}$ as in~\eqref{eqn:DefinitionDistanceGroup}.
%

Let us denote by $k$ the dimension of $\mathcal{H}$. Choose a basis $\{v_1,\dots,v_k\}$ of $\mathcal{H}$ and define $X_i$ to be the left-invariant extension of $v_i$, for every $i=1,\dots,k$. Let $\mathbb E:=\mathbb R^k$ with the canonical basis $\{e_1,\dots,e_k\}$. Define a Lipschitz-vector-field structure $\f$ modelled by $\mathbb R^k$ as
$$
\f_p(e_i):=X_i(p), \quad \text{for all $p\in \mathbb G$, and $i=1,\dots,k$}.
$$
Define the continuously varying norm $N:\mathbb G\times\mathbb E\to \mathbb R$ as 
$$
N(p,w):=b\left(\sum_{i=1}^k w_iv_i\right), \quad \text{for all $p\in \mathbb G$, and $w=\sum_{i=1}^k w_ie_i\in\mathbb E$}.
$$ 
By virtue of Corollary~\ref{Cor:ApproccioAstrattoUgualeApproccioConcreto} we deduce that 
\begin{equation}\label{EqualityDFNDHB}
d_{(\f,N)}(p,q)=d^{(\mathcal{H},b)}(p,q), \quad \text{for all $p,q\in\mathbb G$}.
\end{equation}

 We now give the proof of Theorem~\ref{thm:PerEmilio}.


%
%
\begin{proof}[Proof of Theorem~\ref{thm:PerEmilio}]
	Let us choose a bracket-generating basis $\{v_1,\dots,v_k\}$ of $\mathcal{H}$. By the fact that $\mathcal{H}_n\to \mathcal{H}$, we have that, for every $n\in\mathbb N$, there exists a basis $\{v_1^n,\dots,v_k^n\}$ of $\mathcal{H}_n$ such that $v_i^n\to v_i$ as $n\to +\infty$ for every $i=1,\dots,k$. Let $X_i$ be the left-invariant extension of $v_i$ for every $i=1,\dots,k$, and let $X_i^n$ be the left-invariant extension of $v_i^n$ for every $i=1,\dots,k$, and every $n\in\mathbb N$. Let us fix $\mathbb E:=\mathbb R^k$ with basis $\{e_1,\dots,e_k\}$. 
	
	For every $n\in\mathbb N$, we define the Lipschitz-vector-field structure $\f_n$ modelled by $\mathbb R^k$ as 
	$$
	(\f_n)_p(e_i):=X_i^n(p), \qquad \text{for all $p\in \mathbb G$, for all $i=1,\dots,k$,}
	$$
	and the Lipschitz-vector-field structure $\f$ modelled by $\mathbb R^k$ as
	$$
	\f_p(e_i):=X_i(p), \qquad \text{for all $p\in \mathbb G$, for all $i=1,\dots,k$.}
	$$
	By the convergence $v_i^n\to v_i$ in $\mathfrak g$, we have that $\f_n\to \f$ in the sense of Definition~\ref{defSeqTop}.
	
	 For every $n\in\mathbb N$ we define the continuously varying norm $N_n:\mathbb G\times \mathbb E$ as
	$$
	N_n(p,w):=b_n\left(\sum_{i=1}^k w_iv_i\right), \quad \text{for all $p\in\mathbb G$, and $w=\sum_{i=1}^k w_ie_i\in \mathbb E$},
	$$
	and the continuously varying norm $N$ as
	$$
	N(p,w):=b\left(\sum_{i=1}^k w_iv_i\right), \quad \text{for all $p\in\mathbb G$, and $w=\sum_{i=1}^k w_ie_i\in \mathbb E$}.
	$$
	By the fact that $b_n\to b$ uniformly on compact sets it follows that $N_n\to N$ uniformly on compact sets. 
	
	Hence we showed that $(\f_n,b_n)\to(\f,b)$ as $n\to +\infty$. Let us finally check that the remaining hypotheses of Theorem~\ref{thm60929a0e} are met. Indeed, $\f$ is essentially non-holonomic due to Proposition~\ref{prop:EveryBracket}. Moreover, $(\mathbb G,d^{(\mathcal H,b)})$ is boundedly compact because, by homogeneity and Theorem~\ref{thm60929a0e}(ii), there exists $\varepsilon>0$ such that for every $p\in\mathbb G$ the closed ball $\overline B_{d^{(\mathcal{H},b)}}(p,\varepsilon)$ is compact. Hence an application of Theorem~\ref{thm60929a0e}(iv), together with the equality~\eqref{EqualityDFNDHB}, gives the sought conclusions.
\end{proof}

\begin{remark}
	As a special case of Theorem~\ref{thm:PerEmilio}, when $b_n\equiv b$ is a norm coming from a scalar product, it follows that the Condition 3.9 conjectured in \cite{Lauret21} is always true.
\end{remark}

\subsection{Limit of sub-Finsler distances on a manifold}\label{sec:LimSub}
	In this section we prove a general convergence result that is a consequence of Theorm~\ref{thm60929a0e}.

\begin{theorem}\label{thm6092a1a1}
Let $M^m$ be a smooth connected manifold of dimension $m$. 
Let $k\in\mathbb N$. 
For every $\lambda\in[0,1)$ consider $\mathcal{X}^\lambda:=\{X_1^\lambda,\dots, X_k^\lambda\}$ a family of smooth vector-fields such that
\begin{enumerate}
	\item $X_i^\lambda$ are locally equi-Lipschitz for every $i=1,\dots,k$, and every $\lambda\in [0,1)$;
	\item $X_i^\lambda\to X_i^0$ uniformly on compact sets as $\lambda\to 0$, for every $i=1,\dots,k$.
\end{enumerate}
Let us assume that $\{X_1^0,\dots,X_k^0\}$ is a bracket-generating set of vector fields.

Let $\E:=\mathbb R^k$ with basis $\{e_1,\dots,e_k\}$ and, for every $\lambda\in [0,1)$, let $N_\lambda:M\times\E\to[0,+\infty)$ be a continuously varying norm. Assume that $N_\lambda\to N_0$ uniformly on compact sets.
For each $\lambda$, let $|\cdot|_\lambda$ the sub-Finsler metric defined by
\[
|v|_\lambda = \inf\left\{N_\lambda\left(p,\sum_{j=1}^k x_je_j\right):\ v=\sum_{j=1}^k x_j X_j^\lambda(p)\right\}
\]
for every $v\in T_pM$, and $p\in M$. Let $d_\lambda$ be the sub-Finsler distance associated to the sub-Finsler metric $|\cdot|_\lambda$ (cf.~\eqref{eqn:DxN}). Hence, $d_\lambda\to d_0$ locally uniformly on compact sets of $M\times M$, as $\lambda\to 0$. 

Moreover, if $(M,d_0)$ is a complete metric space, we have that $d_\lambda\to d_0$ uniformly on compact sets of $M\times M$, as $\lambda\to 0$, and for every $p\in M$, we have that $(M,d_\lambda,p)\to (M,d_0,p)$ in the pointed Gromov--Hausdorff topology as $\lambda\to 0$.

\end{theorem}
\begin{proof}
	For every $\lambda\in [0,1)$, let $\f_\lambda\in\LipS(\E^*\otimes TM)$ be defined as
	\[
	\f_\lambda(p)(e_i) :=  X_i^\lambda(p), \quad \text{for all $p\in M$, and for every $i=1,\dots,k$.}
	\]
	From the hypotheses we get that $(\f_\lambda,N_\lambda)\to (\f_0,N_0)$ as $\lambda\to 0$, where the convergence of the first components has to be intended in the sense of Definition~\ref{defSeqTop}, and the convergence of the second components has to be intended in the uniform sense on compact sets. Moreover, as a consequence of Corollary~\ref{Cor:ApproccioAstrattoUgualeApproccioConcreto}, we have that 
	$d_\lambda = d_{(\f_\lambda,N_\lambda)}$ for every $\lambda\in [0,1)$. Hence the results follow from Theorem~\ref{thm60929a0e}(iii) and Theorem~\ref{thm60929a0e}(iv).
\end{proof}

\appendix 
\section{Gronwall Lemma}
\begin{lemma}\label{lem1053}
 	Let $\Omega\subset\R^n$, and $X,Y:\Omega\times[0,T]\to\R^n$. Let $\|\cdot\|$ be the standard Euclidean norm on $\mathbb R^n$.
	Fix $o\in\Omega$.
		
	Suppose that there are $E,K>0$ be such that for all $p,q\in\Omega$ and all $t\in[0,T]$
	\[
	\|X(p,t)-Y(q,t)\|\le E + K \|p-q\| .
	\]
		
	Let $\gamma,\eta:[0,T]\to \Omega$ be two absolutely continuous curves such that $\gamma(0)=\eta(0)=o$,
	$\gamma'(t) = X(\gamma(t),t)$, and $\eta'(t) = Y(\eta(t),t)$ for almost every~$t\in[0,T]$.
	
	Then
	\[
	\|\gamma(t)-\eta(t) \| \le  E \frac{ e^{Kt}-1 }{K},
	\qquad\forall t\in[0,T] .
	\]
\end{lemma}
\begin{proof}
 	Define $f(t) := \|\gamma(t)-\eta(t)\|$.
	Notice that $f:[0,T]\to\R$ is absolutely continuous and $f(0)=0$.
	Moreover, for almost every $t\in[0,T]$ we have
	\begin{multline*}
		2f(t)f'(t) 
	 	= \frac{\dd}{\dd t}(f(t)^2) = \\
		= 2 \langle \gamma(t)-\eta(t),\gamma'(t)-\eta'(t) \rangle 
		\le 2\cdot \|\gamma(t)-\eta(t)\|\cdot\|\gamma'(t)-\eta'(t)\| = \\
		= 2f(t)\cdot \|\gamma'(t)-\eta'(t)\| 
		\le 2f(t)\cdot\|X(\gamma(t),t)-Y(\eta(t),t)\| \le \\
		\le 2f(t)\cdot (E + K\|\gamma(t)-\eta(t)\|) 
		\le 2f(t)\cdot(E + K f(t)).
	\end{multline*}
	So, whenever $f(t)\neq0$ we have
	\[
	f'(t) \le E + K f(t) .
	\]
	
	Let $g(t):= e^{-Kt}f(t)$. Then, whenever $g(t)\neq0$, i.e., whenever $f(t)\neq0$, we have
	\[
	g'(t) = -Ke^{-Kt}f(t) + e^{-Kt}f'(t)
	\le -Ke^{-Kt}f(t) + e^{-Kt} (E+Kf(t))
	= e^{-Kt}E.
	\]
	
	We claim
	\begin{equation}\label{eq1143}
	g(t) \le \int_0^t e^{-Ks}E \dd s,
	\end{equation}
	for almost all $t\in[0,T]$. Indeed, if $g(t)=0$, then there is nothing to show because the right-hand side is positive.
	If $g(t)>0$ instead, since $g$ is absolutely continuous, there is a maximal $\hat t< t$ such that $g(\hat t)=0$, and we have 
	\[
	g(t) = g(\hat t) + \int_{\hat t}^t g'(s) \dd s \le \int_{\hat t}^t e^{-Ks}E \dd s \le \int_0^t e^{-Ks}E \dd s ,
	\]
	and~\eqref{eq1143} is proved.
	
	Finally we obtain
	\[
	e^{-Kt}f(t) = g(t) \le \int_0^t e^{-Ks}E \dd s = \frac EK (1-e^{-Kt})
	\]
	hence the thesis.
\end{proof}

\section{Lemma about open mappings}

We recall a few facts about degree theory.
The $n$-th homology group $H_n(\bb S^n)$ of the $n$-dimensional sphere $\bb S^n$ 
is isomorphic to $\Z$.
If $\phi:\bb S^n\to\bb S^n$ is continuous, then it induces a group morphism $\phi_*:H_n(\bb S^n)\to H_n(\bb S^n)$ of the form $\phi_*(x)=ax$ for some $a\in\Z$.
This coefficient $a$ is called the \emph{degree} of $\phi$ and it is denoted by $\deg(\phi)$.

\begin{lemma}\label{lem6082ea3b}
	Let $\Omega\subset \R^m$ be an open set containing $0$.
	Let $f_\infty:\Omega\to\R^m$ be a continuous map with $f_\infty(0)=0$.
	Suppose that there is $r>0$ such that $B(0,r)\Subset\Omega$, 
	$0\notin f_\infty(\partial B(0,r))$, and the map 
	\[
	\phi:\bb S^{m-1} \to \bb S^{m-1},
	\qquad
	\phi(x) := \frac{f_\infty(rx)}{|f_\infty(rx)|}
	\]
	has nonzero degree.
	Let $f_n:\Omega\to\R^m$ be continuous functions that converge uniformly on $\Omega$ to $f_\infty$.
	
	Then there exist $\delta>0$ and $N\in\N$ such that for all $n\ge N$, we have 
	\[
	B(0,\delta)\subset f_n(\Omega) .
	\]
\end{lemma}
\begin{proof}
	Let $B:=B(0,r)\Subset \Omega$
	and, for every $n\in \N\cup\{\infty\}$, define $S_n:=f_n(\partial B)$. 
	Since $S_\infty$ is compact and it does not contain $0$ by assumption, 
	we have that $\delta:=d(S_\infty,0)>0$. 
	Hence, there exists $N'$ such that for $n\ge N'$ we have 
	$\|f_n-f_\infty\|_{L^\infty(\Omega)}\le \delta/2$. 
	This means that for every $x\in\partial B$ and every $n\ge N'$, we have 
	\[
	|f_n(x)|\ge |f_\infty(x)|-|f_n(x)-f_\infty(x)|\ge \delta-\delta/2=\delta/2,
	\]
	from which we deduce that $d(S_n,0)\ge \delta/2$ for every $n\ge N'$. 
	Hence, for every $n\in \N\cup\{\infty\}$ such that also $n\ge N'$, we can define $g_n:\partial B\to \mathbb S^{m-1}$ as 
	\[
	g_n(x):=\frac{f_n(x)}{|f_n(x)|}.
	\]
	From the hypothesis we have that $g_n\to g_\infty$ uniformly on $\partial B$. 
	Hence, for some $N\ge N'$, we have that, for every $n\ge N$, 
	the map 
	\[
	\alpha_{n,t}(x) := \frac{tg_n(x)+(1-t)g_\infty(x)}{|tg_n(x)+(1-t)g_\infty(x)|}
	\]
	is a homotopy from $\alpha_{n,1}=g_n$ to $\alpha_{n,0}=g_\infty$
	and thus $\deg (g_n)=\deg (g_\infty)=i(f_\infty,0)\neq 0$.
	
	Now we claim that 
	\begin{equation}\label{eq608be5fe}
		B(0,\delta/2)\subset f_n(B) \qquad \forall n\geq N.
	\end{equation}
	Fix $n\ge N$ and suppose by contradiction that~\eqref{eq608be5fe} is not true. 
	Then there exists $p\in B(0,\delta/2)\setminus f_n(B)$. 
	For every $0\le\eta\le1$ let us define $\phi_\eta:\partial B\to \mathbb S^{m-1}$ as 
	\[
	\phi_{\eta}(x):=\frac{f_n(\eta x)-p}{|f_n(\eta x)-p|},
	\]
	which is well defined since $p\notin f_n(B)$.
	Clearly, 
	as $\eta$ varies from $0$ to $1$,
	the maps $\phi_\eta$ are an homotopy between the constant map $\phi_0\equiv p/|p|$ and $\phi_1$:
	therefore, $\deg\phi_1=0$.
	Moreover, for every $0\le\eta\le 1$, we can define $\psi_{\eta}:\partial B\to\mathbb S^{m-1}$ as 
	\[
	\psi_{\eta}(x) := \frac{f_n(x)-\eta p}{|f_n(x)-\eta p|}.
	\]
	Notice that $\psi_\eta$ is well-defined because $|f_n(x)|\ge\delta/2> \eta|p|$ for every $0\le \eta \le1$ and for every $x\in\partial B$. 
	Since the maps $\psi_\eta$ are an homotopy from 
	$\psi_0=g_n$ to $\psi_1=\phi_1$,
	we obtain $\deg g_n=\deg(\phi_1)=0$, 
	which is in contradiction with the fact that $\deg(g_n)\neq 0$ for every $n\ge N$. 
	Therefore,~\eqref{eq608be5fe} must be true, and it directly implies the assertion.
\end{proof}

In order to apply the previous result, the following lemma gives a criterium to check the hypothesis of Lemma~\ref{lem6082ea3b}.

\begin{lemma}\label{lem608be6cc}
	Let $\Omega\subset \R^m$ be an open set containing $0$,
	and let $f:\Omega\to\R^m$ a topological embedding with $f(0)=0$.
	Let $r>0$ such that $B(0,r)\Subset\Omega$ and define $\phi:\bb S^{m-1}\to\bb S^{m-1}$ as
	\[
	\phi(x) = \frac{f(rx)}{|f(rx)|} .
	\]
	Then $\deg(\phi)\in\{-1,1\}$.
\end{lemma}
\begin{proof}
	The function $\phi$ is the composition 
	\[
	\bb S^{m-1}
	\overset{rx}{\longrightarrow}
	\Omega\setminus\{0\}
	\overset{f}{\longrightarrow}
	f(\Omega)\setminus\{0\}
	\into
	\R^m\setminus\{0\}
	\overset{x/|x|}{\longrightarrow}
	\bb S^{m+1} .
	\]
	Notice that the maps $x\mapsto rx$ and $x\mapsto\frac{x}{|x|}$
	are retracts of $\R^m$ to $\bb S^{m-1}$, and thus they induce isomorphisms 
	between the homology groups.
	Since $f$ is an embedding and $f(0)=0$, also $f$ induces an isomorphism between the homology groups of $\Omega\setminus\{0\}$ and those of $f(\Omega)\setminus\{0\}$.
	Finally, since $f(\Omega)$ is open, every homology class of $\R^m$ has a representative inside $f(\Omega)$ and thus the immersion $f(\Omega)\into\R^m$ defines a surjective morphism of the corresponding homology groups.
	
	We conclude that the induced group morphism $\phi_*:H_n(\bb S^{m-1})\to H_n(\bb S^{m-1})$ is a surjective group morphism from $\Z$ to $\Z$ and thus $\deg(\phi)\in\{-1,1\}$.
\end{proof}

The following lemma should be compared with \cite[Proposition 3.5]{MR3739202}. It gives a quantitative open mapping theorem for $C^2$ functions, where the bounds depend explictly on the first and second derivatives of the function. A non-quantitative statement for $C^1$ functions is in \cite[Proposition 3.5]{MR3739202}. Even if its proof is rather standard, we record it here for the reader's convenience. We give the statement for an arbitrary Banach space, and later we will apply it with $\mathbb B=\mathbb R^k$.
\begin{lemma}\label{lem:HOMEO}
	Let $(\mathbb B,\|\cdot\|)$ be a Banach space. For every $\ell, L>0$ there exist $C_1:=1/(2\ell)$ and $C_2:=1/(2\ell L)$ such that the following holds.
	Let $f: B(x_0,r)\subset\mathbb B\to\mathbb B$ be a $C^2$ map such that 
	\begin{itemize}
		\item The linear map $Df_{x_0}$ is an isomorphism with $\|\left(Df_{x_0}\right)^{-1}\|\leq \ell$,
		\item We have $\|(D^2f)_x\|\leq L$ for every $x\in B(x_0,r)$.
	\end{itemize}
	Hence 
	\begin{equation}\label{eqn:SoughtClaim}
	B(f(x_0),C_1\rho)\subseteq f(B(x_0,\rho)), \qquad \text{for all $\rho\in (0,\min\{r,C_2\})$.}
	\end{equation}
		
\end{lemma}
\begin{proof}
	Let us denote for simplicity $A:=Df_{x_0}$. Define, for $x\in B(0,r)$,
	$$
	\widetilde f(x):=A^{-1}(f(x+x_0)-f(x_0))-x.
	$$
	Notice that $\widetilde f(0)=0$, $(D\widetilde f)_{0}=0$, and 
	\begin{equation}\label{eqn:ALEINova}
		(D\widetilde f)_x=A^{-1}((Df)_{x+x_0}-A).
	\end{equation}
	Notice that, given $\rho\leq r$, for every $x\in B(0,\rho)$ we have
	$$
	\|(Df)_{x+x_0}-A\|\leq \sup_{x\in B(0,\rho)}\|(D^2f)_{x+x_0}\|\|x\|\leq L\rho.
	$$
	Since $\|A^{-1}\|\leq \ell$, the previous inequality together with~\eqref{eqn:ALEINova} gives that for every $x\in B(0,\rho)$, with $\rho\leq r$, the following holds 
	$$
	\|(D\widetilde f)_x\|\leq \ell L\rho.
	$$
	The latter implies that for every $x,z\in B(0,\rho)$, with $\rho\leq r$, we have 
	$$
	\| \widetilde f(x)-\widetilde f(z)\|\leq \ell L\rho|x-z|.
	$$
	
	Let us now prove~\eqref{eqn:SoughtClaim}. Fix $\rho<\min\{1/(2\ell L),r\}$. Take an arbitrary $y\in \overline B(0,(1-\ell L \rho)\rho)$. Due to the previous inequality, the fact $\widetilde f(0)=0$, and the triangle inequality, the function
	$$
	x\mapsto -\widetilde f(x)+y
	$$
	maps $\overline B(0,\rho)$ into $\overline B(0,\rho)$. Moreover, the latter function is a contraction, then it has a fixed point. Namely, for every $y\in B(0,(1-\ell L\rho)\rho)$, there exists $x\in B(0,\rho)$ such that 
	$$
	f(x+x_0)=f(x_0)+Ay
	$$
	Now we claim that the set made by $Ay$'s, when $y$ runs in $ B(0,(1-\ell L\rho)\rho)$, contains the ball $B(0,(2\ell)^{-1}\rho)$ . Indeed, if $\eta\in B(0,(2\ell)^{-1}\rho)$ we claim that we can take $y=A^{-1}\eta$. Indeed, $A^{-1}\eta\in B(0,(2\ell)^{-1}\rho\ell)$, since $\|A^{-1}\|\leq \ell$, and $2^{-1}\rho<(1-\ell L\rho)\rho$ by how we chose $\rho$.
\end{proof}




\section{A shorter proof of the convergence result for smooth vector fields}\label{sec:Shorter}
In this section we offer a shorter proof of Theorem~\ref{thm60929a0e} in the case in which the vector fields are smooth. 
For the ease of notation, let us introduce the following terminology, useful for the discussion of this section. Let us fix $\mathbb E:=\mathbb R^k$, for $k\in\mathbb N$, with canonical basis $\{e_1,\dots,e_k\}$.

Let $M$ be a smooth manifold.  
Let 
$$
\f:M\times \mathbb \E\to TM,
$$
be a smooth morphism of bundles. Notice that in particular we have that $\f(p,\cdot):\mathbb E\to T_pM$ is a linear map for every $p\in M$, and moreover, for every $i=1,\dots,k$, we have that $\f(\cdot,e_i)$ is a smooth vector field.
Let 
$$
N:M\times \E\to [0,+\infty),
$$
be a continuous function such that $N(p,\cdot)$ is a norm for every $p\in M$.

Any couple $(\f,N)$ satisfying the two above conditions will be called \textit{CC-bundle structure}, and it induces an energy function $\Enel$, a length functional $\ell$ and a distance $d_{(\f,N)}$ as discussed in Definition~\ref{def:EnergyAndLength}, and Definition~\ref{def:CCdistance}.


\begin{definition}[Continuosuly varying CC-bundle structure]\label{def:CCbundleStructures}
	Let $\Lambda$ be a compact space, which will be called {\em set of parameters}. Let $M$ be a smooth manifold endowed with a Riemannian metric $\rho$. Endow $TM$ with the bundle metric induced by $\rho$.
	
	Let $\f:\Lambda\times M\times \E\to TM$ and $N:\Lambda\times M\times\E\to [0,+\infty)$ be such that for every $\lambda\in\Lambda$ we have that $(\f_\lambda,N_\lambda)$ is a CC-bundle structure, where $\f_\lambda:=\f(\lambda,\cdot,\cdot)$ and $N_\lambda:=N(\lambda,\cdot,\cdot)$. We say that the family $\{(\f_\lambda,N_\lambda)\}_{\lambda\in\Lambda}$ is a {\em continuously varying CC-bundle structure} if
	\begin{enumerate}
		\item  In coordinates around every point $p\in M$, all the partial derivatives in $q$ of $X_\lambda^i(q):=\f_\lambda(q,e_i)$ of order at most two are continuous in $(\lambda,q)$, for every $i=1,\dots,k$;
		\item $N\in C^0(\Lambda\times M\times\E)$ and $\f\in C^0(\Lambda\times M\times\E)$;
		\end{enumerate}
		From item (1) above we consequently have the following. For every compact $K_1\subseteq M$, and every compact $K_2\subseteq \Lambda\times\E$ there exists $L$ such that for every $(\lambda,v)\in K_2$ the vector field 
		\begin{equation}\label{LocLip}
		K_1\ni p \mapsto \f(\lambda,p,v)\in TM
		\end{equation}
		is $L$-Lipschitz with respect to the Riemannian distances.
\end{definition}

We now prove the following theorem, that is essentially a restatement of Theorem~\ref{thm60929a0e}(iv) for smooth vector fields.
\begin{theorem}\label{thm:VIDEO}
	Let $\Lambda$ be a compact space, and let $\{(\f_\lambda,N_\lambda)\}_{\lambda\in\Lambda}$ be a continuously varying CC-bundle structure on a smooth manifold $M$. Let $d_{\lambda}:=d_{(\f_\lambda,N_\lambda)}$ for $\lambda\in\Lambda$. Let $\lambda_0\in\Lambda$ be such that $\{\f(\lambda_0,M,e_i)\}_{i=1,\dots,k}$ is bracket-generating, and assume that the metric space $(M,d_{\lambda_0})$ is boundedly compact.
	
	Then $d_\lambda\to d_{\lambda_0}$ uniformly on compact sets of $M$ as $\lambda\to\lambda_0$. 
\end{theorem}

We give a direct proof of the previous theorem using the following crucial lemma, that contains some ideas already used in Lemma~\ref{lem604be4c3}, and Proposition~\ref{prop604be4c9}. We stress that the proof given below is more direct because we are essentially able to bypass the use of Theorem~\ref{thm6092f873}.
\begin{lemma}\label{LEMMACRUCIALE}
	Let $\Lambda$ be a compact space, and let $\{(\f_\lambda,N_\lambda)\}_{\lambda\in\Lambda}$ be a continuously varying CC-bundle structure on a smooth manifold $M$. Let $d_{\lambda}:=d_{(\f_\lambda,N_\lambda)}$ for $\lambda\in\Lambda$. Assume $\{\f(\lambda_0,M,e_i)\}_{i=1,\dots,k}$ is bracket-generating for some $\lambda_0\in\Lambda$.
	
	For every compact set $K\subseteq M$ and every Riemannian metric $\rho$ on $M$ there exist a compact neighborhood $I_{\lambda_0}\subseteq\Lambda$ of $\lambda_0$ and a continuous nondecreasing function $\beta:(0,+\infty)\to(0,+\infty)$, with $\lim_{s\to 0^+}\beta(s)=0$, such that
	\begin{equation}\label{eqn:VERIF}
	d_\lambda(p,q)\leq \beta(\rho(p,q)),\qquad \text{for all $p,q\in K$, and $\lambda\in I_{\lambda_0}$}.
	\end{equation}
\end{lemma}
\begin{proof}
	Fix such $\rho$ and $K$. Let us denote, for $\lambda\in\Lambda$ and $p\in M$,
	$$
	X_i^{\lambda}(p):=\f(\lambda,p,e_i),
	$$
	where $\{e_1,\dots,e_k\}$ is the standard basis of $\E=\mathbb R^k$. 
	
	We know that $\{X_i^{\lambda_0}\}_{i=1}^k$ is a bracket-generating set of vector fields on $M$. Hence, there exists a compact neighborhood $I_1$ of $\lambda_0$ such that $\{X_i^\lambda\}_{i=1}^k$ is a bracket-generating set of vector fields on $B_\rho(K,1)$.
	
	\textit{Claim 1.} For every $x\in K$, for every $\lambda'\in I_1$, and for every $\varepsilon>0$ there exist $\delta>0$ and a compact neighborhood $I_{\lambda'}\subseteq I_1$ of $\lambda'$ such that 
	\begin{equation}\label{DaFareBene}
	B_\rho(x,\delta)\subseteq B_{d_\lambda}(x,\varepsilon), \quad \text{for all $\lambda\in I_{\lambda'}$.}
	\end{equation}
	
	Let us prove the claim. Let us fix $x\in K$ and $\lambda'\in I_1$ from now on.  Hence, since $\{X_i^{\lambda'}\}_{i=1}^k$ are bracket-generating on $B_{\rho}(K,1)$, the following holds, due to \cite[Lemma 3.33]{MR3971262}. For every $0<\eta<1$, there exist $\{i_1,\dots,i_n\}\subseteq\{1,\dots,k\}$ and $\hat t:=(\hat t_1,\dots,\hat t_n)$ such that $n:=\dim M$, $|\hat t|<\eta$, and
	$$
	\Phi:(t_1,\dots,t_n)\mapsto \Phi_{X_{i_n}^{\lambda'}}^{t_n}\circ\dots\circ\Phi_{X_{i_1}^{\lambda'}}^{t_1}(x),
	$$
	has a regular point at $\hat t$. Hence the map
	$$
	(t_1,\dots,t_n)\mapsto \Phi^{-\hat t_1}_{X_{i_1}^{\lambda'}}\circ\dots\circ\Phi^{-\hat t_n}_{X_{i_n}^{\lambda'}}\circ \Phi_{X_{i_n}^{\lambda'}}^{t_n}\circ\dots\circ\Phi_{X_{i_1}^{\lambda'}}^{t_1}(x),
	$$ 
	has a regular point at $\hat t$ and sends $\hat t$ to $x$. Now, let $\omega:=\omega_1\dots\omega_D$ be a word of $D$ letters such that it contains as a subword every string of $2n$ elements chosen among $\{1,\dots,k\}$. Notice now that the map 
		\begin{equation}\label{eqn:ConcatenationOfCurves}
			\Psi:(\lambda,t_1,\dots,t_D)\mapsto \Phi_{X_{i_{\omega_D}}^{\lambda}}^{t_D}\circ\dots\circ\Phi_{X_{i_{\omega_1}}^{\lambda}}^{t_1}(x),
		\end{equation}
		together with the maps $D_T\Psi,D^2_T\Psi$ - where $D_T$ denotes the differential with respect to the components in $\mathbb R^D$ - are continuous and well defined on $ I_2\times \overline B_{|\cdot|_1}(0,\xi)$, where $\overline B_{|\cdot|_1}(0,\xi)$ is the ball in $\mathbb R^D$ with respect to the $\ell_1$-norm $|\cdot|_1$ centered at $0$ and with a sufficiently small radius $\xi$, and $I_2\subseteq I_1$ is a sufficiently small compact neighborhood of $\lambda'$. The last assertion is a consequence of an iterated application of Gronwall's Lemma (see Lemma~\ref{lem1053}), and the fact that we have the continuity property in Definition~\ref{def:CCbundleStructures}(1).
		
		 Define the compact set $\widetilde K:=\Psi(I_2\times \overline B(0,\xi))$ in $M$. Hence, by continuity of the norm $N$, there exists $L>0$ such that 
	\begin{equation}\label{eqn:EstimateOnNorm}
		N(\lambda,p,v)\leq L|v|_1, \quad \text{for all $\lambda\in I_2$, $p\in \widetilde K$, and $v\in\mathbb E$}.
	\end{equation}
	
	Let us now conclude the proof of the claim. Fix $\varepsilon>0$. Let $$\nu:=\min\{\xi/2,\varepsilon/(2DL)\},$$ where $\xi$ is defined above. Notice that, by what we noticed above, we have that there exists $\widetilde t\in \mathbb R^{D}$ with $|\widetilde t|_1<\nu$ such that $\Psi(\lambda',\cdot)$ has a regular point at $\widetilde t$ and $\Psi(\lambda',\widetilde t)=x$. In addition to this, $\Psi(\lambda,t)\to \Psi(\lambda',t)$ as $\lambda\to\lambda'$, uniformly when $t\in \overline B_{|\cdot|_1}(0,\xi)\subseteq \mathbb R^N$, and the same convergence holds with $D_T\Psi,D^2_T\Psi$. The last assertion is a consequence of the fact that the maps $\Psi,D_T\Psi,D^2_T\Psi$ are continuous, and thus uniformly continuous on compact sets.
	
	Since $\widetilde t$ is a regular point for $\Psi(\lambda',\cdot)$, we can find an $n$-dimensional subspace $\Pi$ of $\mathbb R^D$ such that $\Psi(\lambda',\cdot)$ restricted to $B(\widetilde t,\nu)\cap (\widetilde t+\Pi)$ is a local diffeomorphism around $\widetilde t$. Moreover, there exists a neighborhood of $\lambda'$, which we call $I_{\lambda'}$, such that the maps $\Psi(\lambda,\cdot)$, restricted to $B(\widetilde t,\nu)\cap(\widetilde t+\Pi)$, satisfy the bounds in the hypotheses of Lemma~\ref{lem:HOMEO} uniformly on $\lambda\in I_{\lambda'}$, and where $\widetilde t$ here is the $x_0$ there. Notice that the bounds of Lemma~\ref{lem:HOMEO} hold uniformly on $\lambda\in I_{\lambda'}$ for some neighborhood $I_{\lambda'}$ of $\lambda'$, due to the continuity of $D_T\Psi,D^2_T\Psi$ discussed above.
	
 	Hence, by applying Lemma~\ref{lem:HOMEO}, we have that there exists $\delta>0$ and $I_{\lambda'}$ a neighborhood of $\lambda'$ such that 
	$$
	B_{\rho}(x,\delta)\subseteq \Psi(\lambda,\widetilde t+B_{|\cdot|_1}(0,\nu)), \quad \text{for all $\lambda\in I_{\lambda'}$}.
	$$
	Since $\widetilde t+B_{|\cdot|_1}(0,\nu)\subseteq \overline B_{|\cdot|_1}(0,\xi)$, we have that all the concatenation of the curves associated to $\Psi(\lambda,\widetilde t+B(0,\nu))$ is in $K$.
	Since the estimate~\eqref{eqn:EstimateOnNorm} holds, and since  we also have $\widetilde t+B_{|\cdot|_1}(0,\nu)\subseteq \overline B_{|\cdot|_1}(0,\varepsilon/(DL))$, we get that the  concatenation of the curves associated to $\Psi(\lambda,t)$ (whose controls can be written explicitely as in~\eqref{EsplicitaControlli}), for every $t\in \widetilde t+\overline B_{|\cdot|_1}(0,\nu)$ has length $\leq \varepsilon$ for every $\lambda\in I_{\lambda'}$. Hence the sought claim~\eqref{DaFareBene} holds.
	
	From Claim 1 and a  routine compatness argument, as already done at the end of Lemma~\ref{lem604be4c3}, we have that for every $\varepsilon>0$ there exists $\delta>0$ such that for every $x\in K$ and every $\lambda \in I_1$ we have
	\begin{equation}\label{DaFareBene2}
	B_\rho(x,\delta)\subseteq B_{d_\lambda}(x,\varepsilon).
	\end{equation}
	
	
	From~\eqref{DaFareBene2} the proof of the lemma follows with the following argument. 
	Let $\tilde K$ be a path-connected compact set containing $K$.
	For instance, $\tilde K$ can be chosen to be a closed $\rho$-balls of sufficiently large radius.
	
	The inequality in~\eqref{eqn:VERIF} is trivially satisfied if we define
	$\beta$ as
	\begin{align*}
		\beta(s) &= \sup\{d_\lambda(p,q) : p,q\in \tilde K,\lambda \in I_1,\ \rho(p,q) \le s \},
	\end{align*}
	for every $s\in (0,+\infty)$.
	
	From~\eqref{DaFareBene2} we get that 
	for every $\varepsilon>0$ there is $\delta>0$ such that $\beta(\delta)<\varepsilon$.
	In particular, this implies $\lim_{s\to 0^+}\beta(s) =  0$.
	
	We also claim that $\sup_{s>0}\beta(s)<\infty$.
	Indeed, fix $\delta>0$ such that $\beta(\delta)\le 1$, and let $B_1,\dots, B_N$ be a collection of $\rho$-balls of radius less than $\delta/2$ that covers $\tilde K$.
	Since $\tilde K$ is path-connected, given $p,q\in \tilde K$, there is a sequence $p=p_0,p_1,\dots,p_m=q$ with $m\le 2N$ such that for every $i\in\{0,\dots,m-1\}$ there is $j\in\{1,\dots,N\}$ such that $p_i,p_{i+1}\in B_j$.
	Therefore, for every $\lambda\in I_1$,
	\[
	d_\lambda(p,q) \le \sum_{i=0}^{m-1} d_{\lambda}(p_i,p_{i+1})
	\le 2N\beta(\delta)
	\le 2N .
	\]
	Therefore, $\sup_{s>0}\beta(s)\le2n$.
\end{proof}

We now give the proof of Theorem~\ref{thm:VIDEO}. The strategy is different with respect to the proof of Theorem~\ref{thm60929a0e}(iv). There, we first proved the local uniform convergence, relying on the relaxation result in Proposition~\ref{prop604f912b}, and then we upgrade it to a uniform convergence on compact sets thanks to Lemma~\ref{Lemma2}. Here, instead, we directly obtain the uniform convergence on compact sets by making a careful use of Gronwall's Lemma~\ref{lem1053}.
\begin{proof}[Proof of Theorem {\ref{thm:VIDEO}}]
	We embed $M$ smoothly isometrically into some $\mathbb R^N$, on which we denote with $|\cdot|$ the standard norm. Let us fix a compact set $K$ and a Riemannian metric $\rho$ on $M$. Notice that on every compact set of $M$, $\rho$ and $|\cdot|$ are biLipschitz equivalent. Let us fix $0<\varepsilon<1/2$. 
	
	By continuity, there exists a constant $C>0$ such that $d_{\lambda_0}(p,q)\leq C$ for every $p,q\in K$. Let $K':=\overline B_{\lambda_0}(K,C+1)$ the closed tubular neighborhood of $K$ of radius $C+1$. Since $(M,d_{\lambda_0})$ is boundedly compact, we deduce that $K'$ is compact. 
	
	Let $\beta$ be the function, and $I_{\lambda_0}$ be the compact neighborhood of $\lambda_0$, associated to $K'$ given by Lemma~\ref{LEMMACRUCIALE}. We have that, for some $\vartheta>0$, $\rho(p,q) \leq \vartheta |p-q|$ for every $p,q\in K'$. Thus, up to renaming $\beta$, for every $p,q\in K'$, and for every $\lambda\in I_{\lambda_0}$, 
	$$
	d_{\lambda}(p,q)\leq \beta(|p-q|)\leq \beta(\diam_{|\cdot|} K').
	$$
	Since $N(\lambda,p,\cdot)$ is a norm for every $\lambda\in I_{\lambda_0}$ and every $p\in M$, and since $N$ is continuous, we get that there exists a compact set $K''\subseteq \E$  such that 
	\begin{equation}\label{eqn:ThenInK'}
		\text{if $N(\lambda,x,v)\leq \beta(\diam_{|\cdot|} K')+1$ for some $\lambda\in I_{\lambda_0}$ and $x\in K'$, then $v\in K''$}.
	\end{equation}
	Moreover, by definition of continuously varying CC-structures,~\eqref{LocLip}, we have that there exists $L>0$ such that for every $\lambda\in I_{\lambda_0}$ and $v\in K''$ the map 
	$$
	K'\ni p\mapsto \f(\lambda,p,v),
	$$
	is $L$-lipschitz.
	
	Because of continuity of the functions $N$ and $\f$
	we get that there exist $0<\delta_1<\varepsilon$ and a compact neighborhood $I'_{\lambda_0} \subseteq I_{\lambda_0}$ of $\lambda_0$ with 
		\begin{equation}\label{eqn:ContinuityN}
		|N(\lambda_0,x,v)-N(\lambda,y,v)|<\varepsilon, \quad \text{for all $\lambda\in I'_{\lambda_0}$, $x\in K'$, $v\in K''$, $y\in \overline B_{|\cdot|}(x,\delta_1)$},
	\end{equation}
	and
		\begin{equation}
		|\f(\lambda_0,x,v)-\f(\lambda,x,v)|<a, \quad \text{for all $\lambda\in I'_{\lambda_0}$, $x\in K'$, $v\in K''$,}
	\end{equation}
	where $a$ is chosen such that $a\frac{e^L-1}{L}<\delta_1$. We now prove the following claim.
	
	\textit{Claim 1.} For every $\lambda \in I'_{\lambda_0}$ and every $p,q\in K$, we have 
	$$
	d_{\lambda_0}(p,q)\leq d_{\lambda}(p,q)+2\varepsilon+\beta(\varepsilon).
	$$
	
	Fix $p,q,\lambda$ as in the claim. Up to reparametrization, we can take a curve $\gamma_\lambda$ connecting $p$ and $q$ such that $\gamma_\lambda'=\f(\lambda,\gamma_\lambda,u_\lambda)$ and 
	\begin{equation}\label{eqn:PerK'}
		N(\lambda,\gamma_\lambda(t),u_\lambda(t))\leq d_\lambda(p,q)+\varepsilon, \quad \text{for a.e. $t\in [0,1]$}.
	\end{equation}
	Let $B:=\overline B_{\lambda_0}(p,d_{\lambda_0}(p,q))$. Notice that $B\subseteq K'$. Define 
	$$
	\overline t:=\max\{t\in [0,1]:\gamma_\lambda(s)\in B\,\,\forall s\in[0,t]\}.
	$$
	Denote $q'_\lambda:=\gamma_\lambda(\overline t)$ and notice that $d_{\lambda_0}(p,q'_\lambda)=d_{\lambda_0}(p,q)$. Moreover notice that $(\gamma_\lambda)|_{[p,q_\lambda']}\subseteq K'$. Take now $\gamma_{\lambda,0}$ such that $\gamma_{\lambda,0}'=\f(\lambda_0,\gamma_{\lambda,0},u_\lambda)$ and $\gamma_{\lambda,0}(0)=p$. Call $\overline q_\lambda:=\gamma_{\lambda,0}(\overline t)$. Notice that as a consequence of Gronwall's Lemma, see Lemma~\ref{lem1053}, up to taking a slightly smaller neighborhood $I'_{\lambda_0}$, the curve $\gamma_{\lambda,0}$ is defined up to time $\overline t$. We will use the same argument below.
	
	We now want to estimate $|\overline q_\lambda-q'_\lambda|$. From~\eqref{eqn:PerK'},~\eqref{eqn:ThenInK'}, and the fact that $\gamma_\lambda([0,\overline t])\subseteq K'$ we get that $u_\lambda(t)\in K''$ for a.e. $t\in [0,\overline t]$. Hence we can estimate, for every $x,y\in K'$ and a.e. $t\in [0,\overline t]$,
	\begin{equation}
		\begin{split}
			|\f(\lambda,x,u_\lambda(t))-\f(\lambda_0,y,u_\lambda(t))|&\leq |\f(\lambda,x,u_\lambda(t))-\f(\lambda_0,x,u_\lambda(t))| \\
			&+|\f(\lambda_0,x,u_\lambda(t))-\f(\lambda_0,y,u_\lambda(t))| \\
			&\leq a+L|x-y|.
		\end{split}
	\end{equation}
	Hence Gronwall Lemma in~\ref{lem1053} applied on $K'$ directly implies that
	\begin{equation}\label{eqn:ConsequenceGronwall}
		|\gamma_{\lambda}(t)-\gamma_{\lambda,0}(t)|\leq a\frac{e^{Lt}-1}{L}<\delta_1<\varepsilon, \quad \text{for a.e. $t\in [0,\overline t]$},
	\end{equation}
	and moreover that $(\gamma_{\lambda,0})|_{[0,\overline t]}\subseteq K'$.
	Now let us conclude the estimate of Claim 1. We have 
	\begin{equation}
		\begin{split}
			d_{\lambda_0}(p,q)&=d_{\lambda_0}(p,q_\lambda')\leq d_{\lambda_0}(p,\overline q_\lambda)+d_{\lambda_0}(\overline q_\lambda,q'_\lambda) \\
			&\leq \int_0^{\overline t}N(\lambda_0,\gamma_{\lambda,0}(s),u_{\lambda}(s))\dd s + \beta(|\overline q_\lambda-q'_\lambda|) \\
			&\leq \int_0^{\overline t} N(\lambda,\gamma_\lambda(s),u_\lambda(s))\dd s + \varepsilon + \beta(\varepsilon) \\
			&\leq \int_0^{1} N(\lambda,\gamma_\lambda(s),u_\lambda(s))\dd s + \varepsilon + \beta(\varepsilon) \\
			&\leq d_{\lambda}(p,q) + 2\varepsilon + \beta(\varepsilon),
		\end{split}
	\end{equation}
	where we are using~\eqref{eqn:ConsequenceGronwall},~\eqref{eqn:ContinuityN}, and~\eqref{eqn:PerK'}. We thus obtained the sought claim.
	
	\textit{Claim 2.} For every $\lambda \in I'_{\lambda_0}$ and every $p,q\in K$, we have 
	$$
	d_{\lambda}(p,q)\leq d_{\lambda_0}(p,q)+\varepsilon+\beta(\varepsilon).
	$$
	
	Fix $p,q,\lambda$ as in the claim. Up to reparametrization, for every $0<\varepsilon<1$ we can take a curve $\gamma$ connecting $p$ and $q$ such that $\gamma'=\f(\lambda_0,\gamma,u)$ and 
	\begin{equation}\label{eqn:PerK'2}
		N(\lambda_0,\gamma(t),u(t))\leq d_{\lambda_0}(p,q)+\varepsilon, \quad \text{for a.e. $t\in [0,1]$}.
	\end{equation}
	Notice that $\gamma\subseteq K'$. Take now $\gamma_{\lambda}$ such that $\gamma_{\lambda}'=f(\lambda,\gamma_{\lambda},u)$ and $\gamma_{\lambda}(0)=p$. Again, as a consequence of Gronwall Lemma, the curve $\gamma_{\lambda}$ is defined up to time $t=1$. Call $\overline q_\lambda:=\gamma_{\lambda}(1)$.
	
	We now want to estimate $|\overline q_\lambda-q|$. Arguing verbatim as before we obtain
	\begin{equation}\label{eqn:ConsequenceGronwall2}
		|\gamma_{\lambda}(t)-\gamma(t)|\leq a\frac{e^{Lt}-1}{L}<\delta_1<\varepsilon, \quad \text{for a.e. $t\in [0,1]$},
	\end{equation}
	and moreover $\gamma_\lambda\subseteq K'$.
	Now let us conclude the estimate of Claim 2. We have 
	\begin{equation}
		\begin{split}
			d_{\lambda}(p,q)&\leq d_{\lambda}(p,\overline q_\lambda)+d_{\lambda}(\overline q_\lambda,q) \\
			&\leq \int_0^{1}N(\lambda,\gamma_{\lambda}(s),u(s))\dd s + \beta(|\overline q_\lambda-q|) \\
			&\leq \int_0^{1} N(\lambda_0,\gamma(s),u(s))\dd s + \varepsilon + \beta(\varepsilon) \\
			&\leq d_{\lambda_0}(p,q) + 2\varepsilon + \beta(\varepsilon),
		\end{split}
	\end{equation}
	where we are using~\eqref{eqn:ConsequenceGronwall2},~\eqref{eqn:ContinuityN}, and~\eqref{eqn:PerK'2}. Thus we obtained the sought claim.
	
	From Claim 1 and Claim 2 jointly with the fact that $\beta(\varepsilon)\to 0$ as $\varepsilon\to 0$ we get the proof of the Theorem.
\end{proof}

\printbibliography
\end{document}